\documentclass[10pt,reqno]{amsart}

\usepackage{amsthm, mathrsfs, amsmath, amstext, esint, amsxtra, amsfonts, slashed, dsfont, amssymb}
\usepackage{enumitem}
\usepackage[dvipsnames]{xcolor}
\usepackage[colorlinks, linkcolor=red, citecolor=blue, urlcolor=blue, pagebackref, hypertexnames=false]{hyperref}
\usepackage{extarrows}
\usepackage{graphicx}
\usepackage{mathtools}

\usepackage{lmodern}
\usepackage[T1]{fontenc}

\makeatletter
\def\@tocline#1#2#3#4#5#6#7{\relax
  \ifnum #1>\c@tocdepth 
  \else
    \par \addpenalty\@secpenalty\addvspace{#2}%
    \begingroup \hyphenpenalty\@M
    \@ifempty{#4}{%
      \@tempdima\csname r@tocindent\number#1\endcsname\relax
    }{%
      \@tempdima#4\relax
    }%
    \parindent\z@ \leftskip#3\relax \advance\leftskip\@tempdima\relax
    \rightskip\@pnumwidth plus4em \parfillskip-\@pnumwidth
    #5\leavevmode\hskip-\@tempdima
      \ifcase #1
       \or\or \hskip 1em \or \hskip 2em \else \hskip 3em \fi%
      #6\nobreak\relax
      \dotfill
      \hbox to\@pnumwidth{\@tocpagenum{#7}}
    \par
    \nobreak
    \endgroup
  \fi}
\makeatother

\usepackage[belowskip=2pt,aboveskip=0pt]{caption}
\usepackage{booktabs}
\usepackage{geometry}
\newcommand\norm[1]{\left\lVert#1\right\rVert}

\newcommand\wto{\rightharpoonup}
\newcommand{\ada}{a^\dagger}

\newcommand{\im}{\mathrm{i}}

\newcommand{\bx}{\mathbf{x}}
\newcommand{\by}{\mathbf{y}}

\newcommand{\N}{\mathbb N} 
\newcommand{\R}{\mathbb R} 
\newcommand{\C}{\mathbb C} 
\newcommand{\E}{\mathbb E} 

\newcommand{\cF}{\mathcal F}

\newcommand{\Hc}{\mathcal H}

\newcommand{\cH}{\mathcal H}

\newcommand{\cU}{\mathcal U}

\newcommand{\cM}{\mathcal M}
\newcommand{\cN}{\mathcal N}

\newcommand{\gS}{\mathfrak S}

\newcommand{\gF}{\mathfrak F}

\newcommand{\Nc}{\mathcal N}

\newcommand{\gh}{\mathfrak h}
\newcommand{\eps}{\epsilon}

\DeclareMathOperator*{\NL}{NL}
\DeclareMathOperator{\Tr}{\mathrm{Tr}}
\DeclareMathOperator*{\cl}{cl}
\DeclareMathOperator*{\spann}{span}

\newcommand{\Id}{\mathds{1}}
\newcommand{\sym}{\mathrm{sym}}

\newcommand{\cEMF}{\mathcal{E} ^{\rm MF}}

\numberwithin{equation}{section}
\newtheorem{theorem}{Theorem}[section]
\newtheorem{lemma}[theorem]{Lemma}
\newtheorem{proposition}[theorem]{Proposition}

\newtheorem{definition}[theorem]{Definition}
\newtheorem{assumption}[theorem]{Assumption}

\makeatletter
\def\subsubsection{\@startsection{subsubsection}{3}%
	\z@{.5\linespacing\@plus.7\linespacing}{-.5em}%
	{\normalfont\bfseries}}
\makeatother

\newcommand{\chapter}{\part}

\usepackage[depth=5]{bookmark}

\title[]{From bosonic canonical ensembles to non-linear Gibbs measures} 

\author[V. D. Dinh and N. Rougerie]{Van Duong Dinh and Nicolas Rougerie}
\address[V. D. Dinh]{Ecole Normale Sup\'erieure de Lyon \& CNRS, UMPA (UMR 5669), Lyon, France}
\email{contact@duongdinh.com}
\address[N. Rougerie]{Ecole Normale Sup\'erieure de Lyon \& CNRS, UMPA (UMR 5669), Lyon, France}
\email{nicolas.rougerie@ens-lyon.fr}


\begin{document}
	
	\date{March, 2026}
	
	\begin{abstract}
		We study the mean-field limit of the 1D bosonic canonical ensemble in a superharmonic trap. This is the regime with temperature proportional to particle number, both diverging to infinity, and correspondingly scaled interactions. We prove that the limit model is a classical field theory based on a non-linear Schr\"odinger-Gibbs measure conditioned on the $L^2$ mass, thereby obtaining a canonical analogue of previous results for the grand-canonical ensemble. We take advantage of this set-up with fixed mass to include focusing/attractive interactions/non-linearities in our study. 
	\end{abstract}
	
	\maketitle
	
	\tableofcontents
	
	\newpage
	
	\section{Introduction}
	
	The study of mean-field (MF) limits of large bosonic quantum systems has been the source of a vast body of mathematical physics literature in the past decades. In the spirit of ``molecular chaos'', many-body wave-functions $\Psi_N (\bx_1,\ldots,\bx_N) \in L^2_{\sym} (\R^{dN})$ are well-approximated, for large $N$, by pure tensor powers $u(\bx_1) \ldots u(\bx_N), u \in L^2 (\R^d)$. The limit models, non-linear Schr\"odinger (NLS) type classical field theories have been derived at several levels: 
	
	\medskip
	
	\noindent \textbf{Static, ground-state configurations.} The bosonic lowest eigenfunctions of typical many-body Hamiltonians, say of the form 
	\begin{equation}\label{eq:intro HN}
	H_N= \sum_{j=1}^N \left( -\Delta_{\bx_j} + V(\bx_j)\right) + \frac{g}{N-1}\sum_{1\leq j<k\leq N} w (\bx_j-\bx_k)
	\end{equation}
	acting on $L^2_{\sym} (\R^{dN})$ are approximated~\cite{LieSeiSolYng-05,Rougerie-EMS,Rougerie-LMU} by minimizing the corresponding NLS energy functionals
	$$
	\cEMF [u] = \frac{1}{N} \left\langle u^{\otimes N}, H_N u^{\otimes N} \right\rangle = \int_{\R^d} |\nabla u |^2 + V |u|^2 + \frac{g}{2} \left(w*|u|^2\right)|u|^2.
	$$
	over normalized $u\in L^2 (\R^d)$.
	
	\medskip
	
	\noindent \textbf{Excitation spectra.} The low-lying excited eigenvalues above the ground state energy are given by Bogoliubov theory, i.e. by second-quantizing the Hessian of the NLS functional at the minimizer~\cite{Seiringer-11,GreSei-13,DerNap-13,LewNamSerSol-13,NamSei-14,BocBreCenSch-17,BocBreCenSch-18,BosPetSei-20,NamNap-20,Pizzo-15a,Pizzo-15b,Pizzo-15c}.

	\medskip
	
	\noindent \textbf{Dynamics, Schr\"odinger/Heisenberg evolution.} An initially approximately factorized datum $\Psi_N^0 \simeq \left(u_0\right)^{\otimes N}$ evolves along the many-body Schr\"odinger flow $\im \partial_t \Psi_N = H_N \Psi_N$ as $\Psi_N (t) \simeq \left(u (t)\right)^{\otimes N}$ where $u(t)$ solves~\cite{BenPorSch-15,Golse-13,Schlein-08,Spohn-12} a dynamical NLS equation
	\begin{equation}\label{eq:intro NLS} 
	\im \partial_t u = \left(- \Delta + V + g w * |u|^2 \right)u.
	\end{equation}

	\bigskip
	
	The investigation of MF limits (in the sense to be detailed below) of positive temperature equilibria has been initiated more recently~\cite{FroKnoSchSoh-16,FroKnoSchSoh-17,FroKnoSchSoh-20,FroKnoSchSoh-20b,FroKnoSchSoh-22,LewNamRou-15,LewNamRou-17,LewNamRou-18b,LewNamRou-18_2D,LewNamRou-20,Sohinger-22,RouSoh-22,RouSoh-23}. In this setting, a certain combination of factorized many-body wave-functions better approximates the equilibria. We refer to~\cite{DeuSeiYng-18,DeuSei-19,DeuSei-21} and references therein for different scaling limits of bosonic positive temperature ensembles.
	
	A typical result in the mean-field limit essentially states  that a grand-canonical Gibbs state 
	\begin{equation}\label{eq:intro GC}
	\Gamma_\nu = \frac{1}{Z_\nu} \exp\left( -\frac{1}{T} \bigoplus_{N= 0} ^\infty \left(  H_N - \nu N \right)\right)	
	\end{equation}
    normalized as a positive trace-class operator on the bosonic Fock space 
    $$ 
    \gF \left(L^2 (\R^d)\right) = \bigoplus_{N=0} ^\infty L ^2 _{\sym} (\R^{dN})
    $$
	behaves in an appropriate sense and regime $N,T\to \infty$ as 
	$$ \int_{u\in L^2 (\R^d)} \left|\xi\left(\sqrt{T} u\right)\right\rangle \left\langle \xi\left(\sqrt{T} u\right)\right|   d\mu (u)$$
	where $\xi (\psi)$ is a bosonic coherent state 
	\begin{align} \label{eq:cohe stat} 
	\xi(\psi) = \exp\left(-\frac{1}2 \norm{\psi}_{L^2} ^2 \right) \bigoplus_{N=0} ^\infty \frac{\psi^{\otimes N}}{\sqrt{N!}}
	\end{align}
	and $\mu$ is a nonlinear Schr\"odinger-Gibbs measure, formally given in terms of the (non-existent !) Lebesgue measure $du$ on $L^2 (\R^d)$
	\begin{equation}\label{eq:intro Gibbs} 
	d\mu (u) = \frac{1}{z} \exp\left( - \frac{1}{t} \left(\cEMF[u] - c \int_{\R^d} |u|^2\right)\right) du 
	\end{equation}
where $z,t,c>0$ are effective partition function, temperature and chemical potential respectively. Such results give a new perspective on the random data Cauchy theory for NLS~\cite{DenNahYue-19,DenNahYue-22,DenNahYue-21,LebRosSpe-88,Zhidkov-91, Bourgain-96,Tzvetkov-06, Tzvetkov-08,NahOhReySta-11,ThoTzv-10,BurThoTzv-10,Deng-12,RobSeoTolWan-22,CacSuz-14,Thomann-09,PoiRobTho-14,DinRou-23,DinRouTolWan-23} which considers, inter alias, the well-posedness of~\eqref{eq:intro NLS} on the support of~\eqref{eq:intro Gibbs}, and the invariance of the latter along the flow.  

The present paper aims at extending the previous results to the case of the canonical ensemble with fixed particle number $N$
$$ \Gamma_N = \frac{1}{Z_N} \exp\left( -\frac{1}{T}  H_N \right),$$
seeking a result of flavor (again, in an appropriate sense, see below)
$$ \Gamma_N \simeq \int \left|u ^{\otimes N}\right\rangle \left\langle u ^{\otimes N} \right|   d\mu_m (u)$$
where $\mu_m$ is a surface measure obtained by restricting~\eqref{eq:intro Gibbs} on a $L^2$-sphere 
\begin{equation}\label{eq:intro Gibbs mass} 
	d\mu_m (u) = \frac{1}{z_m} \exp\left( - \frac{1}{t} \cEMF[u] \right) \Id_{\left\{\int_{\R^d} |u|^2 = m\right\}} \, du. 
	\end{equation}
Here again $z_m$ is an effective partition function, normalizing the above. We consider for now the pendant of the frameworks of~\cite{LewNamRou-15,FroKnoSchSoh-16,FroKnoSchSoh-17} where the $L^2$-mass is well-defined on the support of the Gibbs measure~\eqref{eq:intro Gibbs} so that the definition of~\eqref{eq:intro Gibbs mass}	is natural~\cite{OhQua-13,Brereton-19,DinRou-23} and the appropriate limit is $N=m T$ with $m$ fixed and $T \to +\infty$. This limits us essentially to 1D, but allows to study attractive/focusing interactions ($w\leq 0$ in~\eqref{eq:intro HN}) that do not make sense grand-canonically. Indeed, the interactions are quartic in the particle number, and the kinetic energy quadratic. If the former are negative, they always dominate the latter for unbounded particle numbers. The attractive grand-canonical ensemble with a cut-off has been considered recently in~\cite{RouSoh-22,RouSoh-23} in a setting similar to ours but with rather different tools. We will build on the general variational approach of~\cite{LewNamRou-15,LewNamRou-17,LewNamRou-20} based on coherent states/quantum de Finetti theorems and Berezin-Lieb-type inequalities. The main new aspects we tackle are linked to 

\medskip

\noindent $\bullet$ the definition of the Gibbs measure~\eqref{eq:intro Gibbs mass} conditioned on the $L^2$ mass and how it relates to many-body quantum mechanics. 

\medskip

\noindent $\bullet$ the lack of Wick-type theorems for expectations in canonical ensembles (quantum or classical), which undermines many explicit calculations used in~\cite{LewNamRou-15,LewNamRou-17,LewNamRou-20} to tackle the non-interacting problem.

\medskip

\noindent $\bullet$ the control of focusing interactions, made possible by considering the canonical ensemble, but which requires a more careful control of the limit. 

\medskip

We recall that the same type of problem, with a one-particle space restricted to finite dimensions has been considered in~\cite{Gottlieb-05,Knowles-thesis,Rougerie-LMU}. Our main contribution is thus to tackle the above aspects in the infinite-dimensional setting.

\bigskip

\noindent

\noindent\textbf{Acknowledgments:} We thank Florent Foug\`eres for sharing with us his master memoir~\cite{Fougeres-memoire} on the free bosonic canonical ensemble and Nikolay Tzvetkov for his remarks and unpublished notes on mass-conditioned Gibbs measures.

	\section{Main results}

	The non-interacting part of our model will be given by the 1-dimensional Schr\"odinger operator/anharmonic oscillator
	\begin{equation}\label{eq:hamil 1body}
	h=-\partial^2_x + |x|^s, \quad s >6.
	\end{equation}
	When $s=\infty$, we restrict our consideration to $x \in [0,1]$ with Dirichlet boundary condition. It is well-known that $h$ is a non-negative self-adjoint operator with compact resolvent on $\mathfrak {h}= L^2(X)$, where $X=\mathbb R$ if $s<\infty$ and $X=[0,1]$ if $s=\infty$. By the spectral theorem, we can decompose 
	$$
	h=\sum_{j\geq 1}\lambda_j |u_j\rangle \langle u_j|,
	$$
	where $(\lambda_j, u_j)_{j\geq 1}$ are the eigenvalues and eigenfunctions of $h$ satisfying
	$$
	0<\lambda_1\leq \lambda_2\leq ...\leq \lambda_j \to +\infty
	$$
	and $(u_j)_{j\geq 1}$ forms an orthonormal basis of $\mathfrak {h}$. The assumption $s>6$ implies that the $L^2$-mass will be well-defined and finite in the limit we shall consider, without any renormalization~\cite{LewNamRou-15,DinRou-23}. This in fact holds true for any $s>2$, but a technical limitation in our proof imposes the stronger condition $s>6$ to deal with the mean-field limit.
	
	As regards the interacting part of the Hamiltonian we work under (see remarks below for possible generalizations)
	
	\begin{assumption}[\textbf{The interaction potential}]\label{asum:int}\mbox{}\\
	Let $w:\R \mapsto \R$ be an even function (or distribution). We assume that it can be decomposed in positive and negative parts 
	$$ w = w_+- w_-, \quad w_+, w_- \geq 0$$
	with 
	\begin{itemize}
	 \item $w_+\in \cM (X) +  L^p (X) \text{ with } 1< p \leq \infty,$ where $\cM$ is the set of Radon measures
    \item $w_- \in L^p (X)$ with $p>\frac{s}{s-2}$.
	\end{itemize}
		\end{assumption}

	\subsection{Quantum and classical models} We set $g \geq 0$ to play the role of an effective coupling constant and define the $N$-body Hamiltonian acting on $\gh ^N = \gh ^{\otimes_\sym N}=L^2_{\sym} (X^N)$
	\begin{equation}\label{eq:N hamil}
	 H_{N,g} := \sum_{j=1}^N h_{\bx_j} + \frac{g}{N} \sum_{1\leq j<k\leq N} w\left(\bx_j-\bx_k\right) 
	\end{equation}
	The canonical Gibbs state is the trace-class operator
	\begin{equation}\label{eq:C ens}
\Gamma_{N,T,g}^c := \frac{1}{Z^c_{N,T,g}} \exp\left( -\frac{1}{T} H_{N,g} \right)	 
	\end{equation}
where the partition function $Z^c_{N,T,g}$ sets the trace equal to $1$. For any $k\geq 1$ we associate to $\Gamma_{N,T,g}^c$, or any other $N$-particles state (positive trace-class operator over $\gh^N$) $\Gamma_N$, its $k$-particles reduced density matrix via a partial trace over $N-k$ variables
\begin{equation}\label{eq:def red mat}
\Gamma_N ^{(k)}:= {N \choose k} \Tr _{k+1 \to N} \left[\Gamma_N\right].
\end{equation}

We next define the limiting non-linear Gibbs measure, starting with the Gaussian measure with covariance $h^{-1}$. 
	We denote
	\begin{align}\label{P-project}
	P_\Lambda : = \sum_{\lambda_j \leq \Lambda} |u_j \rangle \langle u_j|, \quad P_\Lambda^{\perp} = \Id_{\gh} - P_\Lambda
	\end{align}
	and the associated subspaces 
	\begin{equation}\label{eq:subspaces}
	E_\Lambda = P_\Lambda \mathfrak{h} = \spann\{u_j : \lambda_j \leq \Lambda\}, \quad
	E_\Lambda ^\perp = P_\Lambda^\perp \mathfrak{h}.
	\end{equation}
	Let also, for $\theta \in \R$ 
			\begin{equation}\label{eq:Sob}
			\mathcal H^\theta := \left\{ u =\sum_{j \geq 1} \alpha_j u_j : \sum_{j\geq 1} \lambda_j^{\theta} |\alpha_j|^2<\infty \right\}, \quad \alpha_j = \langle u_j, u\rangle \in \mathbb C
			\end{equation}
			be the Sobolev space associated with $h$ (as defined in~\eqref{eq:hamil 1body}), equipped with the obvious norm $\norm{\, . \,}_{\cH^\theta}$. By convention the $L^2$ norm is denoted $\norm{\cdot}= \norm{\cdot}_{\cH^0}$ with associated scalar product $\langle\cdot,\cdot\rangle.$ 
	
	\begin{definition}[\textbf{The reference Gaussian measure}]\label{def:gauss measure}\mbox{}\\
	 The sequence of probability measures 
	 \begin{equation}\label{eq:gaus meas lambda}
	  d\mu_{0,\Lambda}(u) := \prod_{\lambda_j \leq \Lambda} \frac{\lambda_j}{\pi} e^{-\lambda_j |\alpha_j|^2} d \alpha_j
	 \end{equation}
    over $E_\Lambda$ is tight in $\cH^\theta$ for any $0<\theta < 1/2 - 1/s$. It converges to a limit gaussian measure $\mu_0$ over $\cH^\theta$. Moreover, for all $\Lambda >0$, $\mu_{0,\Lambda}$ is the finite-dimensional cylindrical projection of $\mu_0$ on $E_\Lambda$.
	\end{definition}

	We refer e.g. to the introductory parts of~\cite{LewNamRou-15} for more background on the above. We proceed with
	
	\begin{definition}[\textbf{Fixed mass Gaussian measure}]\label{def:gauss meas mass}\mbox{}\\
	 For $\Lambda \geq \lambda_1$, let $\mu_{0,\Lambda}^\perp$ be the cylindrical projections of $\mu_0$ on $E_\Lambda^\perp$, namely
	 $$
	 d\mu_{0,\Lambda}^\perp(u) = \prod_{\lambda_j >\Lambda} \frac{\lambda_j}{\pi} e^{-\lambda_j |\alpha_j|^2} d \alpha_j.
	 $$
	 Let $f_\Lambda$ be the density function of the random variable $\norm{P_\Lambda^\perp u}^2$ with respect to $d\mu_{0,\Lambda}^\perp (u)$, with the convention that $f_0$ is the density function of $\norm{ u}^2$ with respect to $d\mu_{0} (u)$. 
	 
	 For $m>0$ we define a measure on $E_\Lambda$ by setting 
	 \begin{equation}\label{eq:def mass measure}
	 d \mu_{0,m,\Lambda} (u):= \frac{f_\Lambda \left(m- \norm{P_\Lambda u}^2\right)}{f_0(m)} d \mu_{0,\Lambda}(u).
	 \end{equation}
For any $0<\theta < 1/2 - 1/s$ the sequence of measures $(\mu_{0,m,\Lambda})_\Lambda$ is tight on $\cH^{\theta}$ and admits as limit a probability measure $\mu_{0,m}$ on $\cH^{\theta}$, which is concentrated on the sphere $\left\{ \norm{u}^2 = m\right\}.$
	\end{definition}

	Note that $P_\Lambda^\perp u = u$ on the support of $d\mu_{0,\Lambda}^\perp (u)$. It is however convenient to think of $f_\Lambda$ as the density of $\norm{P_\Lambda^\perp u}^2$. We refer to Section~\ref{sec:def cond meas} below for more details, including the proof of tightness and the fact that the limit measure is a probability. In particular, the measure $\mu_{0,m}$, which is the restriction of $\mu_0$ over a $L^2$-sphere, may naturally be defined as a limit $\eps\to 0$ of restrictions $\mu_{\eps,m}$ of $\mu_0$ over $L^2$-annuli of thickness $\eps \to 0$. It can also be defined as the limit of measures restricted to finite dimensional spheres, see Lemma~\ref{lem:meas sphere}.
	
	The interacting Gibbs measure with fixed mass is absolutely continuous with respect to $\mu_{0,m}$: 
	
	\begin{definition}[\textbf{The non-linear Gibbs measure}]\label{def:int measure}\mbox{}\\
	 The functional
	 \begin{equation}\label{eq:weight} 
	 u \mapsto \exp\left( - \frac{g}{2m} \iint_{X \times X } |u(\bx)| ^2 w (\bx-\by) |u(\bx)| ^2 d\bx d\by\right)	  
	 \end{equation}
	 is in $L^1 (d\mu_{0,m})$. We hence define a probability measure over $\cH ^\theta, 0 < \theta < 1/2 - 1/s$ by setting 
	 \begin{equation}\label{eq:int meas}
	 d\mu_{g,m} (u) := \frac{1}{z^r_m}\exp\left( - \frac{g}{2m} \iint_{X \times X } |u(\bx)| ^2 w (\bx-\by) |u(\bx)| ^2 d\bx d\by\right) d\mu_{0,m} (u)
	 \end{equation}
    with the classical relative partition function
    \begin{equation}\label{eq:class part}
     z^r_m = \int \exp\left( - \frac{g}{2m} \iint_{X \times X} |u(\bx)| ^2 w (\bx-\by) |u(\bx)| ^2 d\bx d\by\right)  d\mu_{0,m}(u).
    \end{equation}
	 \end{definition}

	We will recap why the weight~\eqref{eq:weight} is integrable in Appendix~\ref{sec:sub-inte-meas}, using tools from e.g.~\cite{BurThoTzv-10,DinRou-23}.
	
	\subsection{Semi-classical limit.} We can now state our main theorem:
	
	\begin{theorem}[\textbf{Mean-field limit of the bosonic canonical ensemble}]\label{thm:main}\mbox{}\\
	 Let the interacting canonical Gibbs state $\Gamma_{mT,T,g}^c$ be defined as in~\eqref{eq:C ens} with the particle number set as $N = mT, m >0$. Let the corresponding reduced density matrices be as in~\eqref{eq:def red mat}. We fix $m>0,g \geq 0$ and let $T\to \infty$. Then, for any $k\geq 1$
	 \begin{equation}\label{eq:main DM}
	 \frac{k!}{T^k} \left(\Gamma_{mT,T,g}^c \right)^{(k)} \xrightarrow[T\to \infty]{}\int |u^{\otimes k} \rangle \langle u^{\otimes k} | d\mu_{g,m} (u)
	 \end{equation}
    strongly in the trace-class $\gS ^1 (\gh^k)$. In particular, the interacting Gibbs measure with fixed mass $\mu_{g,m}$ given by Definition~\ref{def:int measure} is the quantum de Finetti measure at scale $T^{-1}$ of the sequence $\left( \Gamma_{mT,T,g}^c\right)_T$, in the sense of~\cite[Definition~4.1, Theorem~4.2]{LewNamRou-15}.
    
    Moreover, the relative quantum free-energy satisfies 
    \begin{equation}\label{eq:main ener}
     - \log \left(Z^c_{mT,T,g}\right) +  \log \left(Z^c_{mT,T,0}\right) \xrightarrow[T\to \infty]{} -\log z^r_m
    \end{equation}
    where $Z^c_{mT,T,g}$ is the partition function normalizing~\eqref{eq:C ens} and $z^r_m$ is the classical relative partition function~\eqref{eq:class part}.
	\end{theorem}

	We assume the exact identity $N=m T$ for simplicity. The result still holds if $N \sim mT$ in the limit $N,T\to \infty$.
	
	The above is the analogue of~\cite[Theorem~5.3]{LewNamRou-15} at fixed particle number, and allowing for attractive interactions, which is not possible in the grand-canonical framework of~\cite{FroKnoSchSoh-16,FroKnoSchSoh-17,FroKnoSchSoh-20,FroKnoSchSoh-22,LewNamRou-15,LewNamRou-17,LewNamRou-20}. Our general approach to the proof is to use the (relative) variational principles defining the quantum state $\Gamma_{mT,T,g}^c$ and the classical measure $\mu_{g,m}$, connecting them in the spirit of $\Gamma$-convergence. As in~\cite{LewNamRou-15,LewNamRou-17,LewNamRou-20} it is of some importance to characterize variationally the difference in the left-hand-side of~\eqref{eq:main ener}, for both terms taken separately diverge quite fast.
	
	\subsection{Proof strategy}
	
	The Gibbs variational principle characterizing $\Gamma_{mT,T,g}^c$ states 
	\begin{equation}\label{eq:quant Gibbs}
	\cF_{g} \left[ \Gamma_{N,T,g}^c\right] = - T \log \left(Z^c_{N,T,g}\right) = \inf\left\{\cF_{g} \left[ \Gamma_{N}\right], \; \Gamma_N \mbox{ a $N$-particles bosonic state} \right\}
	\end{equation}
    with the free-energy functional
	\begin{equation}\label{eq:free ener func}
	 \cF_{g} \left[ \Gamma_{N}\right] := \Tr\left[ H_{N,g} \Gamma_N \right] + T \Tr\left[\Gamma_N \log \Gamma_N\right].
	\end{equation}
    It follows that
	\begin{equation}\label{eq:free ener rel}
	-\log \left(\frac{Z^c_{mT,T,g}}{Z^c_{mT,T,0}}\right) = \inf_{\Gamma \in \mathfrak S^1(\mathfrak h^{mT}) \atop \Gamma \geq 0, {\Tr}[\Gamma]=1} \left\{ \mathcal H(\Gamma, \Gamma^c_{mT,T,0}) + \frac{g}{TN} {\Tr}\left[w(\bx-\by)\Gamma^{(2)}\right] \right\},
	\end{equation}
	with the Von Neumann quantum relative entropy
	$$
	\mathcal H\left(\Gamma, \Xi\right):= {\Tr}\left[\Gamma \right(\log \Gamma-\log \Xi\left)\right].
	$$
	Moreover, $\Gamma^c_{mT,T,g}$ is the unique solution of the above minimization problem. 
	
	Similarly, 
	\begin{equation}\label{eq:free ener class}
	-\log(z^r_m) = \inf_{\nu \text{ prob. meas. }} \left\{ \mathcal H_{\cl}(\nu, \mu_{0,m}) + \frac{g}{2m} \int \left( \iint_{X \times X} |u(\bx)|^2  w(\bx-\by) |u(\by)|^2 d\bx d\by \right) d\nu(u)\right\},
	\end{equation}
	with the classical relative entropy
	\begin{equation}\label{eq:class rel ent}
	\mathcal H_{\cl}(\nu,\sigma):= \int \frac{d\nu}{d\sigma} \log \frac{d\nu}{d\sigma} d\sigma
	\end{equation}
	and $\mu_{g,m}$ is the unique minimizer of the above. We shall prove matching upper and lower bounds relating the quantum~\eqref{eq:free ener rel} and classical~\eqref{eq:free ener class} relative free-energies in the limit $T\to \infty$.

	To obtain a good free-energy upper bound we need to take into account that what~\eqref{eq:main ener} means is that highly energetic particles behave as in the free ensemble $\Gamma_{mT,T,0}^c$ while at low energies we have a semi-classical behavior. We capture this by using a trial state which ``looks like $\Gamma^c_{M,T,0}$ in a subspace with $M$ particles in $E_\Lambda^\perp$ (with large one-particle energy)'' and ``looks like'' 
	$$ \int \left|u^{\otimes (N-M)}\right\rangle \left\langle u^{\otimes(N-M)} \right| d\mu_{g,m} (u)$$
	on a subspace where $N-M=mT-M$ particles are in $E_\Lambda$ (with moderate one-particle energy, $E_\Lambda = \mathds{1}_{h\leq \Lambda} \gh$ essentially). We let $\Lambda \to \infty$ when $T\to \infty$ and optimize/average over $M$. In the grand-canonical case of~\cite{LewNamRou-15} this construction is straightforward, using the factorization property of the Fock space 
	$$ \gF (\gh) \simeq \gF (E_\Lambda) \otimes \gF (E_\Lambda^\perp).$$
	In the more constrained canonical case that we tackle here, the construction and estimate of the free-energy is much more delicate. It will require some new input regarding the classical minimization problem, and detailed bounds on the quantum free canonical ensemble, some of which we borrow from~\cite{DeuSeiYng-18,DeuSei-19,Suto-04}. The restriction to $s>6$ arises in this construction. As per the procedure above, we have to consider, in the low-energy subspace, a superposition of measures conditioned on masses slightly different from $m$ (because some of the mass/particles goes to the high-energy subspace instead). Controlling the error thus made turns out to be delicate, related to asking how much does the normalization of a Gaussian measure depend on its covariance, see Section~\ref{sec:def cond meas} below.
	
	To obtain a free-energy lower bound we use the quantum de Finetti theorem (see~\cite{Rougerie-LMU,Rougerie-spartacus,Rougerie-EMS} for review), which essentially asserts that \emph{any} reasonable sequence of $N$-particle states $(\Gamma_N)_N$ satisfies  (modulo subsequence)
	\begin{equation}\label{eq:deF inf}
	\Gamma_N \simeq \int \left|u^{\otimes N}\right\rangle \left\langle u^{\otimes N} \right| d\nu (u)
	\end{equation}
	for some probability measure $\nu$, called the de Finetti measure of the (sub)sequence  $(\Gamma_N)_N$. Inserting such a form in the interaction energy term of~\eqref{eq:free ener rel} immediately yields a classical energy term akin to that of~\eqref{eq:free ener class}. Hence we aim at passing to the lim inf in the interaction energy in a sufficiently strong sense to be able to use the (rigorous version of)~\eqref{eq:deF inf} in the resulting lower bound, as in~\cite{LewNamRou-15}. However, we need extra control and arguments for the attractive part of the interaction since we allow for such a possibility.

	As regards the relative entropy term in~\eqref{eq:free ener rel} we use the Berezin-Lieb-type inequality of~\cite[Theorem~7.1]{LewNamRou-15}. If the sequences of $N$-particles states $(\Gamma_N)_N,(\Xi_N)_N$ have de Finetti measures $\nu,\sigma$ (vaguely, in the sense of~\eqref{eq:deF inf}) then 
	$$ \liminf_{N\to \infty} \cH (\Gamma_N,\Xi_N) \geq \cH_{\cl} (\nu,\sigma)$$
	which has the desired form to conclude the proof of the free-energy lower bound, provided we can prove separately that the free Gibbs measure conditioned on mass $\mu_{0,m}$ is indeed the de Finetti measure of the free canonical state $\Gamma^c_{mT,T,0}$. The equivalent of this input in the grand-canonical case is rather straightforward. We may apply the quantum and classical Wick theorems to compute expectations of observables in the quantum and classical ensembles respectively, and relate both formulas explicitly. In the canonical case of our concern, such explicit computations are not available to relate $\Gamma^c_{mT,T,0}$ to $\mu_{0,m}$. We follow instead a three-steps procedure (cf Figure~\ref{fig:diagram} below): 
	
	\textbf{First} $\mu_{0,m}$ is related to a sequence $\epsilon\to 0$ of measures  $\mu_{\epsilon,m}$ with a relaxed mass constraint, i.e. weighting the Gaussian measure $\mu_0$ by a factor 
	$$\exp\left(- \eps^{-1} \left(\norm{u} ^2 - m \right) ^2\right)$$
	and normalizing. As mentioned above, this can be seen~\cite{OhQua-13,Brereton-19,DinRou-23} as a definition of $\mu_{0,m}$. 
	
	\textbf{Second} a corresponding bosonic grand-canonical $\Gamma_{\epsilon,m,T}$ ensemble with ``interaction energy'' 
	$$\epsilon ^{-1} T^{-2} \left( \cN - mT\right)^2$$
	is shown to converge when $T\to \infty$ to $\mu_{\epsilon,m}$. 
	
	\textbf{Third}  $\Gamma_{\epsilon,m,T}$ is shown to approximate $\Gamma^c_{mT,T,0}$ when $\epsilon \to 0$. The desired result (i.e. Theorem~\ref{thm:main} in the case $g = 0$) is obtained by commuting the $T\to \infty$ and $\epsilon\to 0$ limits. This is the delicate part: for the second step, which relies on the general framework for grand-canonical ensembles of~\cite{LewNamRou-15,LewNamRou-20} we need to make the arguments quantitative and track down $\epsilon$-dependencies of error terms. For the third step we need a careful investigation of the $N$-dependence of the free canonical ensemble, relying in particular on combinatorial considerations we learned from~\cite{Cannon-73}.
	
	\subsection{Possible extensions} We would like to mention several possible/future extensions of our main result. Most of them have to do with the case of focusing/attractive interactions.
	
	\medskip
	
	\noindent\textbf{Three-body interactions and quintic NLS.} We could probably include with minimal effort (properly scaled) three-particles interaction in~\eqref{eq:N hamil}, or even higher order, as long as they are repulsive. A most interesting question concerns three-particle attractive interactions, leading to the Gibbs measure of a focusing quintic NLS equation as in~\cite{RouSoh-23}. The latter has a ``phase transition''~\cite{LebRosSpe-88,CarFroLeb-16,Bourgain-94,OhSosTol-22,DinRouTolWan-23}, meaning it can be constructed only for sufficiently small $L^2$-mass, and blows up for larger ones. Our analysis can be adapted rather directly to the case of an attractive three-particle potential $0 \geq W\in L^p (\R\times \R)$ for an appropriately chosen $p(s)$ (adapting the condition on $w_-$ from Assumption~\ref{asum:int} that we use for the free-energy lower bound). The derivation would work for any subcritical mass, as long as the limiting classical measure is well-defined. We plan to return to this question, in connection with the next one. 
	
	\medskip
	
	\noindent\textbf{Singular interactions.} As regards repulsive interactions, we have made essentially no unnecessary assumptions. The $L^p$ condition for $w_-$ in Assumption~\ref{asum:int} is however far from optimal, and we hope to be able to relax it greatly in the future. A related, perhaps more directly relevant, direction is to consider a scaled two-particles interaction ($\beta >0$ being fixed)
	$$w_N (\bx) = N^{\beta} w (N^\beta \bx) \xrightarrow[N\to \infty]{} \left(\int_\R w \right)\delta_0$$
	or three-particles interaction ($\gamma >0$)
	$$W_N (\bx,\by) = N^{2\gamma} W (N^\gamma \bx, N^\gamma \by) \xrightarrow[N\to \infty]{} \left(\int_{\R^2} W \right) \delta_{0,0}$$
	to derive bona-fide local NLS equations in the limit. A natural question (left open in~\cite{RouSoh-22,RouSoh-23}) is whether one can control the limit to an extent allowing for $\beta,\gamma >1$ corresponding to dilute interactions\footnote{With range much smaller than the typical interparticle distance.}, especially in the focusing case.

	\medskip
	
	\noindent\textbf{Cases of infinite mass.} When either the spatial dimension $d\geq 2$ or the trap's exponent $s\leq 2$ in~\eqref{eq:hamil 1body}, the $L^2$ norm is not well-defined on the support of the Gaussian measure, for rather different reasons (lack of local regularity in the first case, lack of decay at infinity in the second). One can instead condition on a renormalized version of the mass~\cite{BurThoTzv-10,DinRou-23} and scale the particle number accordingly in the quantum model. For repulsive interactions we expect this to be feasible, adapting the approach of~\cite{LewNamRou-17,LewNamRou-20} in the cases $s<2$ and $d\geq 2$ respectively (the latter would require some effort, since the grand-canonical case is already rather involved~\cite{FroKnoSchSoh-20,FroKnoSchSoh-22,LewNamRou-20}). For attractive interactions in $d\geq 2$ there is a no-go theorem \cite{BrySla-96} regarding the construction of the measure. The question stays interesting for $d=1$ and $s\leq 2$ where the control of attractive interactions would require ideas beyond~\cite{LewNamRou-17}. Probably, finding a way to relax the condition $s>6$ to the more natural $s>2$ in the case without mass renormalization is a pre-requisite for this problem.

	\section{The Gibbs measure conditioned on mass}\label{sec:def cond meas}
	
	In this section we recall some elements from~\cite{OhQua-13,DinRou-23} about the definition of the conditioned Gibbs measure. We also add some elements that we specifically need in the sequel, especially in Section~\ref{sec:dep mass} below. 
	
	\subsection{More on definitions}
	
	In Definition~\ref{def:gauss meas mass} above we have given a concise definition of the Gaussian measure conditioned on mass. We now connect this to a natural approach where the mass is penalized instead of fixed. This will be very useful in Section~\ref{sec:free} below where we connect this measure to many-body quantum mechanics. 
	
	Let $\mu_0$ be the Gaussian measure from Definition~\ref{def:gauss measure} and $\epsilon >0$. We penalize the $L^2$-mass to define
	\begin{equation}\label{eq:def mu epsilon}
	d\mu_{\epsilon,m}(u) = \frac{1}{z_{\epsilon,m}^r} \exp\left(-\frac{1}{\epsilon} \left(\int |u|^2 -m \right)^2 \right) d\mu_0 (u)
	\end{equation}
	where $z_{\epsilon,m}^r$ normalizes $\mu_{\epsilon,m}$ as a probability measure. We connect the above to $\mu_{0,m}$ from Definition~\ref{def:gauss meas mass} by proving the
	
	\begin{proposition}[\textbf{Density matrices of the conditioned Gaussian measure}]\label{pro:fixed mass meas}\mbox{}\\
	 For all $k\geq 1$,
	\begin{align} \label{limi-eps}
		\int |u^{\otimes k} \rangle \langle u^{\otimes k}| d\mu_{\epsilon,m}(u) \xrightarrow[\epsilon \to 0]{}  \int |u^{\otimes k} \rangle \langle u^{\otimes k}| d\mu_{0,m}(u)  \text{ strongly in } \mathfrak{S}^1(\mathfrak {h}^k).
	\end{align}
	Moreover, the classical relative partition function $z_{\epsilon,m}^r$ satisfies
	\begin{align} \label{limi-eps-parti}
		\lim_{\epsilon \to 0} \frac{z^r_{\epsilon,m}}{\sqrt{\epsilon}} =: c(m)>0.
	\end{align}
	\end{proposition}

	This approach to defining the measure originates in~\cite{OhQua-13} and has been used e.g. in~\cite{Brereton-19,DinRou-23}. It is equivalent to taking~\eqref{eq:def mass measure} as starting point as above. We recall this before turning to the proof of Proposition~\ref{pro:fixed mass meas}.
	
		\begin{proposition}[\textbf{Gibbs measure conditioned on mass}] \label{theo-fix-mass-meas}\mbox{}\\
		Let $s>2$.
		\begin{itemize}
			\item[(1)] For any borelian set $A$ of $E_\Lambda$, the limit $\lim_{\epsilon \to 0} \mu_{\epsilon,m}(A)$ exists, and we may define a probability measure on $E_\Lambda$ by setting
			$$
			\mu_{0,m,\Lambda}(A):= \lim_{\epsilon \to 0}\mu_{\epsilon,m}(A).
			$$
			This measure is identical to that defined in~\eqref{eq:def mass measure}.
			
			\item[(2)] There exists a unique measure $\mu_{0,m}$ supported in $\mathcal H^\theta \cap \{\|u\|^2=m\}$ with $0<\theta <\frac{1}{2}-\frac{1}{s}$ such that for all $\Lambda \geq \lambda_1$, $\mu_{0,m,\Lambda}$ is the cylindrical projection of $\mu_{0,m}$ on $E_\Lambda$. Moreover, for any measurable set $A \subset \mathcal H^\theta$,
			$$
			\mu_{0,m}(A) = \lim_{\epsilon \to 0} \mu_{\epsilon,m}(A).
			$$
		\end{itemize}
	\end{proposition}
	
	We will need the following lemma, which is based on ideas from~\cite{OhQua-13} (see also~\cite[Lemma 6.3 and Lemma 6.4]{DinRou-23}).
		
		\begin{lemma}[\textbf{Density functions of the $L^2$ norm}]\label{lem-f-Lambda}\mbox{}\\
			For any $\Lambda \geq 0$, let $f_\Lambda,g_\Lambda$ be the density functions of $\|P^\perp_\Lambda u\|_{L^2}^2$ (respectively $\|P_\Lambda u\|_{L^2}^2$) with respect to $\mu^\perp_{0,\Lambda}$ (respectively $\mu_{0,\Lambda}$). We have the following bounds
			\begin{equation}\label{eq:dens func bound}
			\|f_\Lambda\|_{L^\infty((0,+\infty))} \leq C\Lambda
			\end{equation}
			and 
			\begin{equation}\label{eq:dens func bound 2}
			\|g_\Lambda\|_{L^\infty((0,+\infty))} \leq C, \quad \|g'_\Lambda\|_{L^\infty((0,+\infty))} \leq C
			\end{equation}
			Moreover, as measures\footnote{I.e. the convergence is valid when tested against continuous compactly supported functions.}, 
			\begin{equation}\label{eq:lim f Lambda}
			 f_\Lambda \xrightarrow[\Lambda \to \infty]{} \delta_0
			\end{equation}
            with $\delta_0$ the Dirac delta function at the origin, whereas  
            \begin{equation}\label{eq:lim g Lambda}
			 g_\Lambda \xrightarrow[\Lambda \to \infty]{} f_0.
			\end{equation}
            In addition, $f_0(m)>0$ for any $m>0$.
		\end{lemma}
		
		 \begin{proof}
	Denote $\phi_\Lambda$ the characteristic function of $\norm{P_\Lambda ^\perp u} ^2$ with respect to $\mu_{0,\Lambda}^\perp$. We have
	\begin{align*}
		\phi_\Lambda(s) &= \E_{\mu_{0,\Lambda}^\perp}\left[e^{is \norm{P_\Lambda ^\perp u} ^2}\right] \\
		&= \int_{\Hc^\theta} e^{is \sum_{\lambda_j>\Lambda} |\alpha_j|^2} d\mu_{0,\Lambda}^\perp(u) \\
		&= \int \prod_{\lambda_j>\Lambda} e^{is |\alpha_j|^2} \prod_{\lambda_k>\Lambda} \frac{\lambda_k}{\pi} e^{-\lambda_k|\alpha_k|^2} d\alpha_k \\
		&= \prod_{\lambda_j>\Lambda} \left(\int_{\C}  \frac{\lambda_j}{\pi}e^{-(1-is\lambda_j^{-1})\lambda_j |\alpha_j|^2} d\alpha_j\right) \left(\prod_{\lambda_k\ne \lambda_j\atop \lambda_k> \Lambda} \int_{\C} \frac{\lambda_k}{\pi} e^{-\lambda_k|\alpha_k|^2} d\alpha_k \right) \\
		&=\prod_{\lambda_j>\Lambda} \frac{e^{-is \lambda_j^{-1}}}{1-is \lambda_j^{-1}}. 
	\end{align*}
	Each factor of this product has complex norm smaller than or equal to 1, thus the norm of the product is bounded by the norm of a product of any number of terms. In particular, we have
	\[
	|\phi_\Lambda(s)| \leq \left| \frac{e^{-is\lambda_{j_1}^{-1}}}{1-is\lambda_{j_1}^{-1}}\frac{e^{-is\lambda_{j_2}^{-1}}}{1-is\lambda_{j_2}^{-1}}\right| = \frac{1}{\sqrt{1+s^2\lambda_{j_1}^{-2}}} \frac{1}{\sqrt{1+s^2\lambda_{j_2}^{-2}}} \leq \frac{1}{1+s^2\lambda_{j_2}^{-2}},
	\]
	where $j_1$ and $j_2$ are the first two indices such that $\lambda_{j_2}\geq \lambda_{j_1}>\Lambda$. We deduce that $\phi_\Lambda \in L^1(\R)$ with 
	$$\|\phi_\Lambda\|_{L^1(\R)} \leq C\lambda_{j_2}.$$
	As in the proof of~\cite[Lemma~6.3]{DinRou-23} we can choose $\lambda_{j_2}$ of order $\Lambda$ in the above, leading to~\eqref{eq:dens func bound}, using the relation between density and characteristic functions
	\[
	f_\Lambda(x) = \frac{1}{2\pi} \int_{\R} e^{-isx} \phi_\Lambda(s) ds.
	\]
	For \eqref{eq:lim f Lambda}, we will show that for any continuous and compactly supported test function $\phi$ on $[0,+\infty)$, 
	\begin{align} \label{eq: delta limi}
		\lim_{\Lambda \to 0} \int_0^{+\infty} f_\Lambda(x) \phi(x) dx = \phi(0).
	\end{align}
		Let $\delta>0$. Since $\phi$ is continuous at 0, there exists $M=M(\delta)>0$ such that for $0\leq x\leq M$,
		$$
		|\phi(x)-\phi(0)|<\delta/2.
		$$
		We write
		$$
		\int_0^{+\infty} f_\Lambda(x)\phi(x)dx -\phi(0)= \Big(\int_0^M f_\Lambda(x) \phi(x) dx -\phi(0)\Big)+ \int_M^{+\infty} f_\Lambda (x) \phi(x) dx = (\text{I}) + (\text{II}).
		$$
		We estimate
		$$
		\begin{aligned}
		|(\text{I})| &\leq \Big|\phi(0) \Big(\int_0^M f_\Lambda(x) dx - 1\Big)\Big| + \Big|\int_0^M f_\Lambda(x)(\phi(x)-\phi(0)) dx\Big| \\
		&\leq |\phi(0)| \int_M^{+\infty} f_\Lambda(x) dx + \delta/2 \int_0^M f_\Lambda(x)dx \\
		&\leq |\phi(0)| \int_M^{+\infty }f_\Lambda(x) dx + \delta/2,
		\end{aligned}
		$$
		where we used that $f_\Lambda(x) \geq 0$ for all $x\in [0,\infty)$ and $\displaystyle\int_0^{+\infty} f_\Lambda(x) dx=1$.
		We also have
		$$
		|(\text{II})| \leq \|\phi\|_{L^\infty} \int_M^{+\infty} f_\Lambda(x)dx,
		$$
		hence
		$$
		\Big| \int_0^{+\infty} f_\Lambda(x) \phi(x) dx -\phi(0)\Big| \leq \delta/2 + 2 \|\phi\|_{L^\infty} \int_M^{+\infty} f_\Lambda(x)dx.
		$$
		In addition, we bound for any $M>0$,
		$$ 
		\int_M ^{+\infty} f_\Lambda(s) ds = \int \Id_{\left\{\norm{P_\Lambda^\perp u}_{L^2}^2 \geq M\right\}}  d\mu_{0} (u) \leq \frac{1}{M\Lambda ^\theta} \int \norm{u}_{\cH ^\theta}^2 d\mu_{0} (u) \leq C \frac{1}{M\Lambda ^\theta} \xrightarrow[\Lambda \to \infty]{} 0$$
		where we chose some $0<\theta < 1/2 - 1/s$. Hence, for any $\delta >0$ we may find $\Lambda_\delta$ sufficiently large to get
		$$
		\Big| \int_0^{+\infty} f_\Lambda(x) \phi(x) dx -\phi(0)\Big| \leq \delta,
		$$
		for any $\Lambda \geq \Lambda_\delta$ which vindicates \eqref{eq: delta limi}.
		
		We proceed similarly for the proof of~\eqref{eq:dens func bound 2}. Let $\psi_\Lambda$ be the characteristic function of $\|P_\Lambda u\|_{L^2}^2$ with respect to $\mu_{0,\Lambda}$. We compute
		$$
		\begin{aligned}
			\psi_\Lambda(s) &= \mathbb E_{\mu_{0,\Lambda}} \left[e^{is \|P_\Lambda u\|_{L^2}^2} \right] \\
			&= \int e^{is \|P_\Lambda u\|_{L^2}^2} d\mu_{0,\Lambda}(u) \\
			&= \int e^{is \sum_{\lambda_j \leq \Lambda} |\alpha_j|^2} \prod_{\lambda_k \leq \Lambda} \frac{\lambda_k}{\pi}e^{-\lambda_k |\alpha_k|^2} d\alpha_k \\
			&= \prod_{\lambda_j \leq \Lambda} \frac{1}{1-is \lambda_j^{-1}}.
		\end{aligned}
		$$
		As above, the modulus of this product is bounded by the product of the moduli of any two factors. Taking the first two factors, we get
		$$
		|\psi_\Lambda(s)| \leq \frac{1}{1+s^2 \lambda_2^{-2}}
		$$
		hence $\|\psi_\Lambda \|_{L^1(\R)} \leq C \lambda_2$. In particular, $g_\Lambda$ is uniformly bounded, independently of $\Lambda$. Taking more factors into account we obtain bounds on norms of $\psi_\Lambda$ in $L^1(\R,sds)$ and hence $L^\infty$ bounds on the derivatives of $g_\Lambda$ as in~\eqref{eq:dens func bound 2}. We can then extract a subsequence along which 
		$$ g_\Lambda \underset{\Lambda \to \infty}{\longrightarrow} g_\infty$$
		uniformly. Since $f_\Lambda,g_\Lambda$ are the density functions of independent random variables adding to $\norm{u}^2$, whose density function is $f_0$, we have that 
		$$f_0 = g_\Lambda * f_\Lambda.$$
		Passing to the limit in the above and using~\eqref{eq:lim f Lambda} we deduce that 
		$$ f_0 = g_\infty * \delta_0$$
		and thus that $g_\infty = f_0$, which proves~\eqref{eq:lim g Lambda} by uniqueness of the limit. The proof that $f_0 (m) >0$ is exactly similar to that of~\cite[Lemma~6.4]{DinRou-23}.
		\end{proof}

	\begin{proof}[Proof of Proposition~\ref{theo-fix-mass-meas}]
		We follow closely arguments from~\cite{DinRou-23}, which consist of two main steps:
		
		\begin{itemize}
			\item We first find a sequence of measures on the finite dimensional spaces $E_\Lambda$, which satisfies a cylindrical projection property.
			
			\item We then show that this sequence of measures is tight on a suitable Hilbert space, and the fixed-mass measure can be defined as its limit.
		\end{itemize}
		
		\medskip
		
		\noindent{\bf Step 1. Cylindrical projections.} Let $\Lambda  \geq \lambda_1$ and $A$ be a borelian set of $E_\Lambda$. We claim that, with $\mu_{\epsilon,m}$ as in~\eqref{eq:def mu epsilon}, we have 
		\begin{align}\label{conv-eps}
				\lim_{\epsilon \to 0} \mu_{\epsilon,m}(A) = \int_A \frac{f_\Lambda(m-\|P_\Lambda u\|_{L^2}^2)}{f_0(m)} d\mu_{0,\Lambda}(u) =:\mu_{0,m, \Lambda}(A).
			\end{align}
			In particular, 
			$$
			d\mu_{0,m,\Lambda}(u) = \frac{f_\Lambda(m-\|P_\Lambda u\|_{L^2}^2)}{f_0(m)} d\mu_{0,\Lambda}(u)
			$$
			is indeed as given in~\eqref{eq:def mass measure} and satisfies for any $\Theta \geq \Lambda$,
			\begin{align} \label{cyli-proj}
				\mu_{0,m, \Theta}\vert_{E_\Lambda} = \mu_{0,m,\Lambda}.
			\end{align}
		We start by writing
		$$
		\begin{aligned}
			\mu_{\epsilon, m}(A) &= \frac{1}{z^r_{\epsilon,m}} \int \mathds{1}_A \exp \left(-\frac{1}{\epsilon} (\langle u, u\rangle -m)^2\right) d\mu_0(u) \\
			&= \frac{1}{z^r_{\epsilon,m}} \int \mathds{1}_A \exp \left(-\frac{1}{\epsilon} \left(\|P_\Lambda u\|_{L^2}^2 + \|P^\perp_\Lambda u\|_{L^2}^2 -m\right)^2 \right) d\mu_{0,\Lambda}(P_\Lambda u) d\mu_{0,\Lambda}^\perp (P^\perp_\Lambda u) \\
			&= \frac{1}{z^r_{\epsilon,m}} \int_{A} \left(\int \exp \left(-\frac{1}{\epsilon}\left(\|P^\perp_\Lambda u\|_{L^2}^2 - m + \|P_\Lambda u\|_{L^2}^2\right)^2\right) d\mu_{0,\Lambda}^\perp(u)\right) d\mu_{0,\Lambda}(u),
		\end{aligned}
		$$
		where 
		$$
		d\mu_{0,\Lambda}(u) := \prod_{\lambda_j \leq \Lambda} \frac{\lambda_j}{\pi}e^{-\lambda_j |\alpha_j|^2}d\alpha_j, \quad d\mu_{0,\Lambda}^\perp(u):= \prod_{\lambda_j > \Lambda} \frac{\lambda_j}{\pi}e^{-\lambda_j |\alpha_j|^2}d\alpha_j.
		$$
		In particular, the projection of $\mu_{\epsilon,m}$ on $E_\Lambda$ is given by
		$$
		\begin{aligned}
			d\mu_{\epsilon,m}(u)\vert_{E_\Lambda} &= \frac{1}{z^r_{\epsilon,m}} \left(\int \exp \left(-\frac{1}{\epsilon}\left(\|P^\perp_\Lambda u\|_{L^2}^2 - m + \|P_\Lambda u\|_{L^2}^2\right)^2\right) d\mu_{0,\Lambda}^\perp(u)\right) d\mu_{0,\Lambda}(u) \\
			&=:d\mu_{\epsilon,m, \Lambda}(u).
		\end{aligned}
		$$
		The measure $\mu_{\epsilon,m,\Lambda}$ is indeed the cylindrical projection of $\mu_{\epsilon,m}$ on $E_{\Lambda}$ in the sense that for $\Theta \geq \Lambda$,
		\begin{align} \label{cyli-proj-Lamb}
			\mu_{\epsilon, m, \Theta}\vert_{E_\Lambda} = \mu_{\epsilon, m,\Lambda}.
		\end{align}
		In fact, for a borelian set $A$ of $E_\Lambda$, we observe that $A \times \C^N$ is a borelian set of $E_\Theta$ with $N = \#\{\lambda_j : \Lambda < \lambda_j \leq \Theta\}$, hence
		$$
		\begin{aligned}
			&\mu_{\epsilon,m,\Theta}\vert_{E_\Lambda} (A)  =\int_{A \times \C^N} d\mu_{\epsilon,m,\Theta}(u) \\
			&\quad = \frac{1}{z^r_{\epsilon,m}} \int_{A\times \C^N} \Big(\int \exp\Big(-\frac{1}{\epsilon}\left(\|P_\Theta^\perp u\|_{L^ 2}^2 - m +\|P_\Theta u\|_{L^2}^2\right)^2 \Big) d\mu_{0,\Theta}^\perp(u) \Big) d\mu_{0,\Theta}(u) \\
			&\quad = \frac{1}{z^r_{\epsilon,m}} \int_A \Big(\int_{\C^N} \Big(\int \exp\Big(-\frac{1}{\epsilon}\left(\|P_\Theta^\perp u\|_{L^ 2}^2 - m +\|P_\Theta u\|_{L^2}^2\right)^2 \Big) d\mu_{0,\Theta}^\perp(u) \Big) \prod_{\Lambda < \lambda_j \leq \Theta} \frac{\lambda_j}{\pi} e^{-\lambda_j|\alpha_j|^2} d\alpha_j\Big) d\mu_{0,\Lambda}(u) \\
			&\quad =\frac{1}{z^r_{\epsilon,m}} \int_A \Big(\int \exp\Big(-\frac{1}{\epsilon}\left(\|P^\perp_\Lambda u\|_{L^2}^2 - m +\|P_\Lambda u\|_{L^2}^2\right)^2 \Big) d\mu_{0,\Lambda}^\perp(u) \Big) d\mu_{0,\Lambda}(u) \\
			&\quad = \mu_{\epsilon,m,\Lambda}(A).
		\end{aligned}
		$$
		To further investigate the measures $\mu_{\epsilon,m,\Lambda}$, we set 
		$$
		\nu_\Lambda := \|P_\Lambda u\|_{L^2}^2
		$$
		and denote by $f_\Lambda$ the density function of $\|P^\perp_\Lambda u\|_{L^2}^2$ with respect to $\mu_{0,\Lambda}^\perp$. In particular, we have
		$$
		\int \exp \left(-\frac{1}{\epsilon}\left(\|P^\perp_\Lambda u\|_{L^2}^2 - m + \|P_\Lambda u\|_{L^2}^2\right)^2\right) d\mu_{0,\Lambda}^\perp(u) = \int_0^{+\infty} \exp \left(-\frac{1}{\epsilon}(\eta-m +\nu_\Lambda)^2\right) f_\Lambda(\eta)d\eta 
		$$
		and
		\begin{align} \label{parti-eps}
		z^r_{\epsilon,m} = \int_0^{+\infty} \exp \left(-\frac{1}{\epsilon}(\eta-m)^2\right) f_0(\eta) d\eta.
		\end{align}
		Thus the measure $\mu_{\epsilon,m, \Lambda }$ can be written as
		$$
		d\mu_{\epsilon,m, \Lambda}(u) = \frac{\left(\displaystyle\int_0^{+\infty} \exp \left(-\frac{1}{\epsilon}(\eta-m +\nu_\Lambda)^2\right) f_\Lambda(\eta)d\eta \right)}
			{\left(\displaystyle\int_0^{+\infty} \exp \left(-\frac{1}{\epsilon}(\eta-m)^2\right) f_0(\eta) d\eta\right)} d\mu_{0,\Lambda}(u).
		$$
			We have
			$$
			\begin{aligned}
				\frac{1}{\sqrt{\epsilon}} &\int_0^{+\infty} \exp \left(-\frac{1}{\epsilon}\left(\eta-m+\nu_\Lambda \right)^2\right) f_\Lambda(\eta) d\eta \\
				&=\int_{\frac{1}{\sqrt{\epsilon}}(\nu_\Lambda -m)}^{+\infty} e^{-\eta^2} f_\Lambda(\sqrt{\epsilon} \eta + m -\nu_\Lambda) d\eta \\
				&= \int_{-\infty}^{+\infty} g_\epsilon(\eta) d\eta,
			\end{aligned}
			$$
			where
			$$
			g_\epsilon(\eta) = \mathds{1}_{\left(\frac{1}{\sqrt{\epsilon}} (\nu_\Lambda-m), +\infty\right)} (\eta) e^{-\eta^2} f_\Lambda(\sqrt{\epsilon} \eta +m -\nu_\Lambda).
			$$
			Since $f_\Lambda$ is uniformly continuous, we have for a.e. $\eta\in (-\infty, +\infty)$, $g_\epsilon(\eta) \to e^{-\eta^2} f_\Lambda(m-\nu_\Lambda)$ as $\epsilon \to 0$ and
			$$
			|g_\epsilon(\eta)| \leq e^{-\eta^2} \|f_\Lambda\|_{L^\infty}, \quad \forall \eta \in (-\infty, +\infty).
			$$
			The dominated convergence theorem implies
			$$
			\frac{1}{\sqrt{\epsilon}} \int_0^{+\infty} \exp\left(-\frac{1}{\epsilon}(\eta-m+\nu_\Lambda)^2\right) f_\Lambda(\eta) dx \xrightarrow[\epsilon\to 0]{} f_\Lambda(m-\nu_\Lambda) \sqrt{\pi}.
			$$
			Similarly, we have
			\begin{align} \label{limi-eps-parti-proof}
			\frac{1}{\sqrt{\epsilon}} \int_0^{+\infty} \exp\left(-\frac{1}{\epsilon}(\eta-m)^2 \right) f_0(\eta) d\eta \xrightarrow[\epsilon \to 0]{} f_0(m) \sqrt{\pi}.
			\end{align}
			In particular, we have
			\begin{equation*}
				\frac{\left(\displaystyle\int_0^{+\infty} \exp \left(-\frac{1}{\epsilon}(\eta-m +\nu_\Lambda)^2\right) f_\Lambda(\eta)d\eta \right)}{\left(\displaystyle\int_0^{+\infty} \exp \left(-\frac{1}{\epsilon}(\eta-m)^2\right) f_0(\eta) d\eta\right)} \xrightarrow[\epsilon \to 0]{} \frac{f_\Lambda(m-\nu_\Lambda)}{f_0(m)}.
			\end{equation*}
			On the other hand, 
			$$
			\frac{1}{\sqrt{\epsilon}} \int_0^{+\infty} \exp \left(-\frac{1}{\epsilon} (\eta-m+\nu_\Lambda)^2 \right) f_\Lambda(\eta) d\eta \leq \|f_\Lambda\|_{L^\infty} \sqrt{\pi}
			$$
			and for $\epsilon >0$ small enough,
			$$
			\frac{1}{\sqrt{\epsilon}} \int_0^{+\infty} \exp \left(-\frac{1}{\epsilon} (\eta-m)^2 \right) f_0(\eta) d\eta \geq \frac{1}{2} f_0(m) \sqrt{\pi}.
			$$
			Thus
			$$
				\frac{\left(\displaystyle\int_0^{+\infty} \exp \left(-\frac{1}{\epsilon}(\eta-m +\nu_\Lambda)^2\right) f_\Lambda(\eta)d\eta \right)}{\left(\displaystyle\int_0^{+\infty} \exp \left(-\frac{1}{\epsilon}(\eta-m)^2\right) f_0(\eta) d\eta\right)}  d\mu_{0,\Lambda}(u) \leq \frac{2\|f_\Lambda\|_{L^\infty}}{f_0(m)} d\mu_{0,\Lambda}(u)
			$$
			which is integrable on $A$. The dominated convergence theorem yields \eqref{conv-eps}. The cylindrical projection property~\eqref{cyli-proj} follows from \eqref{conv-eps} and \eqref{cyli-proj-Lamb}.
		
		\medskip
		
		\noindent{\bf Step 2. Tightness.} We have found a sequence of measures $(\mu_{0,m, \Lambda})_{\Lambda \geq \lambda_1}$ on the finite dimensional spaces $(E_\Lambda)_{\Lambda \geq \lambda_1}$ that satisfies the cylindrical projection property \eqref{cyli-proj}. We will show that there exists a unique measure $\mu_{0,m}$ on an infinite dimensional Hilbert space satisfying 
		$$
		\mu_{0,m}\vert_{E_\Lambda} = \mu_{0,m,\Lambda}.
		$$
		By Skorokhod's criterion (see e.g., \cite[Lemma 1]{Skorokhod-74}), it is enough to prove that the sequence $(\mu_{0,m,\Lambda})_{\Lambda \geq \lambda_1}$ is tight in the sense that
		$$
		\lim_{R \to \infty} \sup_{\Lambda \geq \lambda_1} \mu_{0,m,\Lambda}(\{ u \in E_\Lambda : \|u\|_{\mathcal{H}^\theta} \geq R\}) =0
		$$
		for some $\theta <\frac{1}{2}-\frac{1}{s}$. The above tightness condition is satisfied if we can prove
		$$
		\int_{E_\Lambda} \|u\|^2_{\mathcal{H}^\theta} d\mu_{0,m,\Lambda}(u) \leq C(m, \theta), \quad \forall \Lambda \geq \lambda_1
		$$
		which is further reduced to showing
		\begin{align} \label{tigh-cond-prof}
			\int \lambda_j |\alpha_j|^2 d\mu_{0,m,\Lambda}(u) \leq C(m), \quad \forall \Lambda \geq \lambda_1, \quad \forall \lambda_1 \leq \lambda_j \leq \Lambda
		\end{align}
		for some constant $C(m)$ depending only on $m$. In fact, we have
		$$
		\begin{aligned}
			\int_{E_\Lambda} \|u\|^2_{\mathcal{H}^\theta} d\mu_{0,m,\Lambda}(u) &= \int \sum_{\lambda_j \leq \Lambda} \lambda_j^\theta|\alpha_j|^2 d\mu_{0,m, \Lambda}(u) \\
			&= \sum_{\lambda_j \leq \Lambda} \lambda_j^{\theta-1} \int \lambda_j |\alpha_j|^2 d\mu_{0,m,\Lambda}(u) \\
			&\leq C(m) \sum_{\lambda_j \leq \Lambda} \lambda_j^{-(1-\theta)} \\
			&\leq C(m) {\Tr}[h^{-(1-\theta)}] <\infty
		\end{aligned}
		$$
		provided~\cite[Example~3.2]{LewNamRou-15} that $\theta<\frac{1}{2}-\frac{1}{s}$. 
		
		To prove \eqref{tigh-cond-prof}, we first show that for all $\epsilon>0$ sufficiently small,
		\begin{align} \label{tigh-cond-prof-eps}
			\int \lambda_j |\alpha_j|^2 d\mu_{\epsilon, m, \Lambda}(u) \leq C(m), \quad \forall \Lambda \geq \lambda_1, \quad \forall \lambda_1 \leq \lambda_j \leq \Lambda
		\end{align}
		with some constant $C(m)$ depending only on $m$. We have
		$$	
		\begin{aligned}
			&\int \lambda_j|\alpha_j|^2  d\mu_{\epsilon,m,\Lambda}(u) \\
			&\quad = \frac{1}{z^r_{\epsilon,m}}\int  \lambda_j|\alpha_j|^2 \left(\int \exp\left(-\frac{1}{\epsilon} \left( \|P^\perp_\Lambda u\|_{L^2}^2 - m + \|P_\Lambda u\|_{L^2}^2\right)^2 \right) d\mu_{0,\Lambda}^\perp(u) \right) d\mu_{0,\Lambda}(u)\\\
			&\quad =\frac{1}{z^r_{\epsilon, m}} \int \lambda_j |\alpha_j|^2 \left(\int \exp\left(-\frac{1}{\epsilon} \left( \|P_{\ne j} u\|^2 - m + |\alpha_j|^2\right)^2 \right) d\mu_{0,\ne j}(u) \right) \frac{\lambda_j}{\pi} e^{-\lambda_j|\alpha_j|^2} d\alpha_j,
		\end{aligned}
		$$
		where 
		\[
		P_{\ne j} = \sum_{k\ne j} |u_k\rangle \langle u_k|, \quad d\mu_{0,\ne j}(u) = \prod_{k \ne j} \frac{\lambda_k}{\pi} e^{-\lambda_k|\alpha_k|^2} d\alpha_k.
		\]
		Denote $F_j$ the density function of $\|P_{\ne j}u\|^2$ with respect to $\mu_{0,\ne j}$. We rewrite
		\[
		\int \lambda_j|\alpha_j|^2  d\mu_{\epsilon,m,\Lambda}(u) = \frac{1}{z^r_{\epsilon,m}} \int \lambda_j|\alpha_j|^2  \left(\int_0^{+\infty} \exp\left(-\frac{1}{\epsilon} \left(  \eta- m +|\alpha_j|^2\right)^2 \right) F_j(\eta) d\eta\right) \frac{\lambda_j}{\pi} e^{-\lambda_j|\alpha_j|^2} d\alpha_j. 
		\]
		To proceed further, we denote by $\Phi_j$ the characteristic function of $\|P_{\ne j} u\|^2$ with respect to $\mu_{0,\ne j}$. We compute
		$$
		\begin{aligned}
			\Phi_j(s) &= \E_{\mu_{0,\ne j}}[e^{is \|P_{\ne j}u\|^2}] \\
			&=\int e^{is \sum_{k \ne j} |\alpha_k|^2} \prod_{l\ne j} \frac{\lambda_l}{\pi} e^{-\lambda_l |\alpha_l|^2} d\alpha_l \\
			&= \prod_{k \ne j} \Big(\int_{\C} e^{-(1-is \lambda_k^{-1})\lambda_k |\alpha_k|^2} d\alpha_k \Big) \Big(\prod_{\lambda_l \ne \lambda_k \atop l, k \ne j} \int_{\C} \frac{\lambda_l}{\pi} e^{-\lambda_l|\alpha_l|^2} d\alpha_l \Big) \\
			&= \prod_{k\ne j} \frac{1}{1-is\lambda_k^{-1}}.
		\end{aligned}
		$$
		As in the proof of Lemma~\ref{lem-f-Lambda} we bound the norm of this product by the norm of a product of any two terms. Taking the first two terms, we obtain that for all $\lambda_j\geq \lambda_1$,
		\[
		|\Phi_j(s)| \leq \frac{1}{1+s^2 \lambda_3^{-2}}.
		\]
		In particular, $\|\Phi_j\|_{L^1(\R)} \leq C\lambda_3$ for all $\lambda_j\geq \lambda_1$. Thus $F_j$ is bounded (uniformly in $j$) and uniformly continuous for all $\lambda_j\geq \lambda_1$. 
		
		Since 
		$$
		\frac{1}{\sqrt{\epsilon}} z^r_{\epsilon,m} \xrightarrow[\epsilon \to 0]{} f_0(m) \sqrt{\pi}
		$$
		and $f_0(m)>0$, we have for $\epsilon>0$ sufficiently small,
		$$
		\frac{1}{\sqrt{\epsilon}} z^r_{\epsilon,m} \geq \frac{1}{2}f_0(m) \sqrt{\pi}.
		$$
		We also have
		$$
		\begin{aligned}
			\frac{1}{\sqrt{\epsilon}} \int_0^{+\infty} \exp\left(-\frac{1}{\epsilon}(\eta-m+|\alpha_j|^2)^2 \right) F_j(\eta) d\eta &= \int_{\frac{1}{\sqrt{\epsilon}}(|\alpha_j|^2-m)}^{+\infty} e^{-\eta^2} F_j(\sqrt{\epsilon} \eta + m - |\alpha_j|^2) d\eta \\
			&\leq  \|F_j\|_{L^\infty} \sqrt{\pi} \\ 
			&\leq C\|\Phi_j\|_{L^1(\R)} \\
			&\leq C \lambda_3, \quad \forall \lambda_j\geq \lambda_1.
		\end{aligned}
		$$
		It follows that 
		$$
		\begin{aligned}
			\int \lambda_j|\alpha_j|^2  d\mu_{\epsilon, m, \Lambda}(u) &\leq \frac{2C\lambda_3}{f_0(m)} \int \lambda_j|\alpha_j|^2 \frac{\lambda_j}{\pi}e^{-\lambda_j|\alpha_j|^2} d\alpha_j \\
			&=\frac{2C\lambda_3}{f_0(m)} \int_0^\infty \lambda e^{-\lambda} d\lambda \\
			&= C(m)
		\end{aligned}
		$$
		for all $\lambda_j\geq \lambda_1$ and all $\epsilon>0$ sufficiently small. This proves \eqref{tigh-cond-prof-eps}.
		
		We are now able to prove \eqref{tigh-cond-prof}. By the layer cake representation, the problem is reduced to showing that 
		$$
		\int_0^{+\infty} \mu_{0,m, \Lambda}( \lambda_j|\alpha_j|^2 >\lambda) d\lambda = \lim_{\epsilon \rightarrow 0} \int_0^{+\infty} \mu_{\epsilon, m, \Lambda}(\lambda_j|\alpha_j|^2 >\lambda) d\lambda.
		$$
		Since $\mu_{\epsilon, m, \Lambda}(\lambda_j|\alpha_j|^2 >\lambda) \rightarrow \mu_{0,m, \Lambda}( \lambda_j|\alpha_j|^2 >\lambda)$ as $\epsilon \to 0$, \eqref{tigh-cond-prof} follows from the dominated convergence theorem and the fact that
		$$
		\begin{aligned}
			&\mu_{\epsilon,m, \Lambda}(\lambda_j|\alpha_j|^2>\lambda) \\
			&\quad = \int \mathds{1}_{\{\lambda_j|\alpha_j|^2>\lambda\}} d\mu_{\epsilon,m, \Lambda}(u) \\
			&\quad = \frac{1}{z^r_{\epsilon,m}}\int \mathds{1}_{\{\lambda_j|\alpha_j|^2>\lambda\}} \left( \int \exp \left(-\frac{1}{\epsilon}\left(\|P^\perp_\Lambda u\|_{L^2}^2 -m +\|P_\Lambda u\|_{L^2}^2 \right)^2 \right) d\mu_{0,\Lambda}^\perp(u)\right) d\mu_{0,\Lambda}(u) \\
			&\quad =\frac{1}{z^r_{\epsilon,m}}\int \mathds{1}_{\{\lambda_j|\alpha_j|^2>\lambda\}} \left(\int \exp \left( -\frac{1}{\epsilon} \left(\|P_{\ne j} u\|^2-m +|\alpha_j|^2 \right)^2\right) d\mu_{0,\ne j}(u)\right) \frac{\lambda_j}{\pi} e^{-\lambda_j|\alpha_j|^2} d\alpha_j \\
			&\quad = \frac{1}{z^r_{\epsilon,m}}\int \mathds{1}_{\{\lambda_j|\alpha_j|^2>\lambda\}} \left( \int_0^{+\infty} \exp \left(-\frac{1}{\epsilon}\left(\eta-m+|\alpha_j|^2\right)^2 \right) F_j(\eta) d\eta\right)\frac{\lambda_j}{\pi} e^{-\lambda_j|\alpha_j|^2} d\alpha_j \\
			&\quad \leq \frac{2C\lambda_3}{f_0(m)}  \int \mathds{1}_{\{\lambda_j|\alpha_j|^2>\lambda\}} \frac{\lambda_j}{\pi}e^{-\lambda_j|\alpha_j|^2}d\alpha_j \\
			&\quad \leq \frac{2C\lambda_3}{f_0(m)} \int^{+\infty}_{\lambda} e^{-\tau}d\tau \\
			&\quad = \frac{2C\lambda_3}{f_0(m)} e^{-\lambda}
		\end{aligned}
		$$
		which is integrable on $(0,+\infty)$. This proves \eqref{tigh-cond-prof} using \eqref{tigh-cond-prof-eps}, hence the tightness condition. 
		
		\medskip
		
		\noindent\textbf{Step 3. The limit measure lives on a $L^2$-sphere.} There remains to prove that the measure $\mu_{0,m}$ is concentrated on the sphere $\{ \|u\|^2= m\}$. 
		
		We need to prove that 
		$$ \int \varphi\left(\norm{u}^2\right)d\mu_{0,m}(u) = \varphi (m)$$
		for any bounded continuous function $\varphi:\R^+\mapsto \R$. Starting from the definition of $\mu_{0,m}$ as the $\Lambda \to \infty$ limit of~\eqref{eq:def mass measure} this is reduced to the claim 
		\begin{equation}\label{eq:fixed L2 mass}
		I_\Lambda(\varphi):= f_0(m) \int \varphi\left(\|P_\Lambda u\|_{L^2}^2\right)d\mu_{0,m,\Lambda}(u) =  \int \varphi\left(\|P_\Lambda u\|_{L^2}^2\right) f_\Lambda (m-\|P_\Lambda u\|_{L^2}^2) d\mu_{0,\Lambda}(u) \xrightarrow[\Lambda \to \infty]{} f_0 (m) \varphi (m).
		\end{equation}
		
		Denoting $g_\Lambda$ the density function of $\|P_\Lambda u\|_{L^2}^2$ with respect to $\mu_{0,\Lambda}$, we have that 
		$$ I_\Lambda(\varphi) = \int_0^{+\infty} g_\Lambda (\eta) \varphi (\eta) f_\Lambda (m-\eta)d\eta.$$
		We write
		$$
		\begin{aligned}
			\int_0^{+\infty} g_\Lambda(\eta) \varphi(\eta) f_\Lambda(m-\eta) d\eta - f_0(m) \varphi(m) &= \int_0^{+\infty} f_\Lambda(m-\eta)(g_\Lambda(\eta)-f_0(\eta)) \varphi(\eta) d\eta \\
			&\quad + \int_0^{+\infty} f_\Lambda(m-\eta) f_0(\eta)\varphi(\eta) d\eta - f_0(m)\varphi(m) \\
			&= (\text{I}) + (\text{II}).
		\end{aligned}
		$$
		Since $f_\Lambda \to \delta_0$, we have $(\text{II}) \to 0$ as $\Lambda \to \infty$. We write for some $M>0$ to be chosen shortly,
		$$
		(\text{I}) = \int_0^M f_\Lambda(m-\eta) (g_\Lambda(\eta)-f_0(\eta)) \varphi(\eta) d\eta + \int_M^{+\infty} f_\Lambda(m-\eta) (g_\Lambda(\eta)-f_0(\eta)) \varphi(\eta)d\eta
		$$
		Using the boundedness of $g_\Lambda, f_0$ and $\varphi$, the same argument as in the proof of \eqref{eq:lim f Lambda} yields the smallness of the second term when $\Lambda \to \infty$ for any $M>0$. We write
		$$
		\begin{aligned}
		\int_0^M f_\Lambda(m-\eta) (g_\Lambda(\eta)-f_0(\eta))\varphi(\eta) d\eta &= \int_0^M f_\Lambda(m-\eta) (g_\Lambda(\eta)-g_\Lambda(m)) \varphi(\eta) d\eta \\
		&\quad + (g_\Lambda(m)-f_0(m))\int_0^M f_\Lambda(m-\eta) \varphi(\eta) d\eta \\
		&\quad + \int_0^M f_\Lambda(m-\eta)(f_0(m)-f_0(\eta))\varphi(\eta) d\eta
		\end{aligned}
		$$
		Using the continuity of $g_\Lambda$ and $f_0$, the first and third terms in the right hand side are small provided $M$ is taken sufficiently small. The second term is also small since $g_\Lambda$ converges to $f_0$ pointwise as $\Lambda \to \infty$ and $\displaystyle \int_0^{+\infty} f_\Lambda(\eta) d\eta =1$. The claim~\eqref{eq:fixed L2 mass} follows.
	\end{proof}
	
	We may now complete the 
	\begin{proof}[Proof of Proposition~\ref{pro:fixed mass meas}]
	The claim \eqref{limi-eps-parti} follows from \eqref{parti-eps} and \eqref{limi-eps-parti-proof}. Hence there remains to prove the trace-class convergence of the $k$-particle density matrices~\eqref{limi-eps}.  
	
	We recall the fact~\cite[Addendum H]{Simon-79} that strong trace-class convergence is equivalent to weak-$\star$ trace-class convergence plus convergence of the trace-class norm. Since both sides of~\eqref{limi-eps} are non negative operators, convergence of the trace-class norm is reduced to convergence of the trace. Hence the sought-after strong trace-class convergence is equivalent to the convergence of~\eqref{limi-eps} tested against any bounded operator $K$ (in particular the identity)  
		\begin{equation}\label{eq:limi eps K}
\int  \langle u^{\otimes k}, K  u^{\otimes k} \rangle d\mu_{\epsilon,m}(u) \xrightarrow[\epsilon \to 0]{} \int \langle u^{\otimes k}, K  u^{\otimes k} \rangle d\mu_{0,m}(u). 	
		\end{equation}
	We take some $M>0$ and write  
	\begin{align*}
	\begin{aligned}
	 \Big|\int  &\langle u^{\otimes k}, K  u^{\otimes k} \rangle d\mu_{\epsilon,m}(u) - \int \langle u^{\otimes k}, K  u^{\otimes k} \rangle d\mu_{0,m}(u)\Big| \\ 
	 &\leq \left|\int_{\norm{u}^2 \leq M}  \langle u^{\otimes k}, K  u^{\otimes k} \rangle \left(d\mu_{\epsilon,m}(u)  - d\mu_{0,m}(u)\right) \right|
	 + \left|\int_{\norm{u}^2 > M}  \langle u^{\otimes k}, K  u^{\otimes k} \rangle \left(d\mu_{\epsilon,m}(u) - d\mu_{0,m}(u)\right) \right| \\
	 &\leq o_{\eps}(1) + o_M (1)
	\end{aligned}
	\end{align*}
    where $o_{\eps}(1)\to$ when $\eps \to 0$ as per Proposition~\ref{theo-fix-mass-meas} and $o_M(1)\to 0$ when $M\to \infty$, independently of $\eps$, as we explain below. We then let $\eps \to 0$ first and $M\to \infty$ next to obtain~\eqref{eq:limi eps K}. 
    
    We have just used that 
		$$\int_{\norm{u}^2 > M}  \langle u^{\otimes k}, K  u^{\otimes k} \rangle d\mu_{\epsilon,m}(u) \to 0$$
		when $M\to \infty$, uniformly in $\eps$. Indeed, for $M$ large enough,
		\begin{align*}
		 \left|\int_{\norm{u}^2 > M}  \langle u^{\otimes k}, K  u^{\otimes k} \rangle d\mu_{\epsilon,m}(u) \right| &\leq \norm{K} \int_{\norm{u}^2 > M}  \norm{u}^{2k} d\mu_{\epsilon,m}(u) \\
		 &= \frac{\norm{K}}{z_{\epsilon,m}^r} \int_M ^\infty \eta^{2k} f_0 (\eta) \exp\left(-\frac{1}{\epsilon}(\eta - m) ^2\right)d\eta\\
		 &\leq  \frac{C_{K,m}}{\sqrt{\epsilon}} \int_M ^\infty \eta^{2k}  \exp\left(-\frac{1}{\epsilon}(\eta - m) ^2\right)d\eta\\
		 &\leq C_{K,m} \int_{M-m} ^\infty \left(\sqrt{\eps}s + m \right) ^{2k} e^{-s^2} ds \\
		 &\leq C_{K,m,k} e^{-c M^2}  
		\end{align*}
where we return to~\eqref{limi-eps-parti} to bound $z_{\epsilon,m}^r$ from below and use the boundedness (independently of $\Lambda$) of $f_0$ (see \eqref{eq:dens func bound}).  
\end{proof}

We add a lemma that in essence shows that the limits $\eps \to 0$ of~\eqref{eq:def mu epsilon} and $\Lambda \to \infty$ of~\eqref{eq:def mass measure} can be interchanged.
	
\begin{lemma}[\textbf{Measures restricted to finite-dimensional spheres}]\label{lem:meas sphere}\mbox{}\\
Let $\Lambda \geq \lambda_1$ and define a probability measure on the sphere $S_mE_\Lambda:= \left\{ u \in E_\Lambda : \norm{u} ^2 = m \right\}$ in the manner 
\begin{equation}\label{eq:finite meas sphere}
d\sigma_{m,\Lambda} (u) = \frac{1}{z_{m,\Lambda}} \left(\prod_{\lambda_j \leq \Lambda} \frac{\lambda_j}{\pi} \exp\left( -\lambda_j |\alpha_j|^2\right)\right) d_m u 
\end{equation}
where $d_m u$ is the Lebesgue measure on the finite-dimensional sphere $S_mE_\Lambda$. When $m=1$, we simply denote $SE_\Lambda:=S_1E_\Lambda$ and $du:=d_1 u$.

For any uniformly bounded sequence of functions $F_\Lambda$ on $E_\Lambda $ we have that 
\begin{equation}\label{eq:limeps limLambda}
 \int  F_\Lambda (u) d\mu_{0,m,\Lambda} (u) - \int F_\Lambda (u) d\sigma_{m,\Lambda} (u) \xrightarrow[\Lambda \to \infty]{} 0.
\end{equation}
\end{lemma}
	
\begin{proof}
Observe that with the notation above,  
$$ z_{m,\Lambda} = g_\Lambda (m).$$
Hence, in view of~\eqref{eq:def mass measure} and~\eqref{eq:lim g Lambda} we have to prove that the difference between 
$$ 
\mathrm{I} = \int F_\Lambda (\alpha_1, \ldots, \alpha_N) f_\Lambda \left(m - \sum_{_j= 1} ^N |\alpha_j| ^2 \right) d\mu_{0,\Lambda} (\alpha_1, \ldots, \alpha_N)
$$
and 
$$ 
 \mathrm{II} = \int F_\Lambda (\alpha_1, \ldots, \alpha_N) \prod_{j= 1} ^N \frac{\lambda_j}{\pi} \exp\left( -\lambda_j \left|\alpha_j \right|^2\right) \Id_{|\alpha_1| ^2 = m - \sum_{j=2} ^N |\alpha_j|^2} d\theta_1 d\alpha_2 \ldots d\alpha_N
$$
goes to $0$ when $\Lambda \to \infty$. Here $N$ is the number of eigenvalue of $h$ below the threshold $\Lambda$, $\alpha_j = \langle u_j, u\rangle$ identified with a vector in the plane, whose angle is denoted $\theta_j$. 

Reorganizing the integration slightly we have that 
\begin{equation}\label{eq:II}
 \mathrm{II} = \frac{\lambda_1}{\pi} \exp\left(-\lambda_1 m \right) \int \prod_{j= 2} ^N \frac{\lambda_j}{\pi} \exp\left( -\left(\lambda_j -\lambda_1 \right)\left|\alpha_j \right|^2\right) \left(\int F_\Lambda \left(m-\sum_{j=2} ^N |\alpha_j|^2, \theta_1,\alpha_2, \ldots, \alpha_N\right) d\theta_1\right) d\alpha_2 \ldots d\alpha_N
\end{equation}
with $F_\Lambda \left(|\alpha_1|^2, \theta_1,\alpha_2, \ldots, \alpha_N\right) \equiv F_\Lambda (\alpha_1, \ldots, \alpha_N) $. On the other hand 
\begin{align*}
\mathrm{I} &=  \int \prod_{j= 2} ^N \frac{\lambda_j}{\pi} \exp\left( -\lambda_j \left|\alpha_j \right|^2\right) \left(\int \frac{\lambda_1}{\pi} \exp\left( -\lambda_1 \left|\alpha_1\right|^2\right) f_{\Lambda} \left( m - \sum_{j=1}^N |\alpha_j|^2\right) F_\Lambda \left(\alpha_1, \ldots, \alpha_N\right) d\alpha_1 \right) d\alpha_2 \ldots d\alpha_N\\
\end{align*}
Changing variables 
\begin{multline*}
\int \frac{\lambda_1}{\pi} \exp\left( -\lambda_1 \left|\alpha_1\right|^2\right) f_{\Lambda} \left( m - \sum_{j=1}^N |\alpha_j|^2\right) F_\Lambda \left(\alpha_1, \ldots, \alpha_N\right) d\alpha_1 \\ 
= \prod_{j= 2} ^N \exp\left( \lambda_1 |\alpha_j|^2 \right) \int \frac{\lambda_1}{\pi} \exp\left( -\lambda_1 \left|\alpha_1\right|^2\right) f_{\Lambda} \left( m - |\alpha_1|^2\right) F_\Lambda \left(|\alpha_1| ^2 - \sum_{j=2}^N |\alpha_j| ^2, \theta_1, \alpha_2, \ldots, \alpha_N\right) d\alpha_1
\end{multline*}
so that 
\begin{multline}\label{eq:I}
\mathrm{I} = \int \prod_{j= 2} ^N \frac{\lambda_j}{\pi} \exp\left( -\left(\lambda_j -\lambda_1 \right)\left|\alpha_j \right|^2\right) \\ 
\left(\int \frac{\lambda_1}{\pi} \exp\left( -\lambda_1 \left|\alpha_1\right|^2\right) f_{\Lambda} \left( m - |\alpha_1|^2\right) F_\Lambda \left(|\alpha_1| ^2 - \sum_{j=2}^N |\alpha_j| ^2, \theta_1, \alpha_2, \ldots, \alpha_N\right) d\alpha_1 \right) d\alpha_2 \ldots d\alpha_N.
\end{multline}
Using~\eqref{eq:lim f Lambda} the integrand of the $d\alpha_2,\ldots, d\alpha_N$ integral in~\eqref{eq:I} converges pointwise to that in~\eqref{eq:II}, and we can conclude using dominated convergence. 
\end{proof}

\subsection{Dependence on the mass}\label{sec:dep mass}

As hinted at in the introduction, our trial state linking the quantum and classical problems will involve measures in finite-dimensional subspaces (the dimension being ultimately sent to infinity), with shifted mass constraints. Controling the error thus made requires some input on the classical field theory side, that we now provide. 

It is more convenient to start from the finite-dimensional measures restricted to spheres from Lemma~\ref{lem:fin dim}, and to rescale them so that the mass dependence translates to a prefactor affecting the covariance operator. To this end let $\Lambda > 0$ be a kinetic energy cut-off, $m>0$ a mass, $g\geq 0$ a coupling constant, and let us define probability measures on $SE_\Lambda$ - the $L^2$-unit sphere of $E_\Lambda$ in the manner 
\begin{equation}\label{eq:dep mass meas}
 d\rho_{m,g} (u) = \frac{1}{z_{m,g}} \exp\left( -m \left\langle u, h\, u \right\rangle + \frac{mg}{2} \left\langle u^{\otimes 2}, w\, u^{\otimes 2} \right\rangle\right) du.
\end{equation}
We will identify $d\rho_{m,g} (u)$ with its density with respect to the normalized Lebesgue measure $du$ on the sphere. In this notation, the free measure $\rho_{m,0}$ is just $\sigma_{m,\Lambda}$ from Lemma~\ref{lem:fin dim}, rescaled to live on the unit sphere. The dependence on $\Lambda$ will be implicit in this subsection.

The dependence on $m$ of the above will be pretty strong when $\Lambda \to \infty$ (unsuprisingly, since we perturb the measure's covariance), so that it is important in our approach to deal with relative quantities involving differences between related measures.   

The estimate we use in the sequel involves the kinetic energy as follows

\begin{proposition}[\textbf{Varying the mass in the classical field theory}]\label{pro:dep mass}\mbox{}\\
Let $g>0$ be a fixed constant, $0 < m_1 \leq m_2$. With the notation above we have that 
\begin{equation}\label{eq:dep mass}
\Delta_{m_1} ^{m_2}:= \int \left\langle u, h\, u \right\rangle \left(\rho_{m_1,0} (u)- \rho_{m_2,g} (u)\right) du \leq C_{g,m_1,m_2} \max \left(\sqrt{d}, d |m_1-m_2| \right) 
\end{equation}
where $d = \dim E_\Lambda \leq C \Lambda ^{1/2 + 1/s}$ and $C_{g,m_1,m_2}$ stays uniformly bounded for bounded values of $g,m_1,m_2$ and $m_1^{-1},m_2^{-1}$.
\end{proposition}

Note that if $\Delta_{m_1} ^{m_2} \leq 0$ the statement is empty, so that we are free to, and shall, assume that $\Delta_{m_1} ^{m_2} \geq 0$ in all this subsection. Our main inequality leading to Proposition~\ref{pro:dep mass} is the following:

\begin{proposition}[\textbf{From relative entropy to kinetic energy}]\label{pro:use rel ent}\mbox{}\\
Let $d\rho_{m,0} (u)$ be as above and $\rho$ be any other probability measure on the unit $L^2$-sphere $SE_\Lambda$. Then
\begin{equation}\label{eq:use rel ent}
\int \left\langle u, h\, u \right\rangle \left(d\rho_{m,0} (u)- d\rho (u)\right) \leq C_m \sqrt{d} \sqrt{\cH_{\rm cl} (\rho,\rho_{m,0})} 
\end{equation}
with $\cH_{\rm cl}$ the classical relative entropy~\eqref{eq:class rel ent} and $C_m$ a constant depending only on $m$.
\end{proposition}

Note that Pinsker's inequality (see \cite{CarLie-14} and \cite[Section 6]{LewNamRou-20}) gives (identifying the measures with their density with respect to the Lebesgue measure on the unit sphere $SE_\Lambda$)
$$ \int \left|\rho_{m,0} (u)- \rho (u)\right|du \leq C  \sqrt{\cH_{\rm cl} (\rho,\rho_{m,0})} $$
leading, by the definition~\eqref{eq:subspaces} of $E_\Lambda$, to 
$$\int \left\langle u, h\, u \right\rangle \left(d\rho_{m,0} (u)- d\rho (u)\right) \leq C \Lambda \sqrt{\cH_{\rm cl} (\rho,\rho_{m,0})}.$$ 
Since a Cwikel-Lieb-Rozenbljum estimate such as~\cite[Lemma D1]{DinRou-23} yields
$$ d = \dim E_\Lambda \leq C \Lambda ^{1/2 + 1/s}$$
we see that Proposition~\ref{pro:use rel ent} is a net improvement over Pinsker's inequality as soon as $s>2$, which will be crucial in Section~\ref{sec:control fluct} below

Our proof of Proposition~\ref{pro:use rel ent} uses crucially the following estimates on correlations in a Gibbs measure conditioned on the mass:
%

	\begin{lemma}[\textbf{Correlations in a canonical classical ensemble}]\label{lem-Suto-04b-clas}\mbox{}\\
		Let $$
	\gamma_{\rho_{m,0}}^{(1)} = \int |u\rangle \langle u| \rho_{m,0} (u) du, \quad \gamma_{\rho_{m,0}}^{(2)} = \int |u^{\otimes 2}\rangle \langle u^{\otimes 2}| \rho_{m,0} (u) du.
	$$
		We have, for any $j\neq k$, 
		\begin{equation}\label{eq:gauss corr}
		\langle u_j \otimes u_k |\gamma^{(2)}_{\rho_{m,0}}| u_j \otimes u_k\rangle \leq \langle u_j |\gamma^{(1)}_{\rho_{m,0}}|u_j\rangle \langle u_k |\gamma^{(1)}_{\rho_{m,0}}| u_k\rangle
		\end{equation}
		hence, with $\alpha_j = \langle u, u_j \rangle$,
		$$
		\int |\alpha_j|^2 |\alpha_k|^2 \rho_{m,0}(u) du \leq \Big(\int |\alpha_j|^2 \rho_{m,0}(u) du\Big)\Big(\int |\alpha_k|^2 \rho_{m,0}(u) du\Big), \quad \forall j \ne k.
		$$
		In particular
	\begin{align}  \label{est-diff-3}
	{\Tr} \left[ h ^{\otimes 2} \left( \gamma_{\rho_{m,0}}^{(2)} - (\gamma_{\rho_{m,0}}^{(1)})^{\otimes 2}\right)\right] \leq C \dim(E_\Lambda).
	\end{align}
		
	\end{lemma}
	
	Without the conditioning on the sphere, the relationship between $\gamma_{\rho_{m,0}}^{(2)}$ and $\gamma_{\rho_{m,0}}^{(1)}$ would be given by Wick's theorem. The above lemma is a convenient replacement in the canonical case. It might be known to experts but we could not find a reference. Since the quantum analogue is contained in~\cite{Suto-04b} we derive the estimate by semiclassical means, but a direct combinatorial proof is probably feasible.

	\begin{proof}
		Define a bosonic canonical state on $(E_\Lambda)^{\otimes_s M}$ by
		$$
		\Gamma_M = \frac{1}{Z_M} \exp\Big(- \frac{m}{M} \sum_{j=1}^M h_j\Big)
		$$
		and denote for $1\leq k \leq M$,
		$$
		\Gamma^{(k)}_M = \binom{M}{k} {\Tr}_{k+1 \to M}[\Gamma_M].
		$$
		From~\cite[Theorem, Item (ii)]{Suto-04b} we know that
		$$
		{\Tr}\left[\Nc_j \Nc_k \Gamma_M\right] \leq {\Tr}\left[\Nc_j \Gamma_M\right] {\Tr}\left[\Nc_k \Gamma_M\right], \quad j \ne k,
		$$
		where 
		$$\Nc_j = \ada_j a_j = \ada(u_j) a(u_j)$$
		in terms of the annihilation and creation operators $\ada(u_j), a(u_j)$. Since these operators commute for $j \ne k$, the above inequality is equivalent to
		\begin{align} \label{Suto-04b-est}
		2 \langle u_j \otimes u_k |\Gamma^{(2)}_M |u_j \otimes u_k \rangle \leq \langle u_j|\Gamma^{(1)}_M|u_j \rangle \langle u_k |\Gamma^{(1)}_M|u_k\rangle,
		\end{align}
		where we have used that
		$$
		k!\langle u_1 \otimes ... \otimes u_k | \Gamma^{(k)}_M | v_1 \otimes ... \otimes v_k\rangle = {\Tr}\left[\ada(v_1)...\ada(v_k) a(u_1)...a(u_k)\Gamma_M\right].
		$$
		On the other hand, it is known (see for example~\cite[Appendix~B]{Rougerie-LMU} and references therein) that
		$$
		\binom{M}{k}^{-1} \Gamma^{(k)}_M \xrightarrow[M\to \infty]{} \gamma^{(k)}_{\rho_{m,0}}, \quad \forall k \geq 1.
		$$
		In particular,
		\begin{align} \label{limi-Gam-M}
		\begin{aligned}
		\binom{M}{2}^{-1} \langle u_j \otimes u_k |\Gamma^{(2)}_M| u_j \otimes u_k\rangle &- \binom{M}{1}^{-2} \langle u_j|\Gamma^{(1)}_M|u_j \rangle \langle u_k |\Gamma^{(1)}_M| u_k\rangle \\
		 &\xrightarrow[M\to \infty]{}  \langle u_j \otimes u_k |\gamma^{(2)}_{\rho_{m,0}}|u_j \otimes u_k\rangle - \langle u_j|\gamma^{(1)}_M|u_j\rangle \langle u_k|\Gamma^{(1)}_M|u_k\rangle.
		\end{aligned}
		\end{align}
		Using \eqref{Suto-04b-est}, we have
		$$
		\begin{aligned}
		\text{LHS of } \eqref{limi-Gam-M}&\leq \left(\frac{1}{2} \binom{M}{2}^{-1} - \binom{M}{1}^{-2}\right) \langle u_j |\Gamma^{(1)}_M| u_j\rangle \langle u_k|\Gamma^{(1)}_M|u_k\rangle \\
		&\leq \frac{1}{M^2(M-1)} \langle u_j |\Gamma^{(1)}_M| u_j\rangle \langle u_k|\Gamma^{(1)}_M|u_k\rangle \\
		&\leq \frac{1}{M^2(M-1)} \|u_j\|^2 \|u_k\|^2\|\Gamma^{(1)}_M\|^2_{\mathfrak S^1} \\
		& = \frac{1}{M-1}\xrightarrow[M\to \infty]{}  0.
		\end{aligned}
		$$
		In particular, the limit in the right hand side of \eqref{limi-Gam-M} must be non-positive and~\eqref{eq:gauss corr} follows.
		
		Applying~\eqref{eq:gauss corr}, we have
	$$
	\begin{aligned}
		{\Tr}\Big[\Big( h^{\otimes 2} \Big(\gamma_{\rho_{m,0}}^{(2)} -(\gamma_{\rho_{m,0}}^{(1)})^{\otimes 2}\Big)\Big] &= \sum_{j,k = 1} ^d \lambda_j \lambda_k \left(\langle u_j \otimes u_k |\gamma^{(2)}_{\rho_{m,0}}| u_j \otimes u_k\rangle - \langle u_j |\gamma^{(1)}_{\rho_{m,0}}|u_j\rangle \langle u_k |\gamma^{(1)}_{\rho_{m,0}}| u_k\rangle \right)\\
		&\leq \sum_{j=1} ^d \lambda_j^2 \Big(\int |\alpha_j|^4 \rho_{m,0}(u) du - \Big(\int |\alpha_j|^2 \rho_{m,0} (u) du \Big)^2 \Big) \\
		&\leq \sum_{j= 1} ^d \int \left( \lambda_j |\alpha_j|^2\right)^2 \rho_{m,0} (u) du \\
		&\leq C \dim E_\Lambda
	\end{aligned}
	$$
	for some constant $C>0$ independent of $n,T,\eta$. Here we have used that
	\begin{align} \label{boun-mu-eta}
	 \int \left( \lambda_j |\alpha_j|^2\right)^2 \rho_{m,0} (u) du \leq C
	\end{align}
	for some universal constant $C>0$. This is proven similarly to \eqref{tigh-cond-prof} so we omit the detail here.
	\end{proof}
	
	We now turn to the 
	
	\begin{proof}[Proof of Proposition~\ref{pro:use rel ent}]
	 For $\epsilon >0$ we perturb $\rho_{m,0}$ in the manner\footnote{This notation is strictly restriced to this proof, it should not generate confusion with the rest of the text.}
	$$
	\mu_\epsilon = \frac{1}{z_\epsilon} \exp \left(-m(1+\epsilon) \langle u, h u\rangle\right).
	$$
	This is the unique minimizer of 
	$$
	\mathcal F_\epsilon[\mu] = m(1+\epsilon) \int\langle u, h u \rangle \mu(u) du +  \int \mu(u) \log(\mu(u)) du
	$$
	over all probability measures $\mu$ absolutely continuous with respect to the Lebesgue measure on the sphere $SE_\Lambda$. We have that
	$$
	F_\epsilon = \inf \{ \mathcal F_\epsilon[\mu], \mu \mbox{ a probability measure on $SE_\Lambda$} \} = - \log (z_\epsilon).
	$$
	By the classical Gibbs variational principle and the non-negativity of the classical relative entropy, we find, for any probability measure $\mu$
	$$
	\begin{aligned}
	\mathcal H_{\cl}  (\mu, \mu_0) + \epsilon m \int \langle u, h u\rangle \mu(u) du &\geq  \mathcal H_{\cl} (\mu_\epsilon, \mu_0) + \epsilon m \int \langle u, h u\rangle \mu_\epsilon(u) du \\
	&\geq \epsilon m \int \langle u, h u\rangle \mu_\epsilon(u) du.
	\end{aligned}
	$$
	Subtracting $\epsilon m \int \langle u, h u\rangle \mu_0 (u) du$ from both sides leads to
	\begin{align} \label{F-eta}
	\epsilon m \int \langle u, h u\rangle (\mu(u)-\mu_0(u)) du \geq - \mathcal H_{\cl} (\mu, \mu_0) + \epsilon n \int \langle u, h u\rangle (\mu_\epsilon(u)-\mu_0(u)) du.
	\end{align}
	Using
	$$
	F_\eta = - \log(z_\eta) = -  \log \left(\int \exp \big( -m(1+\eta) \langle u, h u\rangle\big) du\right),
	$$
	we have
	$$
	\partial_\eta F_\eta = - \frac{\partial_\eta z_\eta}{z_\eta}, \quad \partial^2_\eta F_\eta = - \frac{\partial^2_\eta z_\eta}{z_\eta} +  \frac{(\partial_\eta z_\eta)^2}{z^2_\eta}.
	$$
	We compute
	$$
	\begin{aligned}
		\partial_\eta z_\eta &= -m \int \langle u, h u\rangle \exp \Big(-m(1+\eta) \langle u, h u\rangle\Big) du \\
		&= - m z_\eta \int \langle u, h u\rangle \mu_\eta(u) du. 
	\end{aligned}
	$$
	In particular, we have
	\begin{align} \label{deri-F-eta}
	\partial_\eta F_\eta = m \int \langle u, h u\rangle \mu_\eta(u) du.
	\end{align}
	We also have
	$$
	\begin{aligned}
		\partial^2_\eta z_\eta &= m^2 \int \langle u, h u\rangle^2 \exp \left(-m(1+\eta)\langle u, h u\rangle\right) du \\
		&= m^2 z_\eta  \int \langle u, h u\rangle^2 \mu_\eta(u) du. 
	\end{aligned}
	$$
	Thus
	$$
	\partial^2_\eta F_\eta = -  \left(\mathbb E_{\mu_\eta}\left[\left(m \langle u, h u\rangle\right)^2\right] - \left(\mathbb E_{\mu_\eta}\left[m\langle u, h u\rangle\right]\right)^2\right) \leq 0
	$$
	by Jensen's inequality.
	
	Thanks to \eqref{deri-F-eta}, we have
	\begin{align}\label{est-diff-1}
	\begin{aligned}
	\epsilon m \int \langle u, h u\rangle (\mu(u)-\mu_0(u)) du &\geq - \mathcal H_{\cl} (\mu, \mu_0) + \epsilon m \int \langle u, h u\rangle (\mu_\epsilon(u)-\mu_0(u)) du \\
	&= - \Hc_{\cl}(\mu,\mu_0) + \epsilon \left(\partial_\eta F_\eta \vert_{\eta = \epsilon} - \partial_\eta F_\eta\vert_{\eta=0} \right) \\
	&= -  \Hc_{\cl}(\mu,\mu_0) + \epsilon \int_0^\epsilon \partial^2_\eta F_\eta d\eta.
	\end{aligned}
	\end{align}
	We estimate
	\begin{align} \label{est-diff-2}
		\Big|\int_0^\epsilon \partial^2_\eta F_\eta d\eta\Big| & \leq \int_0^\epsilon |\partial^2_\eta F_\eta| d\eta \nonumber \\
		&\leq \epsilon m ^2 \sup_{\eta \in [0,\epsilon]} \left( \mathbb E_{\mu_\eta}\left[ \langle u, h u\rangle ^2\right] - \left(\mathbb E_{\mu_\eta}\left[\langle u, h u\rangle\right]\right)^2 \right) \nonumber\\
		& = \epsilon m^2 \sup_{\eta \in [0,\epsilon]} {\Tr} \Big[m^2 h^{\otimes 2} \Big( \gamma_{\mu_\eta}^{(2)} - (\gamma_{\mu_\eta}^{(1)})^{\otimes 2}\Big)\Big],
	\end{align}
	where
	$$
	\gamma_{\mu_\eta}^{(1)} = \int |u\rangle \langle u| \mu_\eta(u) du, \quad \gamma_{\mu_\eta}^{(2)} = \int |u^{\otimes 2}\rangle \langle u^{\otimes 2}| \mu_{\eta}(u) du.
	$$
%
	Applying Lemma \ref{lem-Suto-04b-clas}, we deduce
	\begin{align}  \label{est-diff-4}
	\sup_{\eta \in [0,\epsilon]} {\Tr} \left[ h^{\otimes 2} \left( \gamma_{\mu_\eta}^{(2)} - (\gamma_{\mu_\eta}^{(1)})^{\otimes 2}\right)\right] \leq C \dim(E_\Lambda).
	\end{align}
	Collecting \eqref{est-diff-1}, \eqref{est-diff-2}, and \eqref{est-diff-4}, we obtain
	$$
	\epsilon m \int \langle u, h u\rangle (\mu_0 (u)- \mu(u)) du \leq \Hc_{\cl} (\mu,\mu_0) + C \epsilon^2 \dim(E_\Lambda).
	$$
	There remains to optimize over $\epsilon$, i.e. set
	$$ \epsilon = C_m \sqrt{\frac{\Hc_{\cl} (\mu,\mu_0)}{\dim(E_\Lambda)}}$$
	and recall that, within the notation of this proof $\mu_0 = \rho_{m,0}$ to complete the proof.
	\end{proof}

We complete Proposition~\ref{pro:use rel ent} by a bound on the entropy of the interacting measure relative to the free one:

\begin{lemma}[\textbf{Relative entropy bound}]\label{lem:rel ent}\mbox{}\\
With the notation of Proposition~\ref{pro:dep mass} we have that  
$$ \Hc_{\cl} \left(\rho_{m_2,g},\rho_{m_1,0}\right) \leq C_g \left( |m_1-m_2| \Delta_{m_1} ^{m_2} +1\right),$$
assuming that $\Delta_{m_1} ^{m_2}$ (as defined in~\eqref{eq:dep mass}) is non-negative. 
\end{lemma}

\begin{proof}
By definition, 
$$
\begin{aligned} 
	\cH_{\rm cl} \left(\rho_{m_2,g},\rho_{m_1,0}\right) = \int \rho_{m_2,g}(u) \log \left( \rho_{m_2,g}(u)\right) du &- \int \rho_{m_1,0}(u) \log \left( \rho_{m_1,0}(u)\right) du \\
	&+ m_1 \int \left\langle u , h u\right\rangle \left( \rho_{m_2,g}(u) - \rho_{m_1,0}(u)\right) du.
\end{aligned}
$$
But, by the Gibbs variational principle defining $\rho_{m_2,g}$
\begin{align*}
\begin{aligned}
 \int \Big(m_2 \left\langle u , h u\right\rangle &+ \frac{gm_2}{2} \left\langle u ^{\otimes 2}, w u^{\otimes 2}  \right\rangle  \Big) \rho_{m_2,g} (u) du + \int \rho_{m_2,g}(u) \log \left( \rho_{m_2,g}(u)\right) du \\
&\leq \int \left(m_2 \left\langle u , h u\right\rangle + \frac{gm_2}{2} \left\langle u ^{\otimes 2}, w u^{\otimes 2}  \right\rangle  \right) \rho_{m_1,0} (u) du + \int \rho_{m_1,0}(u) \log \left( \rho_{m_1,0}(u)\right) du. 
\end{aligned}
\end{align*}
Hence  
\begin{align*}
\begin{aligned}
 \cH_{\rm cl} \left(\rho_{m_2,g},\rho_{m_1,0}\right) &\leq ( m_1 - m_2)  \int \left\langle u , h u\right\rangle \left( \rho_{m_2,g} (u)- \rho_{m_1,0}(u)\right) du \\
 &\quad + \frac{gm_2}{2} \int \left\langle u ^{\otimes 2}, w u^{\otimes 2}  \right\rangle   \rho_{m_1,0} (u) du - \frac{gm_2}{2} \int \left\langle u ^{\otimes 2}, w u^{\otimes 2}  \right\rangle   \rho_{m_2,g} (u) du.
 \end{aligned}
\end{align*}
By arguments mimicking those involved in the definition of the interacting Gibbs measure (see Appendix~\ref{sec:sub-inte-meas} and references therein) we find that the last two terms are bounded independently of $\Lambda$. The conclusion follows.
\end{proof}

Finally we can give the final step of the 

\begin{proof}[Proof of Proposition~\ref{pro:dep mass}]
Recall that we are free to assume that $\Delta_{m_1} ^{m_2}\geq 0$. Combining Proposition~\ref{pro:use rel ent} and Lemma~\ref{lem:rel ent} leads to 
$$\Delta_{m_1} ^{m_2} \leq C \sqrt{d} \sqrt{ |m_1-m_2| \Delta_{m_1} ^{m_2} +1}$$
hence 
$$ \Delta_{m_1} ^{m_2} \leq C \sqrt{ d|m_1-m_2| \Delta_{m_1} ^{m_2}} +C\sqrt{d}$$
with a generic notation $C$ for constants depending only on $m_1,m_2,g$. Writing this as 
\begin{equation}\label{eq:pouet}
 \sqrt{ \Delta_{m_1} ^{m_2}} \left( \sqrt{ \Delta_{m_1} ^{m_2}} - C \sqrt{ d|m_1-m_2| }\right) \leq C \sqrt{d}
\end{equation}
we see that either  
$$ \Delta_{m_1} ^{m_2}  \leq  2C^2 d |m_1-m_2|$$
or the parenthesis in the left-hand side of~\eqref{eq:pouet} is bounded below by $C \sqrt{ \Delta_{m_1} ^{m_2}}$ and hence 
$$ 
\Delta_{m_1} ^{m_2} \leq C\sqrt{d}.
$$
The proof is complete.
\end{proof}

	\section{From the free canonical state to the fixed mass Gaussian measure}\label{sec:free}
	
	We now consider the limit $N\propto T \to \infty$ of the free canonical Gibbs state, and relate it to the Gaussian measure conditioned on the $L^2$ mass:
	
	\begin{theorem}[\textbf{Limit of the free canonical state}]\label{thm:free case}\mbox{}\\
	Let $N=mT$, 
	$$ H_{N,0} := \sum_{j=1} ^N h_j$$
	acting on $\gh ^{\otimes_\sym N}=: \gh^N$ and the associated canonical Gibbs state
	\begin{align} \label{Gam-c-mT}
	\Gamma^c_{N,T,0} =\frac{1}{Z^c_{N,T,0}} \exp \left(-\frac{1}{T} H_{N,0}\right).
	\end{align}
	Define the associated reduced density matrices 
	$$ 
	(\Gamma^c_{N,T,0})^{(k)} := {N \choose k} \Tr _{k+1 \to N} \left[\Gamma^c_{N,T,0}\right].
	$$
	Let further $\mu_{0,m}$ be the fixed-mass gaussian measure from Definition~\ref{def:gauss meas mass} and Proposition~\ref{theo-fix-mass-meas}.
	
	In the limit $T\to \infty$ with $m>0$ fixed, $\Gamma^c_{mT,T,0} $ converges to $\mu_{0,m}$ in the sense that, for all $k\in \N$, 
	\begin{align} \label{limi-cano}
		k!\frac{(\Gamma^c_{mT,T,0})^{(k)}}{T^k} \xrightarrow[T\to\infty]{} \int |u^{\otimes k} \rangle \langle u^{\otimes k}| d\mu_{0,m}(u)
	\end{align}
	strongly in the trace-class on $\mathfrak {h}^k$. Hence $\mu_{0,m}$ is the de Finetti measure at scale $T^{-1}$ of the sequence $(\Gamma^c_{mT,T,0})_T$ in the sense of~\cite[Definition~4.1, Theorem~4.2]{LewNamRou-15}.
	\end{theorem}

	The above fixed mass Gaussian measure can be thought of formally as
	$$
	d\mu_{0,m}(u) \propto \exp (-\langle u|hu \rangle) \mathds{1}_{\left\{\int |u|^2=m \right\}} du
	$$
	normalized as a probability. We have seen in Proposition~\ref{pro:fixed mass meas} that $\mu_{0,m}$ is the limit, as $\epsilon \to 0$, of the probability measure 
	$$
	d\mu_{\epsilon,m}(u) = \frac{1}{z_{\epsilon,m}^r} \exp\left(-\frac{1}{\epsilon} \left(\int |u|^2 -m \right)^2 \right) d\mu_0 (u).
	$$
	For fixed $\epsilon >0$, it follows from the main results of~\cite{LewNamRou-20,FroKnoSchSoh-16} that $\mu_{\epsilon,m}$ is a limit, as $T \to \infty$, of the Gibbs state
	\begin{equation}\label{eq:relax Gibbs}
	\Gamma_{\epsilon,m,T} =\frac{1}{Z_{\epsilon,m,T}} \exp \left( -\frac{1}{T} \left(d\Gamma(h) +\frac{T}{\epsilon}\left(\frac{\mathcal{N}}{T}-m \right)^2\right)\right),
	\end{equation}
	where
	$$
	d\Gamma(h) = \bigoplus_{N\geq 0} \sum_{j=1}^N h_{\bx_j}, \quad	\left(\frac{\mathcal{N}}{T}-m \right)^2 = \bigoplus_{N\geq 0} \left(\frac{N}{T}-m\right)^2 \mathds{1}_{\mathfrak h^N}.
	$$
	Here $\mathfrak {h}^N := \mathfrak {h}^{\otimes_s N}$ is the symmetric tensor product of $N$ copies of $\mathfrak {h}$. 
	
	On the other hand, if we take the limit $\epsilon \to 0$ in the first place, the Gibbs state $\Gamma_{\epsilon,m,T}$ clearly converges to the canonical Gibbs state $\Gamma^c_{mT,T}$ we are interested in. Our strategy of proof for Theorem~\ref{thm:free case} is therefore to commute the $T\to \infty$ and $\epsilon \to 0$ limits, which we obtain by investigating the uniformity of the aforementioned convergences, to allow convergence along a sequence $\epsilon (T) \to 0$ when $T \to \infty$ (respectively a sequence $T(\epsilon) \to \infty$ when $\epsilon\to 0$).

	\begin{figure}
		\centering
		\includegraphics[width=0.7\linewidth]{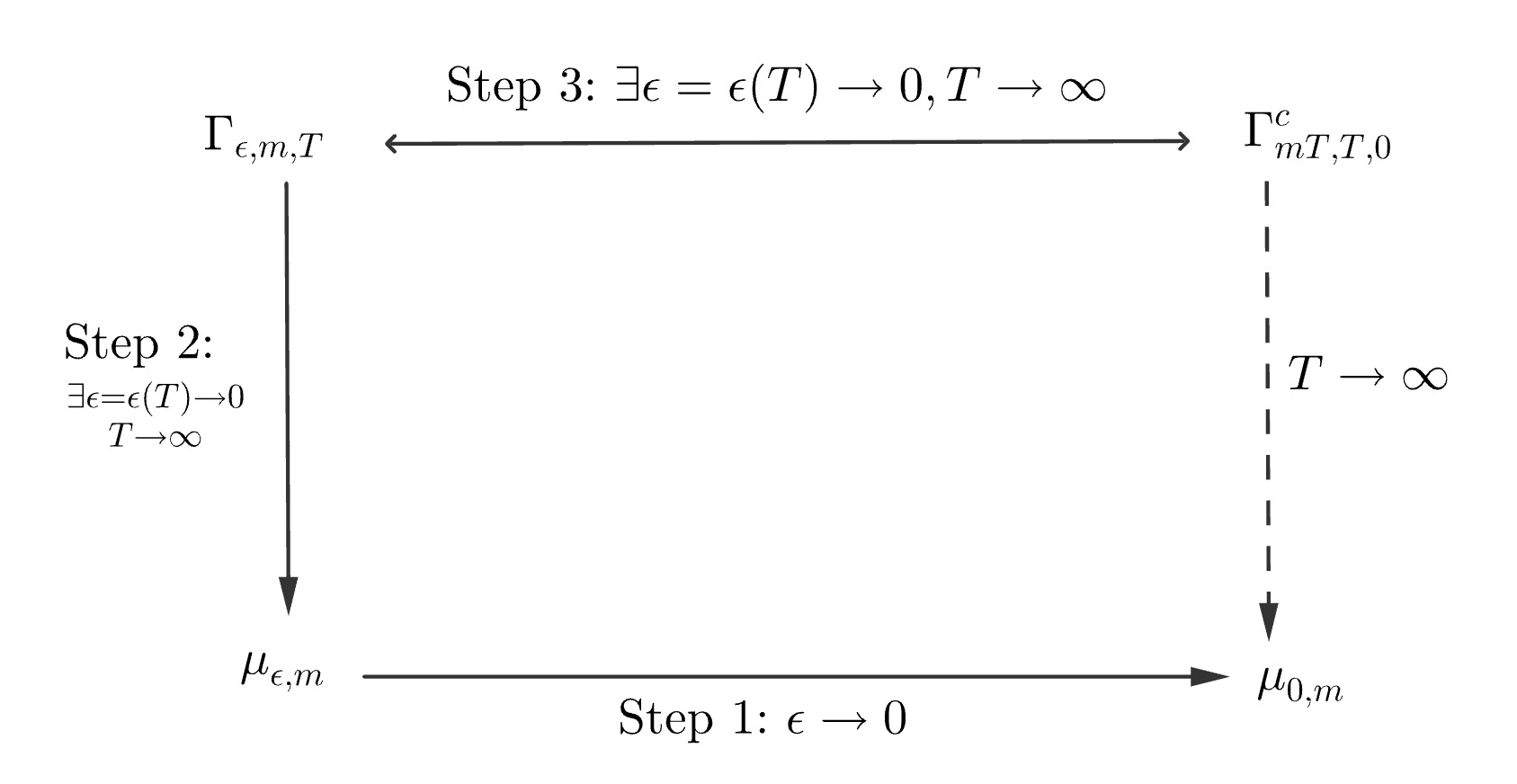}
		\caption{Derivation of the fixed mass Gaussian measure}
		\label{fig:diagram}
	\end{figure}

	We summarize this scheme of proof in Figure~\ref{fig:diagram}. The dashed arrow on the right is the content of Theorem~\ref{thm:free case}, which we obtain by following successively the three solid arrows (Steps 1, 2 and 3). Proposition~\ref{pro:fixed mass meas} provides Step 1 of this scheme (which, again, can be seen as the very definition of $\mu_{0,m}$). To achieve Steps 2 and 3 we shall prove the two Propositions below.

	Step 2 of the proof of Theorem~\ref{thm:free case} follows from (a particular case of) the analysis of grand-canonical states in~\cite{LewNamRou-20}, with some extra care to track the dependence of remainder terms on $\epsilon$:
	
	\begin{proposition}[\textbf{Classical limit of the relaxed free canonical Gibbs state}]\label{pro:class lim relax}\mbox{}\\
	Let $\mu_{\epsilon,m}$ be as in~\eqref{eq:def mu epsilon} and set  
	$$ \epsilon = T^{-a}, \quad 0<a<\frac{s-2}{4s} $$ 
	Then, for all $k\geq 1$
	\begin{align} \label{limi-cano 2}
		k!\frac{\Gamma_{\epsilon,m,T}^{(k)}}{T^k} - \int |u^{\otimes k} \rangle \langle u^{\otimes k}| d\mu_{\epsilon,m}(u) \xrightharpoonup [T\to \infty]{}^\star 0
	\end{align}
	weakly-$\star$ in the trace-class $\gS^1\left(\mathfrak {h}^k\right)$. Moreover
	\begin{equation}\label{eq:class part relax}
	\frac{Z_{\eps,m,T}}{Z_{0,T}} - z_{\eps,m}^r \xrightarrow[T\to\infty]{} 0.
	\end{equation}

	\end{proposition}
	
	As for the third ingredient we show that the Gibbs state with ``relaxed particle number constraint''~\eqref{eq:relax Gibbs} indeed converges to the canonical state $\Gamma^c_{mT,T}$ in the limit $\epsilon \to 0$ with a controlled dependence on $T$ of the limit:

	\begin{proposition}[\textbf{Relaxation of the free canonical Gibbs state}]\label{pro:relax}\mbox{}\\
	Let $a>0$ and set 
	$$ \epsilon = T^{-a}.$$
	Then, for all $k\geq 1$,
	\begin{align} \label{limi-T-eps}
		\frac{\Gamma^{(k)}_{\epsilon,m,T}}{T^k} - \frac{(\Gamma^c_{mT,T,0})^{(k)}}{T^k} \xrightharpoonup[T\to \infty]{}^\star 0 
	\end{align}
	weakly-$\star$ in $\gS ^1 \left(\mathfrak h^k\right)$.
	\end{proposition}

\begin{proof}[Proof of Theorem~\ref{thm:free case}, given Propositions~\ref{pro:class lim relax} and~\ref{pro:relax}]
Theorem~\ref{thm:free case} follows from the combination of Proposition~\ref{pro:fixed mass meas}, Proposition~\ref{pro:relax} and Proposition~\ref{pro:class lim relax}. In fact, the combination of \eqref{limi-eps},~\eqref{limi-cano 2} and \eqref{limi-T-eps} gives weak-$\star$ trace-class convergence in~\eqref{limi-cano}, choosing $\epsilon (T)$ appropriately, e.g. $\eps = T^{\frac{2-s}{8s}}$. Convergence in trace-class norm follows as usual from~\cite[Addendum~H]{Simon-79} because both sides of the equation are positive operators, and that the trace (hence the trace-norm) of the left side is easily seen to converge to the trace of the right side.
\end{proof}

The rest of this section focuses on proving Propositions~\ref{pro:class lim relax} and~\ref{pro:relax} (in two separate subsections), i.e. giving the two missing steps outlined in Figure~\ref{fig:diagram}.

\subsection{Classical limit of the relaxed free canonical Gibbs state}\label{sec:lim relax}

We turn to the proof of Proposition~\ref{pro:class lim relax}. We follow the general variational approach of~\cite{LewNamRou-15,LewNamRou-17,LewNamRou-20}, proving upper and lower bounds to the free energy associated to~\eqref{eq:relax Gibbs} in terms of the classical partition function defined in~\eqref{eq:def mu epsilon}. We skip some details since we are mostly concerned with tracking the dependence on $\epsilon$ of various error terms that have been studied extensively in the mentioned references. 

The Gibbs variational principle gives
	$$
	\begin{aligned}
		-\log\left(\frac{Z_{\epsilon,m,T}}{Z_{0,T}}\right) =\inf_{\Gamma \geq 0 \atop {\Tr}_{\mathcal{F}}[\Gamma]=1}\left\{\mathcal{H}(\Gamma,\Gamma_{0,T}) +\frac{1}{\epsilon}{\Tr}\left[ \left(\frac{\mathcal{N}}{T}-m\right)^2\Gamma\right]\right\}
	\end{aligned}
	$$
	where 
	$$
	\mathcal{H}(\Gamma,\Gamma_{0,T}):= {\Tr}[\Gamma(\log \Gamma-\log \Gamma_{0,T})]
	$$
	is the von Neumann quantum relative entropy and $\Gamma_{0,T}$ is the free Gibbs state
	$$
	\Gamma_{0,T} := \frac{1}{Z_{0,T}} \exp\left(-\frac{1}{T} d\Gamma(h)\right).
	$$
	Moreover, the Gibbs state $\Gamma_{\epsilon,m,T}$ is the unique minimizer of the above. 
	
	In a similar manner, we have at the classical level
	$$
	-\log (z^r_{\epsilon,m}) = \inf_{\nu \text{ prob. meas. } \atop \text{ on } \mathfrak{h}} \left\{\mathcal{H}_{\text{cl}}(\nu,\mu_0) +\frac{1}{\epsilon} \int (\langle u,u\rangle -m)^2 d\nu(u) \right\},
	$$
	where 
	$$
	\mathcal{H}_{\text{cl}}(\nu,\mu_0) =\int\frac{d\nu}{d\mu_0} \log \frac{d\nu}{d\mu_0} d\mu_0
	$$
	is the classical relative entropy and the Gibbs measure $\mu_{\epsilon,m}$ is the unique minimizer of the above minimization problem.
	
	We relate the above variational problems in the limit $T\to \infty$ to provide the Proof of Proposition~\ref{pro:class lim relax}. We begin by controling the expected particle number in $\Gamma_{\epsilon,m,T}$, uniformly in $\eps$:

	\begin{lemma}[\textbf{Particle number in the relaxed canonical Gibbs state}]\label{lem-NTk-GepsmT-est}\mbox{}\\
	Let $k\geq 1$ and $m>0$. Then there exists $C(k,m)>0$ such that for all $T\geq 1$ and all $\epsilon>0$ small,
	\begin{align} \label{NTk-GepsmT-est}
	{\Tr}\left[\left(\frac{\mathcal{N}}{T}\right)^k\Gamma_{\epsilon,m,T}\right]\leq C(k,m).
	\end{align}
	In particular, we get
	\begin{align} \label{GepsmT-k}
	\frac{1}{T^k}{\Tr}\left[\Gamma^{(k)}_{\epsilon,m,T}\right]\leq C(k,m).
	\end{align}
	\end{lemma}
	
	\begin{proof}
	It suffices to prove \eqref{NTk-GepsmT-est} since \eqref{GepsmT-k} clearly follows. We denote
	$$
	\Gamma_M = \frac{1}{Z_M} \exp \left(-\frac{1}{T}d\Gamma(h) - M \left( \frac{\mathcal N}{T} -m\right)^2 \right),
	$$
	where $Z_M$ is the partition function so that $\text{Tr}[\Gamma_M]=1$. In particular, $\Gamma_{\epsilon,m,T}=\Gamma_{M=1/\epsilon}$.
	Recall that $\Gamma_M$ is the unique minimizer for the free energy functional
	$$
	\mathcal F_M(\Gamma):= \text{Tr}\left[ \left(d\Gamma(h) + MT\left(\frac{\mathcal N}{T}-m \right)^2\right) \Gamma\right] + T \text{Tr}[\Gamma \log \Gamma]
	$$
	over all states 
	$$
	\Gamma \in \mathcal S(\mathcal F):= \{ 0\leq \Gamma=\Gamma^*, \quad \text{Tr}[\Gamma]=1 \}.
	$$
	Moreover, 
	$$
	\mathcal F_M(\Gamma_M) = -T \log Z_M.
	$$
	Observe that for any $0<c<1$
	\begin{align} \label{NTk-proof-1}
	\frac{Z_{M(1-c)}}{Z_M} = \text{Tr}\left[ \exp \left( cM \left(\frac{\mathcal N}{T}-m \right)^2\right) \Gamma_M\right].
	\end{align}
	We also note that for all $M, K>0$,
	\begin{align} \label{NTk-proof-2}
	\mathcal F_M(\Gamma_M) = \mathcal F_K(\Gamma_K) + T \mathcal H(\Gamma_M, \Gamma_K) + T(M-K) \text{Tr}\left[\left(\frac{\mathcal N}{T}-m \right)^2 \Gamma_M \right],
	\end{align}
	where
	$$
	\mathcal H(\Gamma_M, \Gamma_K) = \text{Tr}[\Gamma_M(\log \Gamma_M -\log \Gamma_K)]
	$$
 	is the quantum relative entropy. 
	
	Let $0<c<1$ be a constant to be chosen later. Applying \eqref{NTk-proof-2} with $M= M(1-c)$ and $K=M$, we have
	$$
	\begin{aligned}
		\mathcal F_{M(1-c)}(\Gamma_{M(1-c)}) &= T \mathcal H(\Gamma_{M(1-c)},\Gamma_M) + \mathcal F_M(\Gamma_M) - cMT \text{Tr}\left[\left( \frac{\mathcal N}{T}-m\right)^2 \Gamma_{M(1-c)}\right]\\
		&\geq \mathcal F_M(\Gamma_M) - cMT \text{Tr}\left[\left( \frac{\mathcal N}{T}-m\right)^2 \Gamma_{M(1-c)}\right]
	\end{aligned}
	$$
	or
	$$
	-T\log Z_{M(1-c)} \geq -T \log Z_M - cMT \text{Tr}\left[\left( \frac{\mathcal N}{T}-m\right)^2 \Gamma_{M(1-c)}\right].
	$$
	Hence
	\begin{align} \label{NTk-proof-3}
	\log \frac{Z_{M(1-c)}}{Z_M} \leq cM\text{Tr}\left[\left( \frac{\mathcal N}{T}-m\right)^2 \Gamma_{M(1-c)}\right].
	\end{align}
	Applying \eqref{NTk-proof-2} with $M=M(1-c)$ and $K=0$, we have
	$$
	\begin{aligned}
		\mathcal F_{M(1-c)}(\Gamma_{M(1-c)}) &= \mathcal F_0(\Gamma_0) + T\mathcal H(\Gamma_{M(1-c)},\Gamma_0) + M(1-c)T \text{Tr}\left[ \left(\frac{\mathcal N}{T}-m\right)^2 \Gamma_{M(1-c)}\right] \\
		&\geq \mathcal F_0(\Gamma_0) + MT(1-c) \text{Tr}\left[ \left(\frac{\mathcal N}{T}-m\right)^2 \Gamma_{M(1-c)}\right].
	\end{aligned}
	$$
	Since $\Gamma_{M(1-c)}$ is the (unique) minimizer of $\mathcal F_{M(1-c)}(\Gamma)$, we have
	$$
	\mathcal F_{M(1-c)}(\Gamma_{M(1-c)}) \leq \mathcal F_{M(1-c)}(\Gamma_0) = \mathcal F_0(\Gamma_0) + TM(1-c) \text{Tr}\left[ \left(\frac{\mathcal N}{T}-m\right)^2 \Gamma_0\right].
	$$
	Thus we get
	\begin{align} \label{NTk-proof-4}
	\text{Tr}\left[ \left(\frac{\mathcal N}{T}-m\right)^2 \Gamma_{M(1-c)}\right] \leq \text{Tr}\left[ \left(\frac{\mathcal N}{T}-m\right)^2 \Gamma_0\right].
	\end{align}
	From \eqref{NTk-proof-1}, \eqref{NTk-proof-3} and \eqref{NTk-proof-4}, we obtain
	$$
	\text{Tr}\left[ \exp \left( cM \left(\frac{\mathcal N}{T}-m \right)^2\right) \Gamma_M\right] \leq e^{\text{Tr}\left[cM \text{Tr}\left[ \left(\frac{\mathcal N}{T}-m\right)^2 \Gamma_0\right] \right]}.
	$$
	Taking $c=1/M = \epsilon$ and using known bounds on $\Gamma_0$ (see e.g.~\cite[Section~3]{LewNamRou-15}), we obtain for all $T\geq 1$,
	$$
	\text{Tr}\left[\exp\left(\left(\frac{\mathcal N}{T}-m\right)^2\right) \Gamma_{\epsilon,m,T}\right] \leq C(m).
	$$
	In particular,
	\begin{align} \label{NTk-proof-5}
	\text{Tr}\left[ \left( \frac{\mathcal N}{T}-m\right)^{2k} \Gamma_{\epsilon,m,T}\right] \leq C(k,m), \quad \forall k \in \mathbb N.
	\end{align}
	It follows that
	$$
	\begin{aligned}
		\text{Tr}\left[\left(\frac{\mathcal N}{T}\right)^k \Gamma_{\epsilon,m,T}\right] &= \text{Tr}\left[\left(\frac{\mathcal N}{T}-m+m\right)^k \Gamma_{\epsilon,m,T}\right]  \\
		&= \text{Tr}\left[ \sum_{l=0}^k \binom{k}{l} m^{k-l} \left(\frac{\mathcal N}{T}-m\right)^l \Gamma_{\epsilon,m,T}\right]
	\end{aligned}
	$$
	which, together with a Cauchy-Schwarz inequality and \eqref{NTk-proof-5}, proves \eqref{NTk-GepsmT-est}.
	\end{proof}
	
	We now proceed with the
	
	\begin{proof}[Proof of Proposition~\ref{pro:class lim relax}]
	 Recall that we set $ \epsilon = T^{-a}$ with $0<a<\frac{s-2}{4s}$.
	
	\medskip
	
	\noindent\textbf{Step 1. Free energy upper bound.} 
	We first prove by a trial state argument that
	\begin{align} \label{uppe-boun-log}
		\lim_{T \to \infty} -\log \left(\frac{Z_{\epsilon,m,T}}{Z_{0,T}} \right) + \log\left(\int_{P\mathfrak{h}} \exp\left(-\frac{1}{\epsilon}\left(\langle u, u\rangle -m\right)^2\right)d\mu_{0,\Lambda}(u)\right) \leq 0,
	\end{align}
	where $P=P_\Lambda = \mathds{1}_{\{h \leq \Lambda\}}$ with 
	$$
	\Lambda = T^b
	$$
	for some $b>0$ to be chosen later (see~\eqref{cond-b-1} below), and $d\mu_{0,\Lambda}(u)$ is the cylindrical projection of $\mu_0$ on $P\mathfrak h$. As per the results in Proposition~\ref{pro:fixed mass meas} this implies that
	\begin{align} \label{uppe-boun-parti}
		\frac{Z_{0,T}}{Z_{\epsilon,m,T}} \leq \frac{C(m)}{\sqrt{\epsilon}}.
	\end{align} 
	 It is known that the Fock space can be factorized as
	$$
	\mathcal{F} \simeq \mathcal{F}(P\mathfrak{h}) \otimes \mathcal{F}(Q\mathfrak{h}) 
	$$
	in the sense that there exists a unitary map (cf Equation~\eqref{eq:unitary map} below)
	$$\mathcal{U}: \mathcal{F}(P\mathfrak{h}\oplus Q\mathfrak{h}) \to \mathcal{F}(P\mathfrak{h}) \otimes \mathcal{F}(Q\mathfrak{h}).$$
	Here $Q= P^\perp_\Lambda = \mathds{1}_{\{h>\Lambda\}}$.
	
	The free Gibbs state $\Gamma_{0,T}$ can be factorized as
	$$
	\Gamma_{0,T} = \mathcal U^* (\Gamma_{0,T,P} \otimes \Gamma_{0,T,Q}) \mathcal U,
	$$
	where
	\begin{align*}
	\Gamma_{0,T,P} &= {\Tr}_{\mathcal{F}(Q\mathfrak{h})}[\mathcal{U}\Gamma_{0,T}\mathcal{U}^*] = \frac{\exp\left(-\frac{1}{T}d\Gamma(Ph)\right)}{{\Tr}_{\mathcal{F}(P\mathfrak{h})}[\exp\left(-\frac{1}{T}d\Gamma(Ph)\right)]}, \\
	\Gamma_{0,T,Q} &= {\Tr}_{\mathcal{F}(P\mathfrak{h})}[\mathcal{U}\Gamma_{0,T}\mathcal{U}^*] = \frac{\exp\left(-\frac{1}{T}d\Gamma(Qh)\right)}{{\Tr}_{\mathcal{F}(Q\mathfrak{h})}[\exp\left(-\frac{1}{T}d\Gamma(Qh)\right)]}.
	\end{align*}
	We also denote $\Gamma_{\epsilon,m,T}^P$ the $P$-localized interacting Gibbs state 
	\begin{align} \label{Gam-epsmT-P}
	\Gamma_{\epsilon,m,T}^P = \frac{\exp\left(-\frac{1}{T}d\Gamma(Ph) -\frac{1}{\epsilon}\left(\frac{d\Gamma(P)}{T}-m\right)^2\right)}{{\Tr}_{\mathcal{F}(P\mathfrak{h})}\left[\exp\left(-\frac{1}{T}d\Gamma(Ph) -\frac{1}{\epsilon}\left(\frac{d\Gamma(P)}{T}-m\right)^2\right)\right]}.
	\end{align}
	We define the following trial state
	\begin{equation}\label{eq:trial state}
	\boxed{\Gamma = \mathcal{U}^*(\Gamma_{\epsilon,m,T}^P\otimes \Gamma_{0,T,Q})\mathcal{U}.}
	\end{equation}
	By the Gibbs variational principle, we have
	\begin{align} \label{uppe-boun-proo-1}
	-\log\left(\frac{Z_{\epsilon,m,T}}{Z_{0,T}}\right)\leq \mathcal{H}(\Gamma,\Gamma_{0,T}) +\frac{1}{\epsilon}{\Tr}\left[\left(\frac{\mathcal{N}}{T}-m\right)^2\Gamma\right].
	\end{align}
	Following step by step the arguments of the proof of~\cite[Lemma~8.3]{LewNamRou-15} to track the dependence on $\epsilon$ of error terms we obtain first
	\begin{align}\label{eq:red fin dim can}
		-\log\left(\frac{Z_{\epsilon,m,T}}{Z_{0,T}}\right)&\leq {\Tr}_{\mathcal{F}(P\mathfrak{h})}\left[\Gamma_{\epsilon,m,T}^P\left(\log \Gamma_{\epsilon,m,T}^P-\log \Gamma_{0,T,P}\right)\right] + \epsilon^{-1} {\Tr}_{\mathcal{F}(P\mathfrak{h})}\left[\left(\frac{d\Gamma(P)}{T}-m\right)^2\Gamma_{\epsilon,m,T}^P\right] \nonumber\\
		&\quad  + ({\Tr}[Qh^{-1}Q])^2 + C(m) {\Tr}[Q h^{-1}Q] + C(m) \epsilon^{-1} {\Tr}[Qh^{-1}Q].
	\end{align}
	The terms on the second line are errors. We focus on the first line and estimate it in terms of the classical variational problem.
	
	By the Gibbs variational principle, the right hand side in the first line is equal to
	$$
	-\log \left(\frac{{\Tr}\left[\exp\left(-\frac{1}{T}d\Gamma(Ph) -\frac{1}{\epsilon}\left(\frac{d\Gamma(P)}{T}-m\right)^2\right)\right]}{{\Tr}[\exp\left(-\frac{1}{T}d\Gamma(Ph)\right)]}\right).
	$$
	We first have, for the denominator,
	\begin{align} \label{Z0T-P}
	{\Tr}\left[\exp\left(-\frac{1}{T} d\Gamma(Ph)\right)\right] =\prod_{\lambda_j \leq \Lambda} \frac{1}{1-e^{-\lambda_j/T}}.
	\end{align} 
	To estimate the numerator, we first use the resolution of identity in terms of coherent states (see \eqref{eq:cohe stat})
	$$
	\left(\frac{T}{\pi}\right)^{{\Tr}[P]} \int_{P\mathfrak{h}} |\xi(u\sqrt{T})\rangle \langle \xi(u\sqrt{T})| du = \mathds{1}_{\mathcal{F}(P\mathfrak{h})}
	$$
	to obtain
	\begin{align*}
		&{\Tr}\Big[\exp\Big(-\frac{1}{T}d\Gamma(Ph) -\frac{1}{\epsilon}\Big(\frac{d\Gamma(P)}{T}-m\Big)^2\Big)\Big] \\
		&\quad = \Big(\frac{T}{\pi}\Big)^{{\Tr}[P]} \int_{P\mathfrak{h}} {\Tr}\Big[ \exp\Big(-\frac{1}{T}d\Gamma(Ph) -\frac{1}{\epsilon}\Big(\frac{d\Gamma(P)}{T}-m\Big)^2\Big) |\xi(u\sqrt{T}\rangle \langle \xi(u\sqrt{T})|\Big]du\\
		&\quad = \Big(\frac{T}{\pi}\Big)^{{\Tr}[P]} \int_{P\mathfrak{h}}\Big\langle \xi(u\sqrt{T}), \exp\Big(-\frac{1}{T}d\Gamma(Ph) -\frac{1}{\epsilon}\Big(\frac{d\Gamma(P)}{T}-m\Big)^2\Big) \xi(u\sqrt{T})\Big\rangle du.
	\end{align*}
	We next use the Peierls-Bogoliubov inequality: for $F: \mathbb{R} \to \mathbb{R}$ a convex function and $A$ is a self-adjoint operator on a Hilbert space $\mathcal{H}$, 
	$$
	\langle x, F(A) x\rangle \geq F(\langle x, Ax\rangle), \quad \forall x \in \mathcal{H}, \|x\|=1
	$$
	with $F(\lambda)=e^\lambda$, $A=-\frac{1}{T}d\Gamma(Ph) -\frac{1}{\epsilon}\left(\frac{d\Gamma(P)}{T}-m\right)^2$, $x=\xi(u\sqrt{T})$, and $\mathcal{H}=\mathcal{F}$ to get
	\begin{align*}
		&{\Tr}\Big[\exp\Big(-\frac{1}{T}d\Gamma(Ph) -\frac{1}{\epsilon}\Big(\frac{d\Gamma(P)}{T}-m\Big)^2\Big)\Big] \\
		&\quad \geq \Big(\frac{T}{\pi}\Big)^{{\Tr}[P]} \int_{P\mathfrak{h}} \exp\Big(-\Big\langle \xi(u\sqrt{T}), \Big(\frac{1}{T}d\Gamma(Ph) +\frac{1}{\epsilon}\Big(\frac{d\Gamma(P)}{T}-m\Big)^2\Big)\xi(u\sqrt{T})\Big\rangle\Big) du.
	\end{align*}
	Using 
	$$
	d\Gamma(Ph) = \sum_{\lambda_j \leq \Lambda} \lambda_j \ada(u_j)a(u_j),
	$$
	with creation/annihilation operators $\ada(u_j),a(u_j)$,  we have for $u \in P\mathfrak{h}$ that
	\begin{align} \label{xi-uT-P}
	\begin{aligned}
		\langle \xi(u\sqrt{T}), d\Gamma(Ph)\xi(u\sqrt{T})\rangle &= \sum_{\lambda_j \leq \Lambda} \lambda_j \langle \xi(u\sqrt{T}), \ada(u_j)a(u_j)\xi(u\sqrt{T})\rangle\\
		&= \sum_{\lambda_j \leq \Lambda} \lambda_j \langle a(u_j)\xi(u\sqrt{T}), a(u_j)\xi(u\sqrt{T})\rangle\\
		&=\sum_{\lambda_j \leq \Lambda} \lambda_j |\langle u_j, u\sqrt{T}\rangle|^2 \|\xi(u\sqrt{T})\|^2_{\mathcal{F}}\\
		&= T\sum_{\lambda_j \leq \Lambda} \lambda_j |\langle u, u_j\rangle|^2\\
		&=T\langle u, hu\rangle,
	\end{aligned}
	\end{align}
	where we used
	$$
	a(f)\xi(u)=\langle f, u\rangle \xi(u), \quad \|\xi(u)\|_{\mathcal{F}}=1.
	$$
	Similarly, we have
	$$
	\left(\frac{d\Gamma(P)}{T}-m\right)^2 = \frac{1}{T^2}(d\Gamma(P))^2 -\frac{2m}{T}d\Gamma(P)+m^2
	$$
	with
	$$
	d\Gamma(P) = \sum_{\lambda_j \leq \Lambda} \ada(u_j)a(u_j).
	$$
	We observe that
	\begin{align*}
		(d\Gamma(P))^2 &= \sum_{\lambda_j, \lambda_l \leq \Lambda} \ada(u_j)a(u_j)\ada(u_l)a(u_l) \\
		&=\sum_{\lambda_j \leq \Lambda} \ada(u_j) a(u_j) \ada(u_j) a(u_j) +\sum_{\lambda_j, \lambda_l \leq \Lambda \atop j \ne l} \ada(u_j)\ada(u_l)a(u_j) a(u_l) \\
		&=\sum_{\lambda_j\leq \Lambda} \ada(u_j) \ada(u_j) a(u_j) a(u_j) + \sum_{\lambda_j \leq \Lambda} \ada(u_j)a(u_j) +\sum_{\lambda_j, \lambda_l \leq \Lambda \atop j\ne l} \ada(u_j)\ada(u_l)a(u_j) a(u_l) \\
	\end{align*}
	where we used the canonical commutation relations
	$$
	a(u_j)\ada(u_l) = \ada(u_l)a(u_j) - [\ada(u_l),a(u_j)]=\ada(u_l)a(u_j)-\delta_{jl}. 
	$$
	Thus we get for $u\in P\mathfrak{h}$,
	\begin{align*}
		\langle \xi(u\sqrt{T}), d\Gamma(P)\xi(u\sqrt{T})\rangle &= \sum_{\lambda_j\leq \Lambda} \langle \xi(u\sqrt{T}), \ada(u_j)a(u_j)\xi(u\sqrt{T})\rangle\\
		&=T \sum_{\lambda_j \leq \Lambda} |\langle u_j, u\rangle|^2\\
		&= T\langle u, u\rangle
	\end{align*}
	and
	\begin{align*}
		\langle \xi(u\sqrt{T}), (d\Gamma(P))^2\xi(u\sqrt{T})\rangle &= \sum_{\lambda_j \leq \Lambda} \langle \xi(u\sqrt{T}), \ada(u_j)\ada(u_j) a(u_j) a(u_j)\xi(u\sqrt{T})\rangle\\
		&\quad - \sum_{\lambda_j \leq \Lambda} \langle \xi(u\sqrt{T}), \ada(u_j) a(u_j) \xi(u\sqrt{T})\rangle\\
		&\quad + \sum_{\lambda_j, \lambda_l \leq \Lambda \atop j \ne l} \langle \xi(u\sqrt{T}), \ada(u_j)\ada(u_l) a(u_j)a(u_l) \xi(u\sqrt{T}) \rangle \\
		&= -T\langle u, u\rangle + T^2\sum_{\lambda_j\leq \Lambda} |\langle u_j, u\rangle|^4 + T^2\sum_{\lambda_j, \lambda_l \leq \Lambda \atop j \ne l} |\langle u_j, u\rangle|^2|\langle u_l, u\rangle|^2 \\
		&= T^2 \Big(\sum_{\lambda_j\leq \Lambda}|\langle u_j, u\rangle|^2\Big)^2 - T\langle u, u\rangle \\
		&= T^2 (\langle u, u\rangle)^2 - T\langle u, u\rangle.
	\end{align*}
	Thus we obtain for $u \in P \mathfrak h$,
	\begin{align*}
		&\Big\langle \xi(u\sqrt{T}),\Big(\frac{d\Gamma(P)}{T}-m\Big)^2\xi(u\sqrt{T})\Big\rangle \\
		&\quad=\frac{1}{T^2}\langle \xi(u\sqrt{T}), (d\Gamma(P))^2 \xi(u\sqrt{T})\rangle -\frac{2m}{T}\langle \xi(u\sqrt{T}), d\Gamma(P)\xi(u\sqrt{T})\rangle +m^2\\
		&\quad =\frac{1}{T^2}\left(T^2(\langle u, u\rangle)^2 -T\langle u,u\rangle\right) -\frac{2m}{T} T\langle u, u\rangle +m^2 \\
		&\quad = \left(\langle u, u\rangle -m\right)^2 - \frac{1}{T}\langle u, u\rangle.
	\end{align*}
	Putting the above identities together, we get
	\begin{align*}
		&{\Tr}\Big[\exp\Big(-\frac{1}{T}d\Gamma(Ph) -\frac{1}{\epsilon}\Big(\frac{d\Gamma(P)}{T}-m\Big)^2\Big)\Big] \\
		&\quad \geq \Big(\frac{T}{\pi}\Big)^{{\Tr}[P]} \int_{P\mathfrak{h}} \exp\Big(-\langle u,hu\rangle -\frac{1}{\epsilon}\left(\langle u, u\rangle -m\right)^2+\frac{1}{T\epsilon} \langle u, u\rangle\Big)du \\
		&\quad \geq \Big(\frac{T}{\pi}\Big)^{{\Tr}[P]} \int_{P\mathfrak{h}} \exp\Big(-\langle u,hu\rangle -\frac{1}{\epsilon}\left(\langle u, u\rangle -m\right)^2\Big)du.
	\end{align*}
	Observe also that for $u \in P\mathfrak{h}$, 
	$$
	e^{-\langle u, hu\rangle} du = \prod_{\lambda_j\leq \Lambda} e^{-\lambda_j|\alpha_j|^2} d\alpha_j= \Big(\prod_{\lambda_j \leq \Lambda} \frac{\pi}{\lambda_j}\Big) d\mu_{0,\Lambda}(u).
	$$
	Thus,
	\begin{align*}
		&{\Tr}\Big[\exp\Big(-\frac{1}{T}d\Gamma(Ph) -\frac{1}{\epsilon}\Big(\frac{d\Gamma(P)}{T}-m\Big)^2\Big)\Big] \\
		&\quad \geq \Big(\prod_{\lambda_j\leq \Lambda} \frac{T}{\lambda_j}\Big) \int_{P\mathfrak{h}} \exp\Big(-\frac{1}{\epsilon}(\langle u, u\rangle -m)^2\Big)d\mu_{0,\Lambda}(u),
	\end{align*}
	hence
	\begin{align*}
		&-\log \Bigg(\frac{{\Tr}\Big[\exp\left(-\frac{1}{T}d\Gamma(Ph) -\frac{1}{\epsilon}\Big(\frac{d\Gamma(P)}{T}-m\right)^2\Big)\Big]}{{\Tr}\Big[\exp\Big(-\frac{1}{T}d\Gamma(Ph)\Big)\Big]}\Bigg)\\
		&\quad \leq -\log \Bigg(\Big(\prod_{\lambda_j\leq \Lambda} \frac{T(1-e^{-\lambda_j/T})}{\lambda_j}\Big) \int_{P\mathfrak{h}} \exp\Big(-\frac{1}{\epsilon}\Big(\langle u, u\rangle -m\Big)^2\Big)d\mu_{0,\Lambda}(u)\Bigg).
	\end{align*}
	Inserting the above in~\eqref{eq:red fin dim can} we obtain
	\begin{align*}
		-\log\left(\frac{Z_{\epsilon,m,T}}{Z_{0,T}}\right)&\leq -\log \Bigg(\Big(\prod_{\lambda_j\leq \Lambda} \frac{T(1-e^{-\lambda_j/T})}{\lambda_j}\Big) \int_{P\mathfrak{h}} \exp\Big(-\frac{1}{\epsilon}\Big(\langle u, u\rangle -m\Big)^2\Big)d\mu_{0,\Lambda}(u)\Bigg) \\
		&\quad + ({\Tr}[Qh^{-1}Q])^2 + C(m) {\Tr}[Q h^{-1}Q] + C(m) \epsilon^{-1} {\Tr}[Qh^{-1}Q].
	\end{align*}
	We estimate for any $\frac{1}{2}+\frac{1}{s}<q <1$,
	$$
	{\Tr}[Qh^{-1}Q] = \sum_{\lambda_j>\Lambda} \lambda_j^{-1} \leq \Lambda^{-(1-q)} \sum_{\lambda_j >\Lambda} \lambda_j^{-q} \leq \Lambda^{-(1-q)} {\Tr}[h^{-q}],
	$$
	where we recall (see \cite{LewNamRou-17}) that ${\Tr}[h^{-q}]<\infty$ for all $q>1/2+1/s$.	We deduce
	$$
	\begin{aligned}
		-\log \left(\frac{Z_{\epsilon,m,T}}{Z_{0,T}} \right) \leq &- \log\Bigg(\Big(\prod_{\lambda_j\leq \Lambda} \frac{T(1-e^{-\lambda_j/T})}{\lambda_j}\Big) \int_{P\mathfrak{h}} \exp\Big(-\frac{1}{\epsilon}\Big(\langle u, u\rangle -m\Big)^2\Big)d\mu_{0,\Lambda}(u)\Bigg) \\
		&+ C(m,q) \epsilon^{-1} \Lambda^{-(1-q)}
	\end{aligned}
	$$
	for all $\Lambda >0$ sufficiently large and all $\epsilon>0$ sufficiently small. In particular, we have
	$$
	\begin{aligned}
		-\log \left(\frac{Z_{\epsilon,m,T}}{Z_{0,T}} \right) &+ \log\left(\int_{P\mathfrak{h}} \exp\left(-\frac{1}{\epsilon}\left(\langle u, u\rangle -m\right)^2\right)d\mu_{0,\Lambda}(u)\right) \\
		&\leq - \log\Big(\prod_{\lambda_j\leq \Lambda} \frac{T(1-e^{-\lambda_j/T})}{\lambda_j}\Big)+ C(m,q) \epsilon^{-1} \Lambda^{-(1-q)}.
	\end{aligned}
	$$
	To estimate the log term, we write
	$$
	\begin{aligned}
		- \log\Big(\prod_{\lambda_j\leq \Lambda} \frac{T(1-e^{-\lambda_j/T})}{\lambda_j}\Big) &= -\sum_{\lambda_j \leq \Lambda} \log \Big(\frac{T(1-e^{-\lambda_j/T})}{\lambda_j}\Big) \\
		&= \sum_{\lambda_j \leq \Lambda} \log \Big(\frac{\lambda_j/T}{1-e^{-\lambda_j/T}}\Big).
	\end{aligned}
	$$
	Since
	$$
	e^{-\frac{\lambda_j}{T}} = 1 - \frac{\lambda_j}{T} + O\Big(\Big(\frac{\lambda_j}{T}\Big)^2\Big),
	$$
	we have
	$$
	\frac{\lambda_j/T}{1-e^{-\lambda_j/T}} = 1 + O(\lambda_j/T).
	$$
	Since
	$$
	\log(1+x) \leq C x, \quad \forall x\geq 0
	$$
	we get
	$$
	\sum_{\lambda_j \leq \Lambda} \log \left(\frac{\lambda_j/T}{1-e^{-\lambda_j/T}}\right) \leq C \sum_{\lambda_j \leq \Lambda} (\lambda_j/T) \leq C T^{-1} \Lambda^{3/2+1/s},
	$$
	where we have used the Cwikel-Lieb-Rozenbljum law (see~\cite[Lemma D1]{DinRou-23})
	\begin{align} \label{K-boun}
		\#\{\lambda_j : \lambda_j \leq \Lambda\} \sim \Lambda^{1/2+1/s}.
	\end{align}
	Altogether we have proved that
	$$
	\begin{aligned}
	-\log \Big(\frac{Z_{\epsilon,m,T}}{Z_{0,T}} \Big) &+ \log\Big(\int_{P\mathfrak{h}} \exp\Big(-\frac{1}{\epsilon}\left(\langle u, u\rangle -m\right)^2\big)d\mu_{0,\Lambda}(u)\Big) \\
	&\leq C T^{-1} \Lambda^{3/2+1/s} + C(m,q)\epsilon^{-1} \Lambda^{-(1-q)}
	\end{aligned}
	$$
	for any $1/2+1/2<q<1$. With $\epsilon =T^{-a}$ and $\Lambda=T^b$ this reads
	$$
	\begin{aligned}
		-\log \Big(\frac{Z_{\epsilon,m,T}}{Z_{0,T}} \Big) &+ \log\Big(\int_{P\mathfrak{h}} \exp\Big(-\frac{1}{\epsilon}\Big(\langle u, u\rangle -m\Big)^2\Big)d\mu_{0,\Lambda}(u)\Big) \\
		&\leq C T^{b(3/2+1/s) - 1}+ C(m,q) T^{a-b(1-q)}.
	\end{aligned}
	$$
	To make the error small, we need
	\begin{align} \label{cond-b-1}
	\frac{a}{1-q} < b < \frac{2s}{3s+2}.
	\end{align}
	For  $0<a<\frac{s-2}{4s}$ we can always find $q \in (1/2+1/s,1)$ and $b$ so that \eqref{cond-b-1} is satisfied. We pick one and finally obtain the desired upper bound \eqref{uppe-boun-log}.

	\medskip
	
	\noindent\textbf{Step 2. Free energy lower bound.} Let again $P = P_\Lambda$ for $\Lambda = T^b$, $0<b<1$ to be optimized over later.
	
	Let $\Gamma_{\eps,m,T,P}$ be the $P$-localization of $\Gamma_{\epsilon,m,T}$ defined by 
	$$ \Gamma_{\eps,m,T,P} = {\Tr}_{\mathcal{F}(Q\mathfrak{h})}[\mathcal{U}\Gamma_{\eps,m,T}\mathcal{U}^*]$$
	It satisfies (cf~\cite{Lewin-11} or~\cite{Rougerie-LMU})
	$$
	P^{\otimes k} \Gamma^{(k)}_{\epsilon,m,T} P^{\otimes k} = \Gamma^{(k)}_{\epsilon,m,T,P}
	$$
	for all $k\geq 1$. Let 
	$$
	d\mu_{\epsilon,m,T,P}(u) = \Big( \frac{T}{\pi}\Big)^{{\Tr}[P]} \left\langle \xi(u\sqrt{T}), \Gamma_{\epsilon,m,T,P} \xi(u\sqrt{T})\right\rangle_{\mathcal F(P\mathfrak h)} du
	$$
	be the lower symbol of $\Gamma_{\epsilon,m,T}$ on $P\mathfrak h$ at scale $\frac{1}{T}$ and 
	$$
	d\mu_{0,T,P}(u) = \Big( \frac{T}{\pi}\Big)^{{\Tr}[P]} \left\langle \xi(u\sqrt{T}), \Gamma_{0,T,P} \xi(u\sqrt{T})\right\rangle_{\mathcal F(P\mathfrak h)} du
	$$
	be that of the $P$-localization of $\Gamma_{0,T}$. Define further the probability measure
	$$
	d\tilde{\mu}_{\epsilon,m,T,P}(u) := \frac{1}{\tilde{z}_{\epsilon,m,T,P}} \exp \left(-\frac{1}{\epsilon} (\langle u, u\rangle-m)^2\right) d\mu_{0,T,P}(u)
	$$
	with 
	$$
	\tilde{z}_{\epsilon,m,T,P} =\int_{P\mathfrak h} \exp\left(-\frac{1}{\epsilon} (\langle u, u\rangle -m)^2\right) d\mu_{0,T,P}(u).
	$$
	We claim that 
	\begin{align} \label{L1-conv-proof-1}
	\begin{aligned}
	-\log \Big(\frac{Z_{\epsilon,m,T}}{Z_{0,T}}\Big)&\geq 
	 \mathcal H_{\text{cl}}(\mu_{\epsilon,m,T,P}, \tilde{\mu}_{\epsilon,m,T,P}) - \log (\tilde{z}_{\epsilon,m,T,P}) \\
	&\quad - C(m) T^{-(1-a-b(1/2+1/s))} -C(m) T^{-(b(1-p)-a)} \log\left(T^a\right),
	\end{aligned}
	\end{align}
	The second term of the right-hand side will be shown to match the main term of the upper bound from Step 1, hence proving~\eqref{eq:class part relax}. The first term will provide the desired control leading to the convergence of density matrices.  
	
	By (elements of the proof of) the Berezin-Lieb type inequality from~\cite[Theorem 7.1]{LewNamRou-15} or~\cite[Theorem 5.9]{LewNamRou-20}), the entropy term is bounded from below by
	\begin{align} \label{entro-term}
	\mathcal H(\Gamma_{\epsilon,m,T},\Gamma_{0,T}) \geq \mathcal H(\Gamma_{\epsilon,m,T,P}, \Gamma_{0,T,P}) \geq \mathcal H_{\text{cl}} (\mu_{\epsilon,m,T,P}, \mu_{0,T,P}).
	\end{align}
	For the interaction term, we write
	\begin{align*}
	{\Tr}_{\mathcal F(\mathfrak h)}&\Big[\Big(\frac{\mathcal N}{T}-m\Big)^2\Gamma_{\epsilon,m,T}\Big] \\
	&= {\Tr}_{\mathcal F(P\mathfrak h) \otimes \mathcal F(Q\mathfrak h)} \Big[\mathcal U^*\Big(\frac{\mathcal N}{T}-m\Big)^2\Gamma_{\epsilon,m,T} \mathcal U\Big]\\
	&= {\Tr}_{\mathcal F(P\mathfrak h) \otimes \mathcal F(Q\mathfrak h)} \Big[\mathcal U^*\Big(\frac{\mathcal N}{T}-m\Big)^2 \mathcal U \mathcal U^*\Gamma_{\epsilon,m,T} \mathcal U\Big]\\
	&= {\Tr}_{\mathcal F(P\mathfrak h) \otimes \mathcal F(Q\mathfrak h)} \Big[\Big(\frac{d\Gamma(P) + d\Gamma(Q)}{T}-m\Big)^2 \mathcal U^*\Gamma_{\epsilon,m,T} \mathcal U\Big]\\
	&\geq{\Tr}_{\mathcal F(P\mathfrak h) \otimes \mathcal F(Q\mathfrak h)} \Big[\Big(\Big(\frac{d\Gamma(P)}{T}\Big)^2 -2m\frac{d\Gamma(P)}{T} + m^2 +\Big(\frac{d\Gamma(Q)}{T}\Big)^2-2m\frac{d\Gamma(Q)}{T}\Big) \mathcal U^*\Gamma_{\epsilon,m,T} \mathcal U\Big]\\
	&={\Tr}_{\mathcal F(P\mathfrak h)} \Big[\Big(\Big(\frac{d\Gamma(P)}{T}\Big)^2 -2m\frac{d\Gamma(P)}{T} + m^2 \Big) {\Tr}_{\mathcal F(Q\mathfrak h)}[ \mathcal U^*\Gamma_{\epsilon,m,T} \mathcal U]\Big]\\
	&\quad + {\Tr}_{\mathcal F(Q\mathfrak h)} \Big[\Big(\Big(\frac{d\Gamma(Q)}{T}\Big)^2 -2m\frac{d\Gamma(Q)}{T}\Big) {\Tr}_{\mathcal F(P\mathfrak h)}[ \mathcal U^*\Gamma_{\epsilon,m,T} \mathcal U]\Big] \\
	&= {\Tr}_{\mathcal F(P\mathfrak h)} \Big[\Big(\Big(\frac{d\Gamma(P)}{T}\Big)^2 -2m\frac{d\Gamma(P)}{T} + m^2 \Big) \Gamma_{\epsilon,m,T,P}\Big]\\
	&\quad + {\Tr}_{\mathcal F(Q\mathfrak h)} \Big[\Big(\Big(\frac{d\Gamma(Q)}{T}\Big)^2 -2m\frac{d\Gamma(Q)}{T}\Big) \Gamma_{\epsilon,m,T,Q}\Big]
	\end{align*}
	where the inequality uses the fact that $d\Gamma (P) d\Gamma (Q) \geq 0$ as an operator because $P,Q\geq 0$ and $PQ=0$. Observe next that
	$$
	\begin{aligned}
		{\Tr}_{\mathcal F(P\mathfrak h)} &\Big[\Big(\Big(\frac{d\Gamma(P)}{T}\Big)^2 -2m\frac{d\Gamma(P)}{T} + m^2 \Big) \Gamma_{\epsilon,m,T,P}\Big] \\
		&= \frac{2}{T^2} {\Tr}_{\mathcal F(P\mathfrak h)} \Big[ \binom{d\Gamma(P)}{2} \Gamma_{\epsilon,m,T,P}\Big] + \frac{1}{T}\Big(\frac{1}{T}-2m\Big) {\Tr}_{\mathcal F(P\mathfrak h)} \Big[ \binom{d\Gamma(P)}{1} \Gamma_{\epsilon,m,T,P}\Big] + m^2 \\
		&=2 {\Tr}_{\mathcal F(P\mathfrak h)} \Big[\frac{1}{T^2} \Gamma^{(2)}_{\epsilon,m,T,P} \Big] + \Big(\frac{1}{T}-2m\Big) {\Tr}_{\mathcal F(P\mathfrak h)} \Big[\frac{1}{T} \Gamma^{(1)}_{\epsilon,m,T,P} \Big] + m^2.
	\end{aligned}
	$$
	By omitting a non-negative term and noting that
	$$
	{\Tr}_{\mathcal F(Q\mathfrak h)} [d\Gamma(Q) \Gamma_{\epsilon,m,T,Q}] = {\Tr}[Q\Gamma^{(1)}_{\epsilon,m,T}],
	$$
	we obtain
	$$
	\begin{aligned}
		\frac{1}{\epsilon} {\Tr}&\Big[\Big(\frac{\mathcal N}{T}-m \Big)^2 \Gamma_{\epsilon,m,T} \Big] \\
		&\geq \frac{1}{\epsilon} \Big(2 {\Tr}_{\mathcal F(P\mathfrak h)} \Big[\frac{1}{T^2} \Gamma^{(2)}_{\epsilon,m,T,P} \Big] + \Big(\frac{1}{T}-2m\Big) {\Tr}_{\mathcal F(P\mathfrak h)} \Big[\frac{1}{T} \Gamma^{(1)}_{\epsilon,m,T,P} \Big] + m^2 \Big) -\frac{2m}{\epsilon T} {\Tr}[Q\Gamma^{(1)}_{\epsilon,m,T}].
	\end{aligned}
	$$
	By the de Finetti theorem (see \cite[Lemma 6.2 and Remark 6.4]{LewNamRou-15} or \cite[Theorem 5.8]{LewNamRou-20}), we have for all $k\geq 1$,
	\begin{align} \label{eq:deF}
	\int_{P \mathfrak h} |u^{\otimes k} \rangle \langle u^{\otimes k}| d\mu_{\epsilon,m,T,P}(u) =  \frac{k!}{T^k} \Gamma^{(k)}_{\epsilon,m,T,P} + \frac{k!}{T^k} \sum_{l=0}^{k-1} \binom{k}{l} \Gamma^{(l)}_{\epsilon,m,T,P} \otimes_s \mathds{1}_{(P\mathfrak h)^{k-l}}
	\end{align}
	and
	$$
	\begin{aligned}
	\Big\|\frac{k!}{T^k} \Gamma^{(k)}_{\epsilon,m,T,P} &- \int_{P \mathfrak h} |u^{\otimes k} \rangle \langle u^{\otimes k}| d\mu_{\epsilon,m,T,P}(u) \Big\|_{\mathfrak S^1((P\mathfrak h)^k)} \\
	& \leq \frac{1}{T^{k-l}} \sum_{l=0}^{k-1} \binom{k}{l}^2 \frac{(k-l+{\Tr}[P] -1)!}{({\Tr}[P]-1)!} {\Tr}_{\mathcal F (P \mathfrak h)} \Big[\Big(\frac{d\Gamma(P)}{T} \Big)^l \Gamma_{\epsilon,m,T,P} \Big].
	\end{aligned}
	$$
	In particular
	$$
	\frac{1}{T} \Gamma^{(1)}_{\epsilon,m,T,P} = \int_{P\mathfrak h} |u\rangle \langle u| d\mu_{\epsilon,m,T,P}(u) - \frac{1}{T} P
	$$
	and
	$$
	\frac{1}{T^2} \Gamma^{(2)}_{\epsilon,m,T,P} = \frac{1}{2} \int_{P\mathfrak h} |u^{\otimes 2} \rangle \langle u^{\otimes 2}| d\mu_{\epsilon,m,T,P}(u) - \frac{2}{T^2} \Gamma^{(1)}_{\epsilon,m,T,P} \otimes_s P -\frac{2}{T^2} P \otimes_s P.
	$$
	Thus
	$$
	\begin{aligned}
		2 {\Tr}_{\mathcal F(P\mathfrak h)} &\Big[\frac{1}{T^2} \Gamma^{(2)}_{\epsilon,m,T,P} \Big] + \Big(\frac{1}{T}-2m\Big) {\Tr}_{\mathcal F(P\mathfrak h)} \Big[\frac{1}{T} \Gamma^{(1)}_{\epsilon,m,T,P} \Big] + m^2 \\
		&= \int_{P\mathfrak h} \langle u, u \rangle^2 d\mu_{\epsilon,m,T,P}(u) -\frac{4}{T^2} {\Tr}[P] {\Tr}[\Gamma^{(1)}_{\epsilon,m,T,P}] -\frac{4}{T^2} ({\Tr}[P])^2 \\
		&\quad + \Big(\frac{1}{T}-2m\Big) \int_{P\mathfrak h} \langle u, u \rangle d\mu_{\epsilon,m,T,P}(u) +\Big(\frac{1}{T}-2m\Big) \frac{1}{T}{\Tr}[P] +m^2 \\
		&\geq \int_{P\mathfrak h} (\langle u, u \rangle-m)^2 d\mu_{\epsilon,m,T,P}(u) -\frac{4}{T^2} {\Tr}[P] {\Tr}[\Gamma^{(1)}_{\epsilon,m,T,P}] -\frac{4}{T^2} ({\Tr}[P])^2 -2 m \frac{1}{T}{\Tr}[P].
	\end{aligned}
	$$
	Since (by \eqref{GepsmT-k})
	$$
	\frac{1}{T}{\Tr}[\Gamma^{(1)}_{\epsilon,m,T,P}] = \frac{1}{T} {\Tr}[\Gamma_{\epsilon,m,T}] \leq C(m)
	$$
	and ${\Tr}[P]=\Lambda^{1/2+1/s} \ll T$, we have
	$$
	\begin{aligned}
	2 {\Tr}_{\mathcal F(P\mathfrak h)} \Big[\frac{1}{T^2} \Gamma^{(2)}_{\epsilon,m,T,P} \Big] &+ \Big(\frac{1}{T}-2m\Big) {\Tr}_{\mathcal F(P\mathfrak h)} \left[\frac{1}{T} \Gamma^{(1)}_{\epsilon,m,T,P} \right] + m^2 \\
	&\geq \int_{P\mathfrak h} (\langle u, u \rangle-m)^2 d\mu_{\epsilon,m,T,P}(u) - C(m) \frac{{\Tr}[P]}{T}.
	\end{aligned}
	$$
	In particular, we have
	$$
	\frac{1}{\epsilon} {\Tr}\Big[\Big( \frac{\mathcal N}{T}-m\Big)^2 \Gamma_{\epsilon,m,T} \Big] \geq \frac{1}{\epsilon} \int_{P\mathfrak h} (\langle u, u\rangle-m)^2 d\mu_{\epsilon,m,T,P}(u) - C(m) \frac{{\Tr}[P]}{\epsilon T} -\frac{2m}{\epsilon T} {\Tr}[Q\Gamma^{(1)}_{\epsilon,m,T}].
	$$
	Thanks to \eqref{uppe-boun-parti}, the same argument as in the proof of~\cite[Lemma~8.2]{LewNamRou-15} yields
	$$
	\frac{1}{T} \Gamma^{(1)}_{\epsilon,m,T} \leq C(m) \log \Big(\frac{1}{\epsilon}\Big) h^{-1},
	$$
	hence
	$$
	\frac{1}{T} {\Tr}[Q \Gamma^{(1)}_{\epsilon,m,T}] \leq C(m) \log \Big(\frac{1}{\epsilon} \Big) {\Tr}[Qh^{-1}] \leq C(m,q) \log \Big(\frac{1}{\epsilon} \Big) \Lambda^{-(1-q)}
	$$
	for any $q \in (1/2+1/s,1)$. It follows that
	$$
	\begin{aligned}
	\frac{1}{\epsilon} {\Tr}\Big[\Big( \frac{\mathcal N}{T}-m\Big)^2 \Gamma_{\epsilon,m,T}\Big] &\geq \frac{1}{\epsilon} \int_{P\mathfrak h} (\langle u, u\rangle-m)^2 d\mu_{\epsilon,m,T,P}(u) \\
	&\quad - C(m) \epsilon^{-1} T^{-1} \Lambda^{1/2+1/s} -C(m,q) \epsilon^{-1} \log\Big(\frac{1}{\epsilon}\Big) \Lambda^{-(1-q)}.
	\end{aligned}
	$$
	Since $\epsilon =T^{-a}$ and $\Lambda =T^b$, we have
	\begin{align} \label{intera-term}
	\begin{aligned}
	\frac{1}{\epsilon} {\Tr}\Big[\Big( \frac{\mathcal N}{T}-m\Big)^2 \Gamma_{\epsilon,m,T}\Big] &\geq \frac{1}{\epsilon} \int_{P\mathfrak h} (\langle u, u\rangle-m)^2 d\mu_{\epsilon,m,T,P}(u) \\
	&\quad - C(m) T^{-(1-a-b(1/2+1/s))} -C(m,q) T^{-(b(1-q)-a)} \log(T^a).
	\end{aligned}
	\end{align}
	To make the error small, we need
	$$
	a+b(1/2+1/s) <1, \quad a <b(1-q).
	$$
	These conditions are fulfilled if
	\begin{align} \label{cond-b-2}
	\frac{a}{1-q}<b<\frac{1-a}{1/2+1/s}.
	\end{align}
	Again, for $0<a<\frac{s-2}{4s}$, we can always find $q\in (1/2+1/s,1)$ and $b$ so that \eqref{cond-b-2} is satisfied. Collecting \eqref{entro-term} and \eqref{intera-term} gives~\eqref{L1-conv-proof-1}.

	\medskip
	
	\noindent\textbf{Step 3. Comparison of classical measures.} The lower bound~\eqref{L1-conv-proof-1} involves objects that are similar to the classical measures and partition functions of Section~\ref{sec:def cond meas}, but not exactly identical. We bridge this gap by proving that, with $\epsilon= T^{-a}$ and $\Lambda = T^b$, assuming
		\begin{align} \label{cond-b-3}
		0<b<\frac{1-a}{3/2+1/s},
		\end{align}
		then we have
		\begin{align} \label{L1-conv-proof-2}
		\|\tilde{\mu}_{\epsilon,m,T,P} - \mu_{\epsilon,m}\|_{L^1(P\mathfrak h)} \xrightarrow[T\to \infty]{} 0.
		\end{align}  
	and 
	\begin{equation}\label{eq:part tilde}
			\Big|\frac{1}{\tilde{z}_{\epsilon,m,T,P}} - \frac{1}{z_{\epsilon,m,P}}\Big| \leq C(m) \frac{\Lambda^{3/2+1/s}}{T \epsilon}.
	\end{equation}
	with $$
		z_{\epsilon,m,P}:=  \int_{P\mathfrak h} \exp \left( -\frac{1}{\epsilon} (\langle u, u \rangle -m)^2\right) d\mu_{0,\Lambda}(u).
		$$
		We argue as in \cite[Lemma 9.3]{LewNamRou-20}. First we have
		$$
		\begin{aligned}
			|\tilde{z}_{\epsilon,m,T,P}-z_{\epsilon,m,P}|& =\Big| \int_{P\mathfrak h} \exp \left( -\frac{1}{\epsilon} (\langle u, u \rangle -m)^2\right) (d\mu_{0,T,P}(u)- d\mu_{0,\Lambda}(u)) \Big| \\
			&\leq \|\mu_{0,T,P}-\mu_{0,\Lambda}\|_{L^1(P\mathfrak h)}.
		\end{aligned}
		$$
		Recall that, with $K = \mathrm{rank} (P)$
		$$
		d\mu_{0,\Lambda}(u) = \prod_{j=1}^K \frac{\lambda_j}{\pi} \exp(-\lambda_j|\alpha_j|^2) d\alpha_j, \quad u = \sum_{j=1}^K \alpha_j u_j
		$$
		and
		$$
		\begin{aligned}
			d\mu_{0,T,P}(u) &= \left(\frac{T}{\pi}\right)^K \left\langle \xi(u\sqrt{T}), \Gamma_{0,T,P} \xi(u\sqrt{T})\right\rangle_{\mathcal F(P\mathfrak h)} du \\
			&= \Big(\frac{T}{\pi}\Big)^K \frac{1}{{\Tr}\Big[\exp \Big(-\frac{1}{T} d\Gamma(Ph) \Big) \Big]}\Big\langle \xi(u\sqrt{T}), \exp\Big(-\frac{1}{T} d\Gamma(Ph) \Big) \xi(u\sqrt{T})\Big\rangle_{\mathcal F(P\mathfrak h)} du.
		\end{aligned}
		$$
		By the Peierls-Bogoliubov inequality, we have (see \eqref{xi-uT-P}) for $u \in P\mathfrak h$,
		$$
		\begin{aligned}
		\Big\langle \xi(u\sqrt{T}), \exp\Big(-\frac{1}{T} d\Gamma(Ph) \Big) \xi(u\sqrt{T})\Big\rangle &\geq \exp \Big(-\langle \xi(u\sqrt{T}), \frac{1}{T}d\Gamma(Ph) \xi(u\sqrt{T}) \rangle \Big) \\
		& =\exp \Big(-\sum_{j=1}^K \lambda_j |\alpha_j|^2\Big).
		\end{aligned}
		$$
		By \eqref{Z0T-P}, we get
		$$
		d\mu_{0,T,P}(u) \geq \Big(\prod_{j=1}^K \frac{T}{\lambda_j} (1-e^{-\lambda_j/T})\Big) d\mu_{0,\Lambda}(u).
		$$
		Using 
		$$
		\frac{1-e^{-x}}{x} \geq 1-\frac{x}{2}, \quad \forall x > 0
		$$
		and
		$$
		\prod_{j=1}^K (1-a_j) \geq 1- \sum_{j=1}^K a_j, \quad \forall 0<a_j<1,
		$$
		we see that
		$$
		\begin{aligned}
			\prod_{j=1}^K \frac{T}{\lambda_j}(1-e^{-\lambda_j/T}) &\geq \prod_{j=1}^K \Big(1-\frac{\lambda_j}{2T}\Big) \\
			&\geq 1- \frac{1}{2T}\sum_{j=1}^K\lambda_j \\
			&\geq 1- \frac{1}{2T} \Lambda \cdot \#\{\lambda_j : \lambda_j \leq \Lambda\} \\
			&\geq 1- \frac{C}{T} \Lambda^{\frac{3}{2}+\frac{1}{s}}
		\end{aligned}
		$$
		hence
		$$
		d\mu_{0,T,P}(u) \geq \Big(1- \frac{C}{T} \Lambda^{3/2+1/s}\Big) d\mu_{0,\Lambda}(u).
		$$
		This shows that
		$$
		(d\mu_{0,T,P}(u)-d\mu_{0,\Lambda}(u))_- \leq \frac{C}{T} \Lambda^{3/2+1/s} d\mu_{0,\Lambda}(u),
		$$
		where $f_-=\max(-f,0)$ is the negative part. Integrating over $u \in P\mathfrak h$, we get
		$$
		\int_{P\mathfrak h} (d\mu_{0,T,P}(u)-d\mu_{0,\Lambda}(u))_-  \leq  \frac{C}T \Lambda^{3/2+1/s}.
		$$
		Since $\mu_{0,T,P}$ and $\mu_{0,\Lambda}$ are probability measures and $f=f_+-f_-$, we have
		$$
		0= \int_{P\mathfrak h} d\mu_{0,T,P}(u)- d\mu_{0,\Lambda}(u) = \int_{P\mathfrak h} (d\mu_{0,T,P}(u)-d\mu_{0,\Lambda}(u))_+ - \int_{P\mathfrak h} (d\mu_{0,T,P}(u)-d\mu_{0,\Lambda}(u))_-.
		$$
		Because $|f|=f_+ + f_-$, we obtain
		$$
		\int_{P\mathfrak h} |d\mu_{0,T,P}(u)-d\mu_{0,\Lambda}(u)| \leq 2 \int_{P\mathfrak h} (d\mu_{0,T,P}(u)-d\mu_{0,\Lambda}(u))_- \leq \frac{C}T \Lambda^{3/2+1/s}.
		$$
		This shows that
		$$
		|\tilde{z}_{\epsilon,m,T,P}-z_{\epsilon,m,P}| \leq \frac{C}{T} \Lambda^{3/2+1/s}.
		$$
		Since $z_{\epsilon,m,P} \geq C(m) \sqrt{\epsilon}$ (cf Section~\ref{sec:def cond meas}), we infer that
		$$
		\begin{aligned}
			\tilde{z}_{\epsilon,m,T,P} &\geq z_{\epsilon,m,P} -\frac{C}{T}\Lambda^{3/2+1/s} \\
			&\geq C(m)\sqrt{\epsilon} 
		\end{aligned}
		$$
		because of $\Lambda^{3/2+1/s} \ll T\sqrt{\epsilon}$ due to \eqref{cond-b-3}. In particular, we have
		$$
		\begin{aligned}
			\Big|\frac{1}{\tilde{z}_{\epsilon,m,T,P}} - \frac{1}{z_{\epsilon,m,P}}\Big| &= \frac{|\tilde{z}_{\epsilon,m,T,P}-z_{\epsilon,m,P}|}{\tilde{z}_{\epsilon,m,T,P} z_{\epsilon,m,P}} \\
			&\leq C(m) \frac{\Lambda^{3/2+1/s}}{T \epsilon}.
		\end{aligned}
		$$
		We now estimate
		$$
		\begin{aligned}
			\|\tilde{\mu}_{\epsilon,m,T,P}-\mu_{\epsilon,m,P}\|_{L^1(P\mathfrak h)} &\leq \Big|\frac{1}{\tilde{z}_{\epsilon,m,T,P}}-\frac{1}{z_{\epsilon,m,P}}\Big| \int_{P\mathfrak h} \exp \Big(-\frac{1}{\epsilon}(\langle u, u\rangle -m)^2\Big) d\mu_{0,T,P}(u) \\
			&\quad + \int_{P\mathfrak h} \frac{1}{z_{\epsilon,m,P}} \exp\Big(-\frac{1}{\epsilon}(\langle u,u\rangle-m)^2\Big) |d\mu_{0,T,P}(u)-d\mu_{0,\Lambda}(u)| \\
			&\leq \Big|\frac{1}{\tilde{z}_{\epsilon,m,T,P}}-\frac{1}{z_{\epsilon,m,P}}\Big| + \frac{1}{z_{\epsilon,m,P}} \|\mu_{0,T,P}-\mu_{0,\Lambda}\|_{L^1(P\mathfrak h)} \\
			&\leq C(m) \frac{\Lambda^{3/2+1/s}}{T \epsilon} + C(m)\frac{\Lambda^{3/2+1/s}}{T \sqrt{\epsilon}} \\
			&\leq C(m) \frac{\Lambda^{3/2+1/s}}{T \epsilon} \\
			&= C(m) T^{-(1-a - b(3/2+1/s))} \\
			& \xrightarrow[T\to \infty]{} 0
		\end{aligned}
		$$
		thanks to \eqref{cond-b-3}.

		\medskip
		
		\noindent\textbf{Step 4. Convergence of lower symbols and density matrices.} Combining~\eqref{uppe-boun-log} with~\eqref{L1-conv-proof-1} we conclude that 
		$$
	\lim_{T\to \infty} \mathcal H_{\text{cl}}(\mu_{\epsilon,m,T,P}, \tilde{\mu}_{\epsilon,m,T,P}) =0
	$$
	provided that the conditions \eqref{cond-b-1}, \eqref{cond-b-2} and \eqref{cond-b-3} are satisfied simultaniously. These conditions require
	$$
	\frac{a}{1-q} <b<\frac{1-a}{3/2+1/s}
	$$
	for some $q \in (1/2+1/s,1)$. This is possible if we impose
	$$
	0<a<\frac{1-q}{5/2+1/s-q}.
	$$
	Since $q$ can take any value in $(1/2+1/s,1)$, we can match this condition if
	$$
	0<a<\frac{1}{2}\left(\frac{1}{2}-\frac{1}{s}\right).
	$$
	Using Pinsker's inequality (see \cite{CarLie-14})
	$$
	\mathcal H_{\text{cl}}(\mu,\nu) \geq \frac{1}{2} (\|\mu-\nu\|_{L^1})^2,
	$$
	we get
	\begin{align} \label{L1-conv-proof-3}
		\|\mu_{\epsilon,m,T,P}-\tilde{\mu}_{\epsilon,m,T,P}\|_{L^1(P\mathfrak h)} \xrightarrow[T\to \infty]{} 0.
	\end{align}
	Combining with \eqref{L1-conv-proof-2} we deduce
        \begin{align} \label{L1-conv-proof-4}
		\|\mu_{\epsilon,m,T,P}-\mu_{\epsilon,m}\|_{L^1(P\mathfrak h)} \xrightarrow[T\to \infty]{} 0.
	\end{align}
Let now $K$ be a bounded operator on $(P\mathfrak h)^k$. Let $\delta>0$ and set $f(u) = \langle u^{\otimes k}, K u^{\otimes k}\rangle$. For $A>0$ to be chosen later, we estimate
	$$
	\begin{aligned}
		\Big|\int_{P\mathfrak h} f(u) (d\mu_{\epsilon,m,T,P}(u)-d\mu_{\epsilon,m}(u)) \Big| &\leq \Big|\int_{u \in P\mathfrak h, \|u\|\leq A} f(u) (d\mu_{\epsilon,m,T,P}(u)-d\mu_{\epsilon,m}(u)) \Big| \\ 
		&\quad +\Big|\int_{u \in P\mathfrak h,\|u\|>A} f(u) d\mu_{\epsilon,m,T,P}(u)\Big| + \Big|\int_{u \in P\mathfrak h, \|u\|>A} f(u) d\mu_{\epsilon,m}(u)\Big|.
	\end{aligned}
	$$
	Since $|f(u)| \leq \|K\|_{\infty} \|u\|^{2k}$, we have
	$$
	\begin{aligned}
		\Big|\int_{P\mathfrak h} f(u) (d\mu_{\epsilon,m,T,P}(u)-d\mu_{\epsilon,m}(u)) \Big| &\leq \|K\|_\infty A^{2k} \|\mu_{\epsilon,m,T,P} -\mu_{\epsilon,m}\|_{L^1(P\mathfrak h)} \\
		&\quad + A^{-2k} \left( \int_{P\mathfrak h} \|u\|^{4k} d\mu_{\epsilon,m,T,P}(u) + \int \|u\|^{4k} d\mu_{\epsilon,m}(u)\right).
	\end{aligned}
	$$
	By the de Finetti theorem~\eqref{eq:deF} and \eqref{NTk-GepsmT-est}, we have
	$$
	\int_{P\mathfrak h} \|u\|^{4k} d\mu_{\epsilon,m,T,P}(u) \leq C(k,m)
	$$
	uniformly in $\epsilon>0$. On the other hand, from Proposition~\ref{pro:class lim relax} we have
	$$
	\int \|u\|^{4k} d\mu_{\epsilon,m}(u) \leq C(k,m)
	$$
	uniformly in $\epsilon>0$ small. 
	
	By choosing $A=A(\delta)>0$ large, we have
	$$
	A^{-2k} \left( \int_{P\mathfrak h} \|u\|^{4k} d\mu_{\epsilon,m,T,P}(u) + \int \|u\|^{4k} d\mu_{\epsilon,m}(u)\right) <\delta/2.
	$$
	With $A$ chosen above, we can use \eqref{L1-conv-proof-4} to make
	$$
	\|K\|_\infty A^{2k} \|\mu_{\epsilon,m,T,P} -\mu_{\epsilon,m}\|_{L^1(P\mathfrak h)} <\delta/2.
	$$
	This shows that, for any bounded operator $K$
	$$
	\Big|\int_{P\mathfrak h} \langle u^{\otimes k}, K u^{\otimes k} \rangle d\mu_{\epsilon,m,T,P}(u) - \int_{P\mathfrak h} \langle u^{\otimes k}, K u^{\otimes k} \rangle d\mu_{\epsilon,m}(u)\Big| \xrightarrow{T\to \infty } 0.
	$$
	Combining with the de Finetti theorem~\eqref{eq:deF} and \eqref{NTk-GepsmT-est} we deduce that 
	$$
	\Big\|\frac{k!}{T^k} P^{\otimes k} \Gamma^{(k)}_{\epsilon,m,T} P^{\otimes k} - \int |(Pu)^{\otimes k} \rangle \langle (Pu)^{\otimes k}| d\mu_{\epsilon,m}(u) \Big\|_{\mathfrak S^1((P\mathfrak h)^k)} \to 0.
	$$
	It follow that for any fixed $\phi_k\in \gh^k$
	$$\frac{k!}{T^k} \left\langle\phi_k, \Gamma^{(k)}_{\epsilon,m,T} \phi_k\right\rangle - \int \left|\left\langle\phi_k, u^{\otimes k} \right\rangle\right|^2 d\mu_{\epsilon,m}(u) \to 0$$
	and~\eqref{limi-cano 2} follows from combining with uniform bounds in trace-class norm that we have already used. 
	\end{proof}
	
\subsection{Convergence of the relaxed Gibbs state to the canonical Gibbs state}\label{sec:relax}
	
	We now study the limit, when $\epsilon \to 0$, of the relaxed free canonical Gibbs state $\Gamma_{\epsilon,m,T}$, and prove Proposition~\ref{pro:relax}. 
	
	We have the following immediate relations:
	
	\begin{lemma}[\textbf{Relaxed Gibbs state and canonical Gibbs state}] \label{lem-N-sector}\mbox{}\\
		We have
		$$
		\Gamma_{\epsilon,m,T} = \bigoplus_{N\geq 0} (\Gamma_{\epsilon,m,T})_N = \bigoplus_{N\geq 0} \Gamma^c_{N,T,0} a^\eps_N
		$$
		with
		$$
		\Gamma^c_{N,T,0} =\frac{1}{Z^c_{N,T,0}} \exp\left(-\frac{1}{T} H_{N,0}\right)
		$$
		and
		$$
		a^\eps_N=\frac{1}{Z_{\epsilon,m,T}}\exp\left( -\frac{1}{T} \left(F^c_0(N) +\frac{T}{\epsilon}\left(\frac{N}{T}-m \right)^2 \right)\right),
		$$
		where
		$$
		Z^c_{N,T,0} = {\Tr}\left[\exp\left( -\frac{1}{T} H_{N,0}\right)\right], \quad F^c_0(N) = -T\log (Z^c_{N,T,0}).
		$$
		In particular,
		\begin{align} \label{Z-eps-m-T}
		Z_{\epsilon,m,T} = \sum_{N\geq 0} \exp\left( -\frac{1}{T} \left(F^c_0(N) +\frac{T}{\epsilon}\left(\frac{N}{T}-m \right)^2 \right)\right).
		\end{align}
	\end{lemma}
	
	We will need to control the dependence on the particle number of the free canonical Gibbs state:
	
	\begin{lemma}[\textbf{Dependence of the canonical ensemble on the particle number}]\label{lem-cano-diff}\mbox{}\\
	For each $\phi_k \in \mathfrak h^k$, there exists $C>0$ such that 
	$$\left|\left\langle \phi_k | (\Gamma^c_{N,T,0})^{(k)} - (\Gamma^c_{N+1,T,0})^{(k)}| \phi_k\right\rangle\right| \leq C N^{k-1}.$$
	In particular, for $M>N$,
	$$\left|\left\langle \phi_k | (\Gamma^c_{N,T,0})^{(k)} - (\Gamma^c_{M,T,0})^{(k)}| \phi_k\right\rangle\right| \leq C N^{k-1} (M-N).$$
	\end{lemma}
	
	We need the following combinatorial lemma from~\cite[Combinatorial Proposition on Page 102]{Cannon-73}:
	
	\begin{lemma}[\textbf{Combinatorics}] \label{lem-Cannon-73}\mbox{}\\
	Given a non-negative integer $N$ and a sequence of non-negative integers $(g_j) \in \mathbb N^{\mathbb N}$ with $|g|=2N+1$ we let 
	$$S_N = \left\{(n_j)\in \mathbb N^{\mathbb N} : |n|= N, n_j \leq g_j, \forall j \right\}.$$
	Then there exists a bijection $\Phi: S_N \to S_{N+1}$ such that for all $n \in S_N$, there exists $j(n)$ so that
	$$(\Phi (n))_j =  \left\{ \begin{array}{ll}  n_j &\text{if } j \ne j(n), \\ n_{j(n)}+1 &\text{if } j= j(n).\end{array} \right. $$
	\end{lemma}

	\begin{proof}[Proof of Lemma~\ref{lem-cano-diff}]
	Since the eigenfunctions $(u_j)_{j\geq 1}$ form an orthonormal basis of $\mathfrak h$, $(u_{j_1} \otimes_s ...\otimes_s u_{j_k})_{j_1\leq...\leq j_k}$ is an orthogonal basis of $\mathfrak h^k$. It suffices to consider $\phi_k = u_{j_1} \otimes_s ...\otimes_s u_{j_k}$. We have
	$$
	\left\langle \phi_k | (\Gamma^c_{N,T,0})^{(k)}|\phi_k\right\rangle = \text{Tr}\left[\ada_{j_1}...\ada_{j_k} a_{j_1}...a_{j_k} \Gamma^c_{N,T,0}\right],
	$$
	where $a_{j_i}=a(u_{j_i})$ and $\ada_{j_i} = \ada(u_{j_i})$ are the annihilation and creation operators. Using the canonical commutation relation
	$$
	[a(f),a(g)]=[\ada(f), \ada(g)] =0, \quad [a(f), \ada(g)]=\langle f,g\rangle,
	$$
	we have
	\begin{align}\label{eq:normal order}
		\text{Tr}\left[\ada_{j_1}...\ada_{j_k} a_{j_1}...a_{j_k} \Gamma^c_{N,T,0}\right] &= \text{Tr}\left[\ada_{j_1} a_{j_1} \ada_{j_2} a_{j_2}...\ada_{j_k} a_{j_k} \Gamma^c_{N,T,0}\right] \nonumber \\
		&\quad + \text{ lower order terms in } k
	\end{align}
	where the lower order terms come from normal ordering and thus consist of polynomials of degree at most $k-1$, so that 
	\begin{equation}\label{eq:normal order 2}
	 \left| \text{ lower order terms in } k\right| \leq C N^{k-1}.
	\end{equation}
    We thus focus on the first term in~\eqref{eq:normal order}.
	
	One can represent all the possible configurations of occupation of the energy levels by the sequences $n=(n_j) \in \mathbb N^{\mathbb N}$ such that $|n|:=\sum_j n_j = N$. Since the eigenvalues of $H_{N,0}$ are $\sum_{j} n_j \lambda_j$ with $|n| = N$, the partition function is given by
	$$
	Z^c_{N,T,0} = \sum_{|n| =N} \exp \left(-\frac{1}{T} 
	\sum_j \lambda_j n_j\right).
	$$
	We also have
	$$
	\text{Tr}\left[\ada_{j_1} a_{j_1} \ada_{j_2} a_{j_2}...\ada_{j_k} a_{j_k} \Gamma^c_{N,T,0}\right] = \frac{1}{Z^c_{N,T,0}} \sum_{|n| = N} \exp \left(-\frac{1}{T} 
	\sum_j \lambda_j n_j\right) n_{j_1}...n_{j_k}.
	$$
	For a sequence $n=(n_j) \in \mathbb N^{\mathbb N}$, we denote
	$$
	f(n) = n_{j_1}...n_{j_k}, \quad G(n) = \exp \left(-\frac{1}{T} \sum_j \lambda_j n_j\right).
	$$
	In particular,
	$$
	Z^c_{N,T,0} = \sum_{|n|=N} G(n)
	$$
	and
	$$
	\text{Tr}\left[\ada_{j_1} a_{j_1} \ada_{j_2} a_{j_2}...\ada_{j_k} a_{j_k} \Gamma^c_{N,T,0}\right] = \frac{1}{Z^c_{N,T,0}} \sum_{|n| = N} G(n) f(n).
	$$
	It follows that
	$$
	\begin{aligned}
		\text{I}:&=\left| \text{Tr}\left[\ada_{j_1} a_{j_1} \ada_{j_2} a_{j_2}...\ada_{j_k} a_{j_k} (\Gamma^c_{N,T,0}-\Gamma^c_{N+1,T,0})\right] Z^c_{N,T,0} Z^c_{N+1,T,0} \right| \\
		&=\left|\sum_{|m|=N+1} G(m) \sum_{|n|=N} G(n)f(n) - \sum_{|n|=N} G(n) \sum_{|m|=N+1} G(m)f(m) \right| \\
		&= \left| \sum_{|n|=N} \sum_{|m|=N+1} G(n+m) (f(n)-f(m))\right| \\
		&= \left| \sum_{|g|=2N+1} G(g) \Big( \sum_{n\leq g \atop |n|=N} f(n) - \sum_{m\leq g \atop |m|=N+1} f(m)\Big)\right|,
	\end{aligned}
	$$
	where we set $g=n+m$ and $n\leq g$ means $n_j\leq g_j$ for all $j$. 
	
	Using Lemma~\ref{lem-Cannon-73}, we get
	$$
	\begin{aligned}
		\text{I}=\sum_{|g|=2N+1} G(g) \sum_{n\leq g \atop |n|=N} \left(f(n) -f(\Phi (n))\right).
	\end{aligned}
	$$
	Since $\Phi(n)_j$ is $n_j$ for all $j$ except a single one where $\Phi(n)_j=n_j+1$ and $n_j \leq N$ for all $j$, we have
	$$
	|f(n)-f(\Phi(n))| \leq N^{k-1}.
	$$
	We obtain
	$$
	\begin{aligned}
		\text{I} &\leq N^{k-1} \sum_{|g|=2N+1} G(g) \sum_{n\leq g \atop |n|=N} 1 \\
		&= N^{k-1} \sum_{|n|=N} G(n)\sum_{|m|=N+1} G(m) \\
		&= N^{k-1} Z^c_{N,T,0} Z^c_{N+1,T,0}.
	\end{aligned}
	$$
	Dividing by $Z^c_{N,T,0} Z^c_{N+1,T,0}$ this shows that
	$$
	\left| \text{Tr}\left[\ada_{j_1} a_{j_1} \ada_{j_2} a_{j_2}...\ada_{j_k} a_{j_k} (\Gamma^c_{N,T,0}-\Gamma^c_{N+1,T,0})\right]\right| \leq N^{k-1}.
	$$
	Returning to~\eqref{eq:normal order}-\eqref{eq:normal order 2} completes the proof.
	\end{proof}
	
	We may now complete the
	
	\begin{proof}[Proof of Proposition~\ref{pro:relax}]
	From Lemma~\ref{lem-NTk-GepsmT-est} we deduce that the trace of $T^{-k} \Gamma_{\epsilon,m,T}^{(k)}$ is uniformly bounded. We can thus reduce the convergence~\eqref{limi-T-eps} to the proof of 
		\begin{align} \label{weak-conv-phi}
		\frac{1}{T^k} \left\langle \phi_k \Big|\Gamma_{\epsilon,m,T}^{(k)}-(\Gamma^c_{mT,T,0})^{(k)} \Big| \phi_k \right\rangle \xrightarrow[T\to \infty]{} 0
		\end{align}
		for all $\phi_k \in \mathfrak h^k$.	Since $(u_{j(1)} \otimes_s ... \otimes_s u_{j(k)})_{j(1)\leq ...\leq j(k)}$ forms an orthogonal basis of $\mathfrak h^k$, it suffices to show the above convergence with $\phi_k = u_{j(1)} \otimes_s ... \otimes_s u_{j(k)}$. By the definition, we have
		$$
		\Gamma^{(k)}_{\epsilon,m,T} = \sum_{N\geq k} a_N^\eps (\Gamma^c_{N,T,0})^{(k)}.
		$$
		For $\Delta=\Delta(T) \ll T$ a large parameter to be chosen later, we write
		$$
		\Gamma^{(k)}_{\epsilon,m,T} - (\Gamma^c_{mT,T,0})^{(k)} = \sum_{N\geq k \atop |N-mT|\leq \Delta} a_N^\eps (\Gamma^c_{N,T,0})^{(k)} - (\Gamma^c_{mT,T,0})^{(k)} + \sum_{N\geq k \atop |N-mT| >\Delta} a_N^\eps (\Gamma^c_{N,T,0})^{(k)}.
		$$
		Thus
		$$
		\begin{aligned}
			\frac{1}{T^k} \Big\langle \phi_k \Big|\Gamma^{(k)}_{\epsilon,m,T}-(\Gamma^c_{mT,T,0})^{(k)} \Big| \phi_k \Big\rangle &= \frac{1}{T^k} \Big\langle \phi_k \Big| \sum_{N\geq k \atop |N-mT|\leq \Delta} a_N^\eps (\Gamma^c_{N,T,0})^{(k)} - (\Gamma^c_{mT,T,0})^{(k)} \Big| \phi_k \Big\rangle \\
			&\quad + \frac{1}{T^k} \Big\langle \phi_k \Big| \sum_{N\geq k \atop |N-mT|>\Delta} a_N^\eps (\Gamma^c_{N,T,0})^{(k)} \Big|\phi_k \Big\rangle \\
			&=: (\text{I}) + (\text{II}).
		\end{aligned}
		$$
		For (II), we have
		$$
		\begin{aligned}
			|(\text{II})| &\leq \frac{1}{T^k} \|\phi_k\|^2 \Big\| \sum_{N\geq k \atop |N-mT|>\Delta} a_N^\eps (\Gamma^c_{N,T,0})^{(k)} \Big\|_{\mathfrak S^1(\mathfrak h^k)} \\
			&\leq \frac{1}{T^k} \|\phi_k\|^2 \sum_{N\geq k \atop |N-mT|>\Delta} a_N^\eps \|(\Gamma^c_{N,T,0})^{(k)} \|_{\mathfrak S^1(\mathfrak h^k)}.
		\end{aligned}
		$$
		Recall that
		$$
		(\Gamma^c_{N,T,0})^{(k)} = \binom{N}{k} {\Tr}_{k+1\to N} [\Gamma^c_{N,T,0}].
		$$
		Thus
		$$
		\begin{aligned}
			|(\text{II})| &\leq \frac{1}{T^k} \|\phi_k\|^2 \sum_{N\geq k \atop |N-mT|>\Delta} a_N^\eps \binom{N}{k}.
		\end{aligned}
		$$
		We have
		$$
		\begin{aligned}
			\frac{1}{T^k} \sum_{N\geq k \atop |N-mT|>\Delta} a_N^\eps \binom{N}{k} &= \frac{1}{T^k} \sum_{N\geq k \atop |N-mT|>\Delta} \binom{N}{k} \frac{1}{Z_{\epsilon,m,T}} \exp \left(-\frac{1}{T} F^c_0(N) -\frac{1}{\epsilon} \left(\frac{N}{T}-m\right)^2\right) \\
			&\leq \frac{1}{T^k} \exp\left(-\frac{\Delta^2}{\epsilon T^2}\right) \frac{1}{Z_{\epsilon,m,T}} \sum_{N\geq k} \binom{N}{k} \exp\left(-\frac{1}{T}F^c_0(N)\right) \\ 
			&= \frac{1}{T^k} \exp\left(-\frac{\Delta^2}{\epsilon T^2}\right) \frac{1}{Z_{\epsilon,m,T}} \sum_{N\geq k} \binom{N}{k} \text{Tr}\left[\exp\left(-\frac{1}{T} H_{N,0}\right) \right]\\ 
		\end{aligned}
		$$
		Recall that
		$$
		\Gamma_{0,T} = \frac{1}{Z_{0,T}} \exp \left( -\frac{1}{T} d\Gamma(h)\right) = \frac{1}{Z_{0,T}} \bigoplus_{N\geq 0} \exp \left(-\frac{1}{T} H_{N,0}\right)
		$$
		and
		$$
		\text{Tr}[\Gamma_{0,T}^{(k)}] = \text{Tr}\left[ \binom{\mathcal N}{k} \Gamma_{0,T}\right] = \frac{1}{Z_{0,T}}\sum_{N\geq k} \binom{N}{k} \text{Tr}\left[\exp \left(-\frac{1}{T} H_{N,0}\right) \right].
		$$
		It follows that
		$$
		\sum_{N\geq k} \binom{N}{k} \text{Tr}\left[\exp\left( -\frac{1}{T} H_{N,0}\right) \right] = Z_{0,T} \text{Tr}[\Gamma^{(k)}_{0,T}]
		$$
		hence
		$$
		\frac{1}{T^k} \sum_{N\geq k \atop |N-mT|>\Delta} a_N^\eps \binom{N}{k} \leq \exp \left(-\frac{\Delta^2}{\epsilon T^2}\right) \frac{Z_{0,T}}{Z_{\epsilon,m,T}} \text{Tr}\left[ \frac{1}{T^k} \Gamma^{(k)}_{0,T}\right].
		$$
		From~\eqref{eq:class part relax} and~\eqref{limi-eps-parti} we deduce
		$$
		\frac{Z_{0,T}}{Z_{\epsilon,m,T}} \leq \frac{C(m)}{\sqrt{\epsilon}}
		$$
		and since we have (see e.g.~\cite[Section~2]{LewNamRou-15})
		$$
		\text{Tr}\left[ \frac{1}{T^k} \Gamma^{(k)}_{0,T}\right] \leq C(k).
		$$
		Thus we conclude that
		\begin{align} \label{term-II}
		|(\text{II})| \leq C(m,k) \|\phi_k\|^2  \exp \left(-\frac{\Delta^2}{\epsilon T^2}\right) \frac{1}{\sqrt{\epsilon}}.
		\end{align}
		
		For (I), we write
		$$
		\sum_{N\geq k \atop |N-mT|\leq \Delta} a_N^\eps (\Gamma^c_{N,T,0})^{(k)} = \sum_{N\geq k \atop |N-mT|\leq \Delta} a_N^\eps \left( (\Gamma^c_{N,T,0})^{(k)} - (\Gamma^c_{mT,T,0})^{(k)}\right) + \Big(\sum_{N\geq k \atop |N-mT| \leq \Delta} a_N^\eps \Big) (\Gamma^c_{mT,T,0})^{(k)}.
		$$
		Thus
		$$
		\begin{aligned}
			|(\text{I})| &\leq \frac{1}{T^k} \Big| \Big\langle \phi_k \Big| \sum_{N\geq k \atop |N-mT|\leq \Delta} a_N^\eps \left( (\Gamma^c_{N,T,0})^{(k)} - (\Gamma^c_{mT,T,0})^{(k)}\right) \Big| \phi_k \Big\rangle\Big|\\
			&\quad + \frac{1}{T^k} \Big| \Big\langle \phi_k \Big| \Big(\sum_{N\geq k \atop |N-mT| \leq \Delta} a_N^\eps -1\Big) (\Gamma^c_{mT,T,0})^{(k)} \Big| \phi_k \Big\rangle\Big| \\
			&=: (\text{I}_1) + (\text{I}_2).
		\end{aligned}
		$$
		We have
		$$
		|(\text{I}_2)| \leq \|\phi_k\|^2  \text{Tr}\Big[\frac{1}{T^k} (\Gamma^c_{mT,T,0})^{(k)}\Big] \Big|\sum_{N\geq k \atop |N-mT|\leq \Delta} a_N^\eps -1 \Big|.
		$$
		Since $\text{Tr}[\Gamma_{\epsilon, m,T}] =1$, we have that $\sum_{N} a_N^\eps =1$ and thus
		$$
		\sum_{N\geq k \atop |N-mT|\leq \Delta} a_N^\eps -1 = \sum_{N <k} a_N^\eps + \sum_{N\geq k \atop |N-mT| >\Delta} a_N^\eps \leq \sum_{N\geq 0 \atop |N-mT|>\Delta} a_N^\eps
		$$
		as $|N-mT| \geq mT - N \gg \Delta$ for $N <k$ and $T$ large. It follows that
		$$
		\begin{aligned}
			\Big| \sum_{N\geq k \atop |N-mT| \leq \Delta} a_N^\eps -1 \Big| &\leq \sum_{|N-mT| >\Delta} a_N^\eps \\
			&= \sum_{|N-mT|>\Delta} \frac{1}{Z_{\epsilon,m,T}} \exp \left(-\frac{1}{T} F^c_0(N) -\frac{1}{\epsilon} \left( \frac{N}{T}-m\right)^2\right) \\
			&\leq \frac{1}{Z_{\epsilon,m,T}} \exp \left(-\frac{\Delta^2}{\epsilon T^2}\right) \sum_{|N-mT| >\Delta} \exp\left(-\frac{1}{T} F^c_0(N)\right) \\
			&\leq \frac{1}{Z_{\epsilon,m,T}} \exp \left(-\frac{\Delta^2}{\epsilon T^2}\right) \sum_{N\geq 0} \text{Tr} \left[\exp \left(-\frac{1}{T} H_{N,0}\right) \right] \\
			&\leq \frac{Z_{0,T}}{Z_{\epsilon,m,T}} \exp \left(-\frac{\Delta^2}{\epsilon T^2}\right) \\
			&\leq C(m) \exp \left(-\frac{\Delta^2}{\epsilon T^2}\right) \frac{1}{\sqrt{\epsilon}}.
		\end{aligned}
		$$
		In addition, we have
		$$
		\text{Tr}\left[\frac{1}{T^k} (\Gamma^c_{mT,T,0})^{(k)} \right] = \frac{1}{T^k} \binom{mT}{k} \leq C(m,k).
		$$
		This shows that
		\begin{align}\label{term-I-2}
		|(\text{I}_2)| \leq C(m,k) \exp \left(-\frac{\Delta^2}{\epsilon T^2}\right)  \frac{1}{\sqrt{\epsilon}}.
		\end{align}
		For $(\text{I}_1)$, we use the weak convergence of $k$-particle density matrice of $\Gamma^c_{N,T,0}$ given in Lemma \ref{lem-cano-diff} to get
		\begin{align} \label{term-I-1}
		\begin{aligned}
			|(\text{I}_1)| &\leq \frac{1}{T^k} \sum_{N\geq k  \atop |N-mT| \leq \Delta} a_N^\eps | \langle \phi_k |(\Gamma^c_{N,T,0})^{(k)}- (\Gamma^c_{mT,T,0})^{(k)}| \phi_k \rangle| \\
			&\leq \frac{(mT)^{k-1}}{T^k} \sum_{N\geq k \atop |N-mT| \leq \Delta}  |N-mT| a_N^\eps \\
			&\leq m^{k-1}\frac{\Delta}{T}
		\end{aligned}
		\end{align}
		because $a_N^\eps$ are non-negative and $\sum_N a_N^\eps =1$. Collecting \eqref{term-II}, \eqref{term-I-2} and \eqref{term-I-1}, we obtain
		$$
		\Big| \frac{1}{T^k} \langle \phi_k |(\Gamma^{(k)}_{\epsilon,m,T}-(\Gamma^c_{mT,T,0})^{(k)})| \phi_k \rangle\Big| \leq C(m,k) \left(\exp \left(-\frac{\Delta^2}{\epsilon T^2}\right)  \frac{1}{\sqrt{\epsilon}} + \frac{\Delta}{T} \right).
		$$		
		Since $\epsilon=T^{-a}$ with $a>0$, we can take
		$$
		\Delta = T^{1-a/4} \ll T
		$$
		and obtain
		$$
		\Big| \frac{1}{T^k} \langle \phi_k |(\Gamma^{(k)}_{\epsilon,m,T}-(\Gamma^c_{mT,T,0})^{(k)})| \phi_k \rangle\Big| \leq C(m,k) \left(\exp(-T^{a/2}) T^{a/2} + T^{-a/4}\right) \to 0
		$$
		as $T\to \infty$. This proves \eqref{weak-conv-phi} and thus completes the proof of \eqref{limi-T-eps}.
	\end{proof}
		
	\section{Free energy upper bound}\label{sec:up}
	
	Here we prove the upper bound corresponding to~\eqref{eq:main ener}.
	
	\begin{proposition}[\textbf{Free-energy upper bound}]\label{pro:up bound}\mbox{}\\
	 Under the assumptions of Theorem~\ref{thm:main} we have that 
	\begin{equation}\label{uppe-boun-cano}
		\limsup_{T\to \infty} -\log \left(\frac{Z^c_{mT,T,g}}{Z^c_{mT,T,0}}\right) \leq -\log (z^r_m)
	\end{equation}
	with $z^r_m$ as in~\eqref{eq:class part}.
	\end{proposition}

	We obtain the above by inserting an appropriate trial state in the variational principle defining $-\log Z^c_{mT,T,g}$. 
	
	\subsection{Trial state and reduction to a finite dimensional estimate}\label{sec:trial state}

	Let $\Lambda>0$ be an energy cutoff. We define the projections
	$$
	P^-:= P_\Lambda, \quad P^+ := P_\Lambda^\perp, 
	$$
	where $P_\Lambda$ and $P_\Lambda^\perp$ are as in~\eqref{P-project}. We denote
	$$
	P^{\#} := \left\{
	\begin{array}{ll}
		\mathds{1} &\text{if } \# = \varnothing \\
		P^+ &\text{if } \# = + \\
		P^- &\text{if } \# = -
	\end{array}
	\right\}
	$$
	and
	$$
	h^{\#}= P^{\#} h P^{\#},\quad \mathfrak h^{\#} = P^{\#} \mathfrak h . 
	$$
	Let $n \in \mathbb N$. We denote the free canonical Gibbs state by
	$$
	\Gamma^{c,\#}_{n,T,0} = \frac{1}{Z^{c,\#}_{n,T,0}} \exp \left(-\frac{1}{T} H^{\#}_{n,0}\right)
	$$
	where
	$$
	H^{\#}_{n,0} = \sum_{j=1}^n h^{\#}_j.
	$$
	As usual $\Gamma^{c,\#}_{n,T,0}$ is the unique minimizer of the free (canonical) energy
	$$
	F^{c,\#}_0(n) = \inf \left\{\mathcal F^{c,\#}_{n,0}[\Gamma] := {\Tr}[H^{\#}_{n,0} \Gamma] + T {\Tr}[\Gamma \log \Gamma] : \Gamma \in \mathcal S((\mathfrak h^{\#})^{\otimes_s n})\right\},
	$$
	where $\mathcal S((\mathfrak h^{\#})^{\otimes_s n})$ is the set of states (non-negative self-adjoint operators with unit trace) on $(\mathfrak h^{\#})^{\otimes_s n}$. Moreover, we have
	$$
	F^{c,\#}_0(n) = \mathcal F^{c,\#}_{n,0}[\Gamma^{c,\#}_{n,T,0}] = -T \log (Z^{c,\#}_{n,T,0}).
	$$
	We will need a simple observation on the relationship between the full free Gibbs state~\eqref{eq:C ens} and the corresponding state on the projected spaces $\Gamma^{c,\#}_{n,T,0}$. It will replace the factorization property~\eqref{eq:trial state} of free grand-canonical Gibbs states used at length in~\cite{LewNamRou-15,LewNamRou-17,LewNamRou-20}. This uses the unitary map $\mathcal U$ from $\mathcal F (\mathfrak h^- \oplus \mathfrak h^+)$ to  $\mathcal F(\mathfrak h^-) \otimes \mathcal F(\mathfrak h^+)$ defined by its action
	\begin{equation}\label{eq:unitary map}
\cU a^\dagger (f) \cU ^* = a^\dagger (P^-f) \otimes \Id + \Id \otimes a^\dagger (P^+f)	 
	\end{equation}
	on creation/annihilation operators, see e.g.~\cite[Appendix~A]{HaiLewSol-09b}

	\begin{lemma}[\textbf{Factorization of free canonical Gibbs states}]\label{lem:fac Gibbs}\mbox{}\\
		For $N \in \mathbb N$, we have
		$$\mathcal U \Gamma^c_{N,T,0} \mathcal U^* = \bigoplus_{n=0}^N c_n \Gamma^{c,-}_{n,T,0} \otimes \Gamma^{c,+}_{N-n,T,0}$$
		where 
		\begin{align} \label{cn}
			c_n = \frac{1}{Z^c_{N,T,0}} Z^{c,-}_{n,T,0} Z^{c,+}_{N-n,T,0}.
		\end{align}
		Moreover, define for $M\leq N$
        \begin{equation}\label{eq:coeffs}
        D_M := \sum_{n= M} ^N c_n
        \end{equation}
        and set
		\begin{equation}\label{eq:choice M}
		M= N-T\delta
		\end{equation}
		for some $\delta>0$ small. Taking $N=mT$ for a fixed $m>0$ we have for $\Lambda>0$ large,
		\begin{align} \label{choi-D}
		D_M= 1+ O\left(\delta^{-1} \Lambda^{\frac{1}{s}-\frac{1}{2}}\right).
		\end{align}
	\end{lemma}
	
	\begin{proof}
	We write
	$$
	\Gamma^c_{N,T,0} = \frac{1}{Z^c_{N,T,0}} \exp \left(-\frac{1}{T} d\Gamma(h)\right) \Id_{\{\mathcal N= N\}},
	$$
	where 
	$$
	d\Gamma(h) = \bigoplus_{n\geq 1} H_{n,0}, \quad \mathcal N = d\Gamma(\Id)=\bigoplus_{n\geq 0} n \Id_{\mathfrak h^{\otimes_s n}}.
	$$
	Denote
	$$
	\mathcal N^{\#} = d\Gamma(P^{\#}).
	$$
	Since $\cU$ acts on creation/annihilation operators as in~\eqref{eq:unitary map}
	and $\cU \cU^* = \Id$, writing the $N$-particles space as the span of vectors $a^\dagger (f_1) \ldots a^\dagger (f_N), f_1,\ldots, f_N \in \gh$ we find that 
	$$ \mathcal U \Id_{\{\mathcal N=N\}} \mathcal U^* = \bigoplus_{n=0}^N \Id_{\{\mathcal N^- = n\}} \otimes \Id_{\{\mathcal N^+ = N-n\}}.$$
	Hence
	$$
	\begin{aligned}
		\mathcal U \Gamma^c_{N,T,0} \mathcal U^* &= \frac{1}{Z^c_{N,T,0}} \left( \mathcal U \exp\left(-\frac{1}{T} d\Gamma(h)\right) \mathcal U^*\right) \left( \mathcal U \Id_{\{\mathcal N=N\}} \mathcal U^*\right) \\
		&= \frac{1}{Z^c_{N,T,0}} \left(\exp\left(-\frac{1}{T} d\Gamma(h^-)\right) \otimes \exp\left(-\frac{1}{T} d\Gamma(h^+)\right)\right) \left(\bigoplus_{n=0}^N \Id_{\{\mathcal N^- = n\}} \otimes \Id_{\{\mathcal N^+ = N-n\}} \right).
	\end{aligned}
	$$
	Since these operators commute, we get
	$$
	\begin{aligned}
		\mathcal U \Gamma^c_{N,T,0} \mathcal U^* &= \frac{1}{Z^c_{N,T,0}} \bigoplus_{n=0}^N \left(\exp\left(-\frac{1}{T} d\Gamma(h^-)\right) \Id_{\{\mathcal N^-=n\}} \right) \otimes \left(\exp\left(-\frac{1}{T} d\Gamma(h^+)\right) \Id_{\{\mathcal N^+=N-n\}}\right) \\
		&= \frac{1}{Z^c_{N,T,0}} \bigoplus_{n=0}^N (Z^{c,-}_{n,T,0}\Gamma^{c,-}_{n,T,0}) \otimes (Z^{c,+}_{N-n,T,0} \Gamma^{c,+}_{N-n,T,0}) \\
		&= \bigoplus_{n=0}^N c_n \Gamma^{c,-}_{n,T,0} \otimes \Gamma^{c,+}_{N-n,T,0}.
	\end{aligned}
	$$
	We now turn to the proof of~\eqref{choi-D}. We have
		$$
		\begin{aligned}
		mT = {\Tr}\left[(\Gamma^c_{mT,T,0})^{(1)}\right] = {\Tr}\left[P^- (\Gamma^c_{mT,T,0})^{(1)}\right] + {\Tr}\left[P^+ (\Gamma^c_{mT,T,0})^{(1)}\right].
		\end{aligned}
		$$
		Since
		$$
		(\Gamma^c_{mT,T,0})^{(1)} = \sum_{n=0}^{mT} c_n \left((\Gamma^{c,-}_{n,T,0})^{(1)} + (\Gamma^{c,+}_{mT-n,T,0})^{(1)}\right)
		$$
		and by definition
		$$
		(\Gamma^{c,-}_{n,T,0})^{(1)} = P^- (\Gamma^c_{n,T,0})^{(1)}P^-, \quad (P^-)^2 = P^-, \quad P^- P^+ = 0,
		$$
		we have
		$$
		P^-(\Gamma^c_{mT,T,0})^{(1)}= \sum_{n=0}^{mT} c_n (\Gamma^{c,-}_{n,T,0})^{(1)}
		$$
		hence
		$$
		{\Tr}\left[P^-(\Gamma^c_{n,T,0})^{(1)}\right] = \sum_{n=0}^{mT}c_n {\Tr}\left[(\Gamma^{c,-}_{n,T,0})^{(1)}\right] = \sum_{n=0}^{mT} n c_n.
		$$
		In particular, we get
		\begin{align} \label{choi-D-proof-1}
		mT-\sum_{n=0}^{mT} n c_n = {\Tr}\left[P^+(\Gamma^c_{mT,T,0})^{(1)}\right].
		\end{align}
		Using~\eqref{DM-1-ope}, we have
		$$
		\begin{aligned}
		mT-\sum_{n=0}^{mT} nc_n &\leq C {\Tr}\left[P^+ (\Gamma^\nu_T)^{(1)}\right] \\
		& = \sum_{\lambda_j>\Lambda} \frac{1}{e^{(\lambda_j+\nu)/T}-1} \\
		& = T \sum_{\lambda_j >\Lambda} \frac{1}{T(e^{(\lambda_j+\nu)/T}-1)} \\
		&\leq T \sum_{\lambda_j>\Lambda} \frac{1}{\lambda_j+\nu}.
		\end{aligned}
		$$
		Here $\nu >-\lambda_1$ is such that
		$$
		mT = \sum_{j\geq 1} \frac{1}{e^{(\lambda_j+\nu)/T-1}}.
		$$
		If $\nu\geq 0$, then we simply bound 
		$$
		\frac{1}{\lambda_j+\nu} \leq \lambda_j^{-1}.
		$$
		If $\nu \in (-\lambda_1,0)$, then for $\Lambda>0$ large
		$$
		\frac{1}{\lambda_j+\nu} \leq 2\lambda_j^{-1},\quad \forall \lambda_j>\Lambda.
		$$
		In particular, we have
		$$
		mT-\sum_{n=0}^{mT} nc_n \leq 2T \sum_{\lambda_j>\Lambda} \lambda_j^{-1}.
		$$
		By the tail estimate (see \cite[Corollary 2.7]{DinRouTolWan-23}), we know
		$$
		\sum_{\lambda_j>\Lambda} \lambda_j^{-1} \sim \Lambda^{\frac{1}{s}-\frac{1}{2}}
		$$
		hence
		$$
		mT-\sum_{n=0}^{mT} n c_n \leq C T \Lambda^{\frac{1}{s}-\frac{1}{2}}.
		$$
		It follows that
		$$
		\begin{aligned}
		\sum_{n=0}^{mT} n c_n & = \sum_{n=M}^{mT} nc_n + \sum_{n=0}^M nc_n \\
		&\leq mT \sum_{n=M}^{mT} c_n + M \sum_{n=0}^M c_n \\
		&= mT \sum_{n=M}^{mT} c_n + (mT-T\delta) \sum_{n=0}^{M} c_n \\
		&= mT \sum_{n=0}^{mT} c_n - T\delta \sum_{n=0}^M c_n.
		\end{aligned}
		$$
		This implies
		$$
		T\delta \sum_{n=0}^M c_n = mT- \sum_{n=0}^{mT} nc_n \leq C T\Lambda^{\frac{1}{s}-\frac{1}{2}}
		$$
		hence the desired
		$$
		\sum_{n=0}^M c_n \leq C \delta^{-1} \Lambda^{\frac{1}{s}-\frac{1}{2}}.
		$$
	\end{proof}

	For $g>0$ we denote (again with $\#$ standing for $+,-$ or $\varnothing$)
	$$
	H^{\#}_{n,g} = \sum_{j=1}^n h^{\#}_j + \frac{g}{N} \sum_{1\leq j<k \leq n} w^{\#}_{ij},
	$$
	where
	$$
	w^{\#}_{jk} = (P^{\#}_j \otimes P^{\#}_k) w_{jk} (P^{\#}_j \otimes P^{\#}_k)
	$$
	with 
	$$
	w_{jk} = w(\bx_j-\bx_k).
	$$
	The interacting canonical state is defined as
	$$
	\Gamma^{c,\#}_{n,T,g} = \frac{1}{Z^{c,\#}_{n,T,g}} \exp \left(-\frac{1}{T} H^{\#}_{n,g}\right).
	$$
	It is the unique minimizer of the interacting (canonical) energy
	$$
	F^{c,\#}_g (n) = \inf \left\{\mathcal F^{c,\#}_{n,g}[\Gamma] := {\Tr}[H^{\#}_{n,g} \Gamma] + T {\Tr}[\Gamma \log \Gamma] : \Gamma \in \mathcal S((\mathfrak h^{\#})^{\otimes_s n}) \right\}.
	$$
	Moreover, we have
	$$
	F^{c,\#}_{g}(n) = \mathcal F^{c,\#}_{n,g}[\Gamma^{c,\#}_{n,T,g}] = - T \log(Z^{c,\#}_{n,T,g}).
	$$
	In this notation, the targeted quantity in Proposition~\ref{pro:up bound} reads as
	$$
	-\log \left(\frac{Z^c_{mT,T,g}}{Z^c_{mT,T,0}} \right) = \frac{1}{T} \left(F^c_g(mT) - F^c_0(mT) \right).
	$$
	To obtain an upper bound of the right hand side, we pick a large cut-off $\Lambda \geq \lambda_1$ and define a trial state using the unitary map (cf~\eqref{eq:unitary map})
	$$
	\cU : \gF \left(\gh ^- \oplus \gh^+\right) \mapsto \gF \left(\gh ^- \right) \otimes \left(\gh^+\right)
	$$
	in the manner
	\begin{align} \label{Xi-mT}
	\boxed{\Xi_{mT} = \mathcal U^* \left(\bigoplus_{n=0}^{mT} d_n \Gamma^{c,-}_{n,T,g} \otimes \Gamma^{c,+}_{mT-n,T,0} \right) \mathcal U,}
	\end{align}
	where 
	\begin{equation} \label{dn}
	d_n=\begin{cases} c_n D_M^{-1} \mbox{ if } n \geq M \\
	0 \mbox{ if } n < M	     
	    \end{cases}
	\end{equation}
	with $M>0$ to be chosen later and $c_n$  as in \eqref{cn} with $N=mT$. Note that the definition~\eqref{eq:coeffs} of $D_M$ ensures 
	\begin{align} \label{sum-dn}
	\sum_{n=0}^{mT} d_n=1
	\end{align}
	so that $\Xi_{mT}$ is indeed a state on $\mathfrak h^{mT}$.
	
	\begin{lemma}[\textbf{Reduction to a finite dimensional upper bound}]\label{lem:fin dim}\mbox{}\\
	Fixing $m,g >0$ and setting $N= mT$, choosing $\Lambda$ such that $\Lambda \xrightarrow[T\to \infty]{} \infty$ and $M$ as in~\eqref{eq:choice M} with $\delta\gg \Lambda^{1/s-1/2}$, we have (using the above notation)
	\begin{align}\label{eq:up bound 1}
	\limsup_{T\to \infty} -\log \left(\frac{Z^c_{mT,T,g}}{Z^c_{mT,T,0}}\right) &\leq \limsup_{T\to \infty} \frac{1}{T} \left( \mathcal F^{c}_{n,g}\left[\Xi_{mT}\right]- F^c_0(mT) \right) \nonumber \\
	&\leq \limsup_{T\to \infty} \frac{1}{T}\sum_{n=0}^{mT} d_n \left(\mathcal F^{c,-}_{n,g}[\Gamma^{c,-}_{n,T,g}] - \mathcal F^{c,-}_{n,0} [\Gamma^{c,-}_{n,T,0}]\right).
	\end{align}
	
	\end{lemma}
	
	This lemma effectively dispatches the contribution of high one-body energy modes and allows us to focus on estimates in the finite dimensional space $\gh^-$. 
	
	\begin{proof}
	 Using the trial state~\eqref{Xi-mT} we have
	$$
	\begin{aligned}
		F^c_g (mT) - F^c_0(mT) &\leq \mathcal F_{mT,g}^c[\Xi_{mT}] - \mathcal F_{mT,0}^c[\Gamma^c_{mT,T,0}] \\
		&= T \mathcal H(\Xi_{mT}, \Gamma^c_{mT,T,0}) + \frac{g}{mT} {\Tr}[w \, \Xi_{mT}^{(2)}].
	\end{aligned}
	$$
	Since the relative entropy is unaffected by the unitary map $\cU$ and the summands live on orthogonal subspaces, we have
	$$
	\begin{aligned}
		\mathcal H(\Xi_{mT},\Gamma^c_{mT,T,0}) &= \mathcal H \left( \bigoplus_{n=0}^{mT} d_n \Gamma^{c,-}_{n,T, g} \otimes \Gamma^{c,+}_{mT-n,T,0}\, ,\, \bigoplus_{n=0}^{mT} c_n \Gamma^{c,-}_{n,T,0} \otimes \Gamma^{c,+}_{mT-n,T,0}\right) \\
		&= \sum_{n=0}^{mT} \mathcal H \left(d_n \Gamma^{c,-}_{n,T,g} \otimes \Gamma^{c,+}_{mT-n,T,0}, c_n \Gamma^{c,-}_{n,T,0} \otimes \Gamma^{c,+}_{mT-n,T,0}\right) \\
		&= \sum_{n=0}^{mT} \text{Tr}\left[d_n \Gamma^{c,-}_{n,T,g} \otimes \Gamma^{c,+}_{mT-n,T,0} \left(\log(d_n \Gamma^{c,-}_{n,T,g} \otimes \Gamma^{c,+}_{mT-n,T,0})-\log(c_n \Gamma^{c,-}_{n,T,0} \otimes \Gamma^{c,+}_{mT-n,T,0})\right)\right].
	\end{aligned}
	$$
	Using
	$$
	\log(A\otimes B) = (\log A) \otimes \Id + \Id \otimes(\log B),
	$$
	we have
	$$
	\log(d_n \Gamma^{c,-}_{n,T,g} \otimes \Gamma^{c,+}_{mT-n,T,0})-\log(c_n \Gamma^{c,-}_{n,T,0} \otimes \Gamma^{c,+}_{mT-n,T,0})) = \log\left(\frac{d_n}{c_n}\right) +  \left(\log(\Gamma^{c,-}_{n,T,g}) -\log(\Gamma^{c,-}_{n,T,0})\right) \otimes \Id.
	$$
	Thus
	\begin{align*}
		\begin{aligned}
		\text{Tr}\Big[d_n \Gamma^{c,-}_{n,T,g} &\otimes \Gamma^{c,+}_{mT-n,T,0} \left(\log(d_n \Gamma^{c,-}_{n,T,g} \otimes \Gamma^{c,+}_{mT-n,T,0}) -\log(c_n \Gamma^{c,-}_{n,T,0} \otimes \Gamma^{c,+}_{mT-n,T,0})\right)\Big] \\
		&= d_n \log \left(\frac{d_n}{c_n}\right) + d_n \text{Tr}\left[ \Gamma^{c,-}_{n,T,g}\left(\log(\Gamma^{c,-}_{n,T,g}) -\log(\Gamma^{c,-}_{n,T,0})\right) \otimes \Gamma^{c,+}_{mT-n,T,0}\right] \\
		&=  d_n \log \left(\frac{d_n}{c_n}\right) +  d_n \mathcal H\left(\Gamma^{c,-}_{n,T,g},\Gamma^{c,-}_{n,T,0}\right).
		\end{aligned}
	\end{align*}
	By the choice of $d_n$ (see \eqref{dn} and \eqref{sum-dn}), we have
	$$
	\sum_{n=0}^{mT} d_n \log \left(\frac{d_n}{c_n}\right) = \log \left(\frac{1}{D_M}\right)
	$$
	which tends to zero as long as (see \eqref{choi-D})
	\begin{align} \label{cond-del-Lam}
		\delta^{-1} \Lambda^{\frac1s-\frac12} \to 0 \text{ as } \Lambda \to \infty.
	\end{align}
	For the interaction term, we use
	$$
	\Xi_{mT}^{(2)} = \sum_{n=0}^{mT} d_n \left( (\Gamma^{c,-}_{n,T,g})^{(2)} + (\Gamma^{c,+}_{mT-n,T,0})^{(2)} + (\Gamma^{c,-}_{n,T,g})^{(1)} \otimes (\Gamma^{c,+}_{mT-n,T,0})^{(1)} + (\Gamma^{c,+}_{mT-n,T,0})^{(1)} \otimes (\Gamma^{c,-}_{n,T,g})^{(1)} \right)
	$$
	to obtain
	$$
	\begin{aligned}
	{\Tr}[w \, \Xi^{(2)}_{mT}] &= \sum_{n=0}^{mT} d_n \left({\Tr}\left[w (\Gamma^{c,-}_{n,T,g})^{(2)}\right] + {\Tr}\left[w (\Gamma^{c,+}_{mT-n,T,0})^{(2)}\right] \right)  \\
	&\quad + \sum_{n=0}^{mT} d_n \left( {\Tr}\left[w (\Gamma^{c,-}_{n,T,g})^{(1)} \otimes (\Gamma^{c,+}_{mT-n,T,0})^{(1)}\right] + {\Tr}\left[w (\Gamma^{c,+}_{mT-n,T,0})^{(1)} \otimes (\Gamma^{c,-}_{n,T,g})^{(1)}\right]  \right).
	\end{aligned}
	$$
	We now choose $\nu \in \R$ such that~\eqref{eq:GC comp} holds. Using~\eqref{DM-2},~\eqref{sum-dn} and\footnote{We allow for the possibility that $w$ has a measure part in the $p=1$ case, with a slight abuse of notation.} $w\in L^p(\R)$, we have
	$$
	\begin{aligned}
		\sum_{n=0}^{mT} d_n {\Tr}\left[w (\Gamma^{c,+}_{mT-n,T,0})^{(2)}\right] &= \sum_{n=0}^{mT} d_n \iint w(\bx-\by) \left(\Gamma^{c,+}_{mT-n,T,0}\right)^{(2)}(\bx,\by;\bx,\by) d\bx d\by \\
		&\leq  \sum_{n=0}^{mT} d_n \iint |w(\bx-\by)| \left(\Gamma^{c,+}_{mT-n,T,0}\right)^{(2)}(\bx,\by;\bx,\by) d\bx d\by \\
		&\leq C \iint |w(\bx-\by)| \left(P^+ (\Gamma^\nu_T)^{(1)} P^+\right)(\bx,\bx) \left(P^+ (\Gamma^\nu_T)^{(1)}P^+\right)(\by,\by) d\bx d\by \\
		&\leq C \norm{\left(P^+ (\Gamma^\nu_T)^{(1)} P^+\right) (.,.)}_{L^q(\R)} \norm{|w|\ast \left(P^+ (\Gamma^\nu_T)^{(1)}P^+\right)}_{L^p (\R)} \\
		&\leq C \|w\|_{L^p (\R)} \norm{P^+(\Gamma^\nu_T)^{(1)} P^+ (\cdot,\cdot)}_{L^q(\R)} \norm{P^+(\Gamma^\nu_T)^{(1)} P^+ (\cdot,\cdot)}_{L^1(\R)},
	\end{aligned}
	$$
	where $1= \frac{1}{p} + \frac{1}{q}$. Since
	$$
	\begin{aligned}
	(\Gamma^{\nu}_T)^{(1)}(x,x) &= \sum_{j\geq 1} \frac{1}{e^{(\lambda_j+\nu)/T}-1} |u_j(x)|^2 \\
	&=T \sum_{j\geq 1} \frac{1}{T(e^{(\lambda_j+\nu)/T}-1)} |u_j(x)|^2 \\
	&\leq T \sum_{j\geq 1} \frac{1}{\lambda_j+\nu} |u_j(x)|^2,
	\end{aligned}
	$$
	we have for $\Lambda>0$ large,
	$$
	\begin{aligned}
	P^+ (\Gamma^{\nu}_T)^{(1)}P^+ \, (x,x) &\leq T \sum_{\lambda_j>\Lambda} \frac{1}{\lambda_j+\nu} |u_j(x)|^2 \\
	&\leq 2T \sum_{\lambda_j>\Lambda} \lambda_j^{-1}|u_j(x)|^2 \\
	&\leq 2T P^+ h^{-1} P^+ \, (x,x)
	\end{aligned}
	$$
	Since $h^{-1}(.,.) \in L^q(\R)$ for all $1\leq q \leq \infty$ (see \cite[Lemma 3.2]{LewNamRou-17}), the dominated convergence theorem implies
	$$
	\norm{P^+ h^{-1} P^+(.\,,.) }_{L^q(\R)} \to 0 \quad \text{as } \Lambda \to \infty.
	$$
	In particular, we have
	$$
	\sum_{n=0}^{mT} d_n {\Tr}\left[w (\Gamma^{c,+}_{mT-n,T,0})^{(2)}\right] =o_\Lambda(1) T^2.
	$$
	Similarly
	$$
	\begin{aligned}
	\sum_{n=0}^{mT}d_n {\Tr}\Big[w (\Gamma^{c,-}_{n,T,g})^{(1)} &\otimes (\Gamma^{c,+}_{mT-n,T,0})^{(1)}\Big] \\
	&= \sum_{n=0}^{mT} d_n \iint w(\bx -\by) (\Gamma^{c,-}_{n,T,g})^{(1)}(\bx, \bx) (\Gamma^{c,+}_{n,T,0})^{(1)}(\by, \by) d\bx d\by\\
	&\leq \sum_{n=0}^{mT} d_n \iint |w(\bx-\by)| (\Gamma^{c,-}_{n,T,g})^{(1)}(\bx, \bx) (\Gamma^{c,+}_{n,T,0})^{(1)}(\by, \by) d\bx d\by\\
	&\leq C \sum_{n=0}^{mT} d_n \iint |w(\bx -\by)| (\Gamma^{c,-}_{n,T,g})^{(1)}(\bx, \bx) \left(P^+ (\Gamma^\nu_T)^{(1)} P^+\right)(\by,\by) d\bx d\by \\
	&\leq C \left\| \left(P^+ (\Gamma^\nu_T)^{(1)} P^+\right)(\cdot,\cdot)\right\|_{L^q(\R)} \sum_{n=0}^{mT} \Big\||w|\ast (\Gamma^{c,-}_{n,T,g})^{(1)}\Big\|_{L^p(\R)} \\
	&\leq C \|w\|_{L^p(\R)}  \left\| \left(P^+ (\Gamma^\nu_T)^{(1)} P^+\right)(\cdot,\cdot)\right\|_{L^q(\R)} \sum_{n=0}^{mT} d_n \|(\Gamma^{c,-}_{n,T,g})^{(1)}\|_{L^1(\R)} \\
	&\leq C \|w\|_{L^p(\R)}  \left\| \left(P^+ (\Gamma^\nu_T)^{(1)} P^+\right)(\cdot,\cdot)\right\|_{L^q(\R)} \sum_{n=0}^{mT} d_n {\Tr}\left[(\Gamma^{c,-}_{n,T,g})^{(1)}\right] \\
	&\leq C \|w\|_{L^p(\R)}  \left\| \left(P^+ (\Gamma^\nu_T)^{(1)} P^+\right)(\cdot,\cdot)\right\|_{L^q(\R)} \sum_{n=0}^{mT} n d_n.
	\end{aligned}
	$$
%
We may also bound 
$$
	\sum_{n=0}^{mT} n d_n = \sum_{n=M}^{mT} n c_n \leq \sum_{n=0}^{mT} n c_n \leq  mT,
	$$
because 
	\begin{align} \label{mT-expan}
	\begin{aligned}
		mT &={\Tr}\left[(\Gamma^c_{mT,T,0})^{(1)}\right] \\
		&= {\Tr}\left[(\Gamma^{c,-}_{mT,T,0})^{(1)}\right] + {\Tr}\left[P^+  (\Gamma^c_{mT,T,0})^{(1)}\right] \\
		&= \sum_{n=0}^{mT} n c_n + {\Tr}\left[P^+  (\Gamma^c_{mT,T,0})^{(1)}\right]. 
	\end{aligned}
	\end{align}
	This yields
	$$
	\sum_{n=0}^{mT}d_n {\Tr}\Big[w (\Gamma^{c,-}_{n,T,g})^{(1)} \otimes (\Gamma^{c,+}_{mT-n,T,0})^{(1)}\Big] = o_\Lambda(1) T^2.
	$$
	
	Collecting the previous bounds we deduce 
	$$
	\begin{aligned}
		-\log \left(\frac{Z^c_{mT,T,g}}{Z^c_{mT,T,0}}\right) &\leq \frac{1}{T} \sum_{n=0}^{mT} d_n \left( T \mathcal H(\Gamma^{c,-}_{n,T,g}, \Gamma^{c,-}_{n,T,0}) + \frac{g}{m} {\Tr}[w (\Gamma^{c,-}_{n,T,g})^{(2)}]\right) \\
		&\quad + o_\Lambda(1) + \log \left(\frac{1}{D_M}\right).
	\end{aligned}
	$$
	
%
	The main term is rewritten as
	$$
	\frac{1}{T}\sum_{n=0}^{mT} d_n \left(T \mathcal H(\Gamma^{c,-}_{n,T,g}, \Gamma^{c,-}_{n,T,0}) + \frac{g}{m}{\Tr}\left[w (\Gamma^{c,-}_{n,T,g})^{(2)}\right] \right) = \frac{1}{T}\sum_{n=0}^{mT} d_n (\mathcal F^{c,-}_{n,g}[\Gamma^{c,-}_{n,T,g}] - \mathcal F^{c,-}_{n,0} [\Gamma^{c,-}_{n,T,0}]),
	$$
	which concludes the proof. 
	\end{proof}

	\subsection{Semiclassics for the finite dimensional canonical ensemble.}
	
	The problem is now reduced to the finite dimensional setting. We are thus in a position to insert estimates from~\cite[Appendix~B]{Rougerie-LMU} and references therein to obtain 
	
	\begin{lemma}[\textbf{Finite-dimensional semi-classics}]\label{lem:semiclassics}\mbox{}\\
	Let $\mu_{mT,T,g}$ be the probability measure over the unique sphere $S\gh^-$
	be defined by
	\begin{equation}\label{eq:class meas scale}
	d\mu_{mT,T,g}(u) = \frac{1}{z_{mT,T,g}} \exp \left(-\frac{1}{T} \left(mT\langle u, h^- u\rangle + \frac{g(mT-1)}{2}\left\langle u^{\otimes 2}, w^- u^{\otimes 2} \right\rangle \right) \right) du
	\end{equation}
	with $du$ the Lebesgue measure on $S\gh^{-}$.
	
	We have
	\begin{align}\label{eq:semiclassics}
	\begin{aligned}
	\frac{1}{T}\sum_{n=0}^{mT} d_n &\left(\mathcal F^{c,-}_{n,g}[\Gamma^{c,-}_{n,T,g}] - \mathcal F^{c,-}_{n,0} [\Gamma^{c,-}_{n,T,0}]\right) \leq - \log z_m^r + o_{T} (1) + o_{\Lambda} (1) + C\frac{\Lambda}{T} \mathrm{dim} (\gh^-)\\
	&+ \frac{1}{T   }\sum_{n=0} ^{mT} d_n (mT-n) \int_{S\mathfrak h^-} \langle u, h^- u\rangle \left( \mu_{n,T,0}(u) - \mu_{mT,T,g}(u)\right) du, \\
	&+ \frac{1}{T}\sum_{n=0} ^{mT} d_n  \Big(\frac{n(n-1)}{2}-\frac{mT(mT-1)}{2}\Big) \frac{g}{mT} \int_{S\mathfrak h^-} \left\langle u^{\otimes 2}, w^- u^{\otimes 2}\right\rangle \mu_{mT,T,g}(u) du,
	\end{aligned}
	\end{align}
	where $o_T(1),o_\Lambda(1) \to 0$ respectively when $T\to \infty$ and $\Lambda \to \infty.$
	\end{lemma}

	The main term we are after now appears explicitly on the first line of~\eqref{eq:semiclassics}. Note however that we still need to control the contribution from the second and third lines, which involve the measures~\eqref{eq:class meas scale}, i.e. projected and rescaled versions of Definition~\ref{def:int measure}.
	
	\begin{proof}
	We want to bound $\mathcal F^{c,-}_{n,g}[\Gamma^{c,-}_{n,T,g}]$ from above and $\mathcal F^{c,-}_{n,0} [\Gamma^{c,-}_{n,T,0}]$ from below. 
	
	\noindent \textbf{Upper bound on $\mathcal F^{c,-}_{n,g}[\Gamma^{c,-}_{n,T,g}]$.} We are free to insert a convenient trial state in place of the true minimizer $\mathcal F^{c,-}_{n,g}[\Gamma^{c,-}_{n,T,g}]$. We take
	$$
	\Gamma_n =\int_{S \mathfrak h^-} |u^{\otimes n}\rangle \langle u^{\otimes n}| \mu (u)du, \quad \mu = \mu_{mT,T,g}.
	$$
	We have
	$$
	\mathcal F_{n,g}^{c,-}[\Gamma_n] = {\Tr}[H^-_{n,g} \Gamma_n] + T {\Tr}[\Gamma_n \log \Gamma_n].
	$$
	A direct computation gives
	$$
	\begin{aligned}
		{\Tr}[H^-_{n,g} \Gamma_n] &={\Tr}\left[ \Big( \sum_{j=1}^n h^-_j + \frac{g}{mT} \sum_{1\leq j<k\leq n} w^-_{jk}\Big) \int_{S\mathfrak h^-} |u^{\otimes n}\rangle \langle u^{\otimes n}| \mu(u) du\right] \\
		&=n \int_{S\mathfrak h^-} \langle u, h^- u\rangle \mu(u) du + \frac{gn(n-1)}{2mT} \int_{S\mathfrak h^-} \left\langle u^{\otimes 2}, w^- u^{\otimes 2}\right\rangle \mu(u) du.
	\end{aligned}
	$$
	We apply the second Berezin-Lieb inequality (see \cite[Lemma B.4]{Rougerie-LMU}) with $f(x)=x \log x$ to get
	$$
	\begin{aligned}
		{\Tr}[\Gamma_n \log \Gamma_n] &\leq \dim\left(\left(\mathfrak h^-\right)^{\otimes_s n}\right) \int_{S\mathfrak h^-} \frac{\mu(u)}{\dim\left(\left(\mathfrak h^-\right)^{\otimes_s n}\right)} \log \left(\frac{\mu(u)}{\dim\left(\left(\mathfrak h^-\right)^{\otimes_s n}\right)} \right) du \\
		&=-\log \left(\dim\left(\left(\mathfrak h^-\right)^{\otimes_s n}\right)\right) + \int_{S\mathfrak h^-} \mu(u) \log\left(\mu(u)\right) du.
	\end{aligned}
	$$
	It follows that
	$$
	\begin{aligned}
		\mathcal F^{c,-}_{n,g}[\Gamma^{c,-}_{n,T,g}] &= F^{c,-}_g(n) \\
		&\leq \mathcal F^{c,-}_{n,g}[\Gamma_n] \\
		&\leq n \int_{S\mathfrak h^-} \langle u, h^- u\rangle \mu(u) du + \frac{n(n-1)}{2N} g \int_{S\mathfrak h^-} \left\langle u^{\otimes 2}, w^- u^{\otimes 2}\right\rangle \mu (u) du \\
		&\quad + T \int_{S\mathfrak h^-} \mu (u) \log\left(\mu (u)\right) du - T \log \left(\dim\left(\left(\mathfrak h^-\right)^{\otimes_s n}\right)\right).
	\end{aligned}
	$$

	\noindent\textbf{Lower bound on $\mathcal F^{c,-}_{n,0} [\Gamma^{c,-}_{n,T,0}]$.} Let us  denote
	$$
	\tilde{\mu}_n(u)= \dim\left(\left(\mathfrak h^-\right)^{\otimes_s n}\right) \left\langle u^{\otimes n}, \Gamma^{c,-}_{n,T,0} u^{\otimes n}\right\rangle
	$$
	the lower symbol of $\Gamma^{c,-}_{n,T,0}$. 
	
	Using the quantitative quantum de Finetti theorem (see \cite[Theorem 4.1]{Rougerie-LMU}), we have
	 \begin{equation}\label{eq:deF1}
	  {\Tr}\Big| (\Gamma^{c,-}_{n,T,0})^{(1)} - \int_{S\mathfrak h^-} |u\rangle \langle u | \tilde{\mu}_n(u) du \Big| \leq C\frac{\dim(\mathfrak h^-)}{n}
	\end{equation}

	which implies
	\begin{equation}\label{eq:deF2}
	\begin{aligned}
		{\Tr}[H^-_{n,0} \Gamma^{c,-}_{n,T,0}] &= n {\Tr}[h^- (\Gamma^{c,-}_{n,T,0})^{(1)}] \\
		&= n {\Tr}\Big[h^- \int_{S\mathfrak h^-} |u\rangle \langle u| \tilde{\mu}_n(u) du\Big] + n {\Tr}\Big[h^- \Big((\Gamma^{c,-}_{n,T,0})^{(1)} - \int_{S\mathfrak h^-} |u\rangle \langle u| \tilde{\mu}_n(u) du\Big)\Big] \\
		&\geq n\int_{S\mathfrak h^-} \langle u, h^- u\rangle \tilde{\mu}_n(u) du - n \|h^-\|_{\text{Op}} \Big\|(\Gamma^{c,-}_{n,T,0})^{(1)} - \int_{S\mathfrak h^-} |u\rangle \langle u| \tilde{\mu}_n(u) du\Big\|_{\mathfrak S^1} \\
		&\geq n\int_{S\mathfrak h^-} \langle u, h^- u\rangle \tilde{\mu}_n(u) du - C \Lambda \dim(\mathfrak h^-),
	\end{aligned}
	\end{equation}
	where $\|h^-\|_{\text{Op}} \leq \Lambda$.
	
	For the entropy term, we use the first Berezin-Lieb inequality (see \cite[Lemma B.3]{Rougerie-LMU}) to get
	$$
	\begin{aligned}
	{\Tr}[\Gamma^{c,-}_{n,T,0} \log \Gamma^{c,-}_{n,T,0}] &\geq \dim\left(\left(\mathfrak h^-\right)^{\otimes_s n}\right) \int_{S\mathfrak h^-} \frac{\tilde{\mu}_n(u)}{\dim\left(\left(\mathfrak h^-\right)^{\otimes_s n}\right)} \log \Big(\frac{\tilde{\mu}_n(u)}{\dim\left(\left(\mathfrak h^-\right)^{\otimes_s n}\right)} \Big)du \\
	&= - \log\left(\dim\left(\left(\mathfrak h^-\right)^{\otimes_s n}\right)\right)+ \int_{S\mathfrak h^-} \tilde{\mu}_n(u) \log (\tilde{\mu}_n(u)) du.
	\end{aligned}
	$$
	In particular, we have
	$$
	\begin{aligned}
	\mathcal F^{c,-}_{n,0}&[\Gamma^{c,-}_{n,T,0}] \\
	&\geq n \int_{S\mathfrak h^-} \langle u, h^- u\rangle \tilde{\mu}_n(u) du + T \int_{S\mathfrak h^-} \tilde{\mu}_n(u) \log (\tilde{\mu}_n(u)) du - T \log\left(\dim\left(\left(\mathfrak h^-\right)^{\otimes_s n}\right)\right)- C \Lambda \dim(\mathfrak h^-) \\
	& \geq n \int_{S\mathfrak h^-} \langle u, h^- u \rangle \mu_{n,T,0}(u) du + T \int_{S\mathfrak h^-} \mu_{n,T,0}(u) \log(\mu_{n,T,0}(u)) du \\
	&\quad -  T \log\left(\dim\left(\left(\mathfrak h^-\right)^{\otimes_s n}\right)\right)- C \Lambda \dim(\mathfrak h^-),
	\end{aligned}
	$$
	where 
	\begin{equation}\label{eq:mu nT}
	d\mu_{n,T,0}(u) = \frac{1}{z_{n,T,0}} \exp \left(-\frac{n}{T} \langle u, h^- u\rangle \right) du
	\end{equation}
	is the unique minimizer of 
	$$
	\mathcal F_{n,T,0}(\mu) = n \int_{S\mathfrak h^-} \langle u, h^- u\rangle \mu(u) du + T \int_{S\mathfrak h^-} \mu(u) \log(\mu(u)) du
	$$
	over all $\mu \in \mathcal P(S\gh^-)$, where $\mathcal P(S\gh^-)$ is the set of all probability measures on $S\gh^-$. It follows that
	$$
	\begin{aligned}
		\frac{1}{T}\sum_{n=0}^{mT} &d_n \left(\mathcal F^{c,-}_{n,g}[\Gamma^{c,-}_{n,T,g}] - \mathcal F^{c,-}_{n,0} [\Gamma^{c,-}_{n,T,0}] \right) \\
		&\leq \frac{1}{T}\sum_{n=0}^{mT} d_n \Big(n \int_{S\mathfrak h^-}  \langle u, h^- u\rangle \mu(u) du + \frac{n(n-1)}{2mT} g \int_{S\mathfrak h^-} \left\langle u^{\otimes 2}, w^- u^{\otimes 2}\right\rangle \mu(u) du + T \int_{S\mathfrak h^-} \mu(u) \log (\mu(u)) du \\
		&\quad \quad \quad -n \int_{S\mathfrak h^-} \langle u, h^- u\rangle \mu_{n,T,0}(u) du  - T \int_{S\mathfrak h^-} \mu_{n,T,0}(u) \log(\mu_{n,T,0}(u)) du + C\Lambda\dim(\mathfrak h^-) \Big)
	\end{aligned}
	$$
	where we took $\mu = \mu_{mT,T,g}$ defined as in~\eqref{eq:class meas scale}, leading to
	$$
	\begin{aligned}
		\frac{1}{T}\sum_{n=0}^{mT} &d_n \left(\mathcal F^{c,-}_{n,g}[\Gamma^{c,-}_{n,T,g}] - \mathcal F^{c,-}_{n,0} [\Gamma^{c,-}_{n,T,0}] \right) \\
		&\leq \frac{1}{T} \sum_{n=0}^{mT} d_n \Bigg(n \int_{S\mathfrak h^-} \langle u, h^- u\rangle \mu_{mT,T,g} (u) du + \frac{n(n-1)g}{2mT} \int_{S\mathfrak h^-} \left\langle u^{\otimes 2}, w^- u^{\otimes 2}\right\rangle \mu_{mT,T,g}(u) du   \\
		&\quad \quad \quad + T \int_{S \mathfrak h^-} \mu_{mT,T,g}(u) \log (\mu_{mT,T,g}(u)) du   \\
		&\quad \quad \quad - n \int_{S\mathfrak h^-} \langle u, h^- u\rangle \mu_{n,T,0}(u) du- T \int_{S\mathfrak h^-} \mu_{n,T,0}(u) \log(\mu_{n,T,0}(u)) du \Bigg) + \frac{C}{T} \Lambda \dim(\mathfrak h^-).
	\end{aligned}
	$$
	We write the quantity between parenthesis on the right-hand side as 
	$$
	(\cdots) = (\text{I}) + (\text{II}) + (\text{III}),
	$$
	where
	$$
	\begin{aligned}
	(\text{I}) &= mT\int_{S\mathfrak h^-} \langle u, h^- u\rangle \mu_{mT,T,g} (u) du + \frac{mT(mT-1)}{2mT}g \int_{S\mathfrak h^-} \left\langle u^{\otimes 2}, w^- u^{\otimes 2}\right\rangle \mu_{mT,T,g}(u) du \\
	&\quad + T \int_{S \mathfrak h^-} \mu_{mT,T,g}(u) \log (\mu_{mT,T,g}(u)) du \\
	&\quad  - mT \int_{S\mathfrak h^-} \langle u, h^- u\rangle \mu_{n,T,0}(u) du - T \int_{S\mathfrak h^-} \mu_{n,T,0}(u) \log(\mu_{n,T,0}(u)) du, \\
	(\text{II}) &= (mT-n) \int_{S\mathfrak h^-} \langle u, h^- u\rangle (\mu_{n,T,0}(u)-\mu_{mT,T,g}(u)) du, \\
	(\text{III}) &= \Big(\frac{n(n-1)}{2}-\frac{mT(mT-1)}{2}\Big) \frac{g}{mT} \int_{S\mathfrak h^-} \left\langle u^{\otimes 2}, w^- u^{\otimes 2}\right\rangle \mu_{mT,T,g}(u) du.
	\end{aligned}
	$$
	The terms (II) and (III) appear on the right side of~\eqref{eq:semiclassics}, so that there remains to control (I).
%
%
	Observe that
	$$
	\begin{aligned}
		mT \int_{S\mathfrak h^-} \langle u, h^- u\rangle \mu_{n,T,0}(u) du &+ T \int_{S\mathfrak h^-} \mu_{n,T,0}(u) \log(\mu_{n,T,0}(u)) du \\
		&= mT \left(\int_{S\mathfrak h^-} \langle u, h^- u\rangle \mu_{n,T,0}(u) du + \frac{1}{m} \int_{S\mathfrak h^-} \mu_{n,T,0}(u) \log(\mu_{n,T,0}(u)) du\right) \\
		&\geq mT \inf\left\{ \int_{S\mathfrak h^-} \langle u, h^- u\rangle \mu(u) du + \frac{1}{m} \int_{S\mathfrak h^-} \mu(u)\log(\mu(u)) du : \mu \in \mathcal P(S\mathfrak h^-)\right\} \\
		&\geq mT \Big(\int_{S\mathfrak h^-} \langle u, h^- u\rangle \rho_m(u) du + \frac{1}{m} \int_{S\mathfrak h^-} \rho_m(u)\log(\rho_m(u)) du\Big),
	\end{aligned}
	$$
	where
	$$
	d\rho_m(u) = \frac{1}{\tilde{z}_m} \exp (-m\langle u, h^- u\rangle) du.
	$$
	It follows that
	$$
	\begin{aligned}
	(\text{I}) &\leq  mT\int_{S\mathfrak h^-} \langle u, h^- u\rangle \mu_{mT,T,g} (u) du + \frac{g(mT-1)}{2} \int_{S\mathfrak h^-} \left\langle u^{\otimes 2}, w^- u^{\otimes 2}\right\rangle \mu_{mT,T,g}(u) du \\
	&\quad + T \int_{S \mathfrak h^-} \mu_{mT,T,g}(u) \log (\mu_{mT,T,g}(u)) du \\
	&\quad  - mT \int_{S\mathfrak h^-} \langle u, h^- u\rangle \rho_m(u) du - T \int_{S\mathfrak h^-} \rho_m(u) \log(\rho_m(u)) du.
	\end{aligned}
	$$
	A direct computation gives
	$$
	\begin{aligned}
	(\text{I}) &\leq - T \log (z_{mT,T,g}) + T \log (\tilde{z}_m) \\
	&= - T \log \Big(\frac{z_{mT,T,g}}{\tilde{z}_m}\Big) \\
	&= - T \log \Big( \int_{S \mathfrak h^-} \exp \Big(-\frac{g(mT-1)}{2T} \left\langle u^{\otimes 2}, w^- u^{\otimes 2}\right\rangle\Big) \rho_m(u) du\Big)
	\end{aligned}
	$$
	with the partition function $z_{mT,T,g}$ as in~\eqref{eq:class meas scale}.
	
	Thus, using \eqref{sum-dn} and the notation of Lemma~\ref{lem:meas sphere} we have
	\begin{align*}
	\frac{1}{T} \sum_{n=0}^{mT} d_n \times (\text{I}) &\leq - \log \Big( \int_{S \mathfrak h^-} \exp \Big(-\frac{g(mT-1)}{2T} \left\langle u^{\otimes 2}, w^- u^{\otimes 2}\right\rangle\Big) \rho_m(u) du\Big)\\
	&\leq - \log \Big( \int_{S \mathfrak h^-} \exp \Big(-\frac{mg}{2} \left\langle u^{\otimes 2}, w^- u^{\otimes 2}\right\rangle\Big) \rho_m(u) du\Big) + o_{T} (1)\\
	&= - \log \Big( \int \exp \Big(-\frac{g}{2m} \left\langle u^{\otimes 2}, w^- u^{\otimes 2}\right\rangle\Big) d\sigma_{m,\Lambda} (u)\Big) + o_{T} (1)\\
	&= - \log\left( z^r_m\right)+ o_{T} (1) + o_{\Lambda} (1).
	\end{align*}
 Here we have changed variables $u\to \sqrt{m} u$ to go to the third line, and used Lemma~\ref{lem:meas sphere} to go the last one.

	\end{proof}

	\subsection{Controling particle number fluctuations and conclusion}\label{sec:control fluct}
	
	There remains to estimate the spurious terms from the right-hand side of~\eqref{eq:semiclassics} to complete the 
	
	\begin{proof}[Proof of Proposition~\ref{pro:up bound}]
	We will again use repeatedly the estimate from~\cite[Lemma D1]{DinRou-23}
	$$ \dim \gh ^- = \dim E_\Lambda \leq C \Lambda ^{1/2 + 1/s}.$$
	We combine Lemma~\ref{lem:fin dim} with Lemma~\ref{lem:semiclassics} and choose 
    \begin{equation}\label{eq:choice delta}
     \delta = \Lambda^{1/s-1/2} \Lambda ^{\alpha}
    \end{equation}
	with $\alpha > 0$ suitably small, in particular so that 
	$$ \delta \xrightarrow[\Lambda \to \infty]{} 0.$$
	This way the assumptions of Lemma~\ref{lem:fin dim} are fulfilled. We choose $\Lambda (T)\to \infty$ in dependence with respect to $T$ such that 
	$$ \Lambda ^{3/2 + 1/s} T^{-1} \to 0.$$
	This way the three error terms of the first line of~\eqref{eq:semiclassics} are negligible in the limit $T\to\infty$.
	
	We quickly dispatch the third line of~\eqref{eq:semiclassics}. This term is bounded by
	$$
	\frac{g}{mT^2} \sum_{n=0}^{mT} d_n \left(\frac{mT(mT-1)}{2} - \frac{n(n-1)}{2}\right)  \int_{S\mathfrak h^-} \left|\left\langle u^{\otimes 2}, w^- u^{\otimes 2}\right\rangle\right| \mu_{mT,T,g}(u) du.
	$$
	Bounding the interaction term as in Appendix~\ref{sec:sub-inte-meas}, using~\eqref{eq:choice M},~\eqref{sum-dn} and $d_n = D^{-1} c_n \mathds{1}_{\{n\geq M\}}$, the above quantity is bounded by
	$$
	\frac{1}{T^2} \sum_{n=0}^{mT} d_n (mT-n) \frac{mT+n-1}{2} \leq \delta m
	$$
	which tends to $0$ as soon as $\delta \to 0$.
	
	Our main task is to control the second line of~\eqref{eq:semiclassics}, using Section~\ref{sec:dep mass}. Since the sum is limited to $n\sim mT$ by our choice of $d_n$ we may apply Proposition~\ref{pro:dep mass} with $m_1 = n/T$, $m_2 = m$ (and the immaterial change $g\to g(1-O(N^{-1}))$) to obtain 
	$$\int_{S\mathfrak h^-} \langle u, h u\rangle \left(\mu_{n,T,0}(u)-\mu_{mT,T,g}(u)  \right) du \leq C \Lambda ^{\frac{1}{4} + \frac{1}{2s}} + C \left|\frac{n}{T} - m\right| \Lambda ^{\frac{1}{2} + \frac{1}{s}}$$
	provided the left-hand side is non-negative. Thus, since $\sum_{n=0}^{mT} d_n = 1$ and with the choice~\eqref{eq:choice delta}
	\begin{align} \label{est-II}
	\frac{1}{T} \sum_{n=0}^{mT}d_n (mT-n)\int_{S\mathfrak h^-} \langle u, h u\rangle \left(\mu_{n,T,0}(u)-\mu_{mT,T,g}(u)\right) du \leq C \delta \Lambda ^{\frac{1}{4} + \frac{1}{2s}} + C \delta^2 \Lambda ^{\frac{1}{2} + \frac{1}{s}} \leq C \Lambda ^{\frac{3}{2s} - \frac{1}{4} + \alpha}.
	\end{align}
	Since we assume $s>6$, we may always take $\alpha >0$ small enough so that 
	$$ \frac{3}{2s} - \frac{1}{4} + \alpha < 0$$
	and thus~\eqref{eq:choice delta} tends to $0$ when $\Lambda \to \infty$, as needed, which concludes the proof.
	\end{proof}

	\section{Free energy lower bound}\label{sec:low}
	
	In this section, we turn to the energy lower bound, completing the proof of Theorem~\ref{thm:main}:
	
	\begin{proposition}[\textbf{Free energy lower bound}]\label{pro:low bound}\mbox{}\\
	Let the interacting canonical Gibbs state $\Gamma_{mT,T,g}^c$ be defined as in~\eqref{eq:C ens} with the particle number set as $N = mT$, with $m >0 ,g\geq 0$ fixed. 
	
	There exists a probability measure $\nu$ over $L^2 (\R)$ such that, modulo subsequence, for any $k \geq 1$,
	 \begin{equation}\label{eq:main DM alt}
	 \frac{k!}{T^k} \left(\Gamma_{mT,T,g}^c \right)^{(k)} \xrightharpoonup[T\to \infty]{}^\star \int |u^{\otimes k} \rangle \langle u^{\otimes k} | d\nu (u)
	 \end{equation}
    weakly-$\star$ in the trace-class $\gS ^1 (\gh^k)$, where the reduced density matrices are defined as in~\eqref{eq:def red mat}. 
    
    The relative quantum free-energy satisfies 
    \begin{multline}\label{eq:ener low}
     \liminf_{T\to \infty} \left(- \log \left(Z^c_{mT,T,g}\right) +  \log \left(Z^c_{mT,T,0}\right)\right) \\ 
     \geq \mathcal H_{\cl}(\nu, \mu_{0,m}) + \frac{g}{2m} \int \left( \iint_{X \times X} |u(\bx)|^2  w(\bx-\by) |u(\by)|^2 d\bx d\by \right) d\nu(u) \geq  -\log z^r_m
    \end{multline}
    where $z^r_m$ is the classical relative partition function~\eqref{eq:class part}.
	\end{proposition}

	\begin{proof}[Proof of Proposition~\ref{pro:low bound} and completion of the proof of Theorem~\ref{thm:main}]
	The existence of the measure $\nu$ in~\eqref{eq:main DM alt} is a general fact, following from the weak quantum de Finetti theorem, e.g. in the version of~\cite[Section~2]{LewNamRou-14}.

	Combining~\eqref{eq:ener low} and Proposition~\ref{pro:up bound} gives the energy convergence~\eqref{eq:main ener}. Moreover it shows that any limit measure $\nu$ in~\eqref{eq:main DM alt} must minimize the classical relative free energy, and hence be equal to $\mu_{g,m}$. By uniqueness of the limit this leads to~\eqref{eq:main DM} in the sense of weak-$\star$ convergence. Since $\mu_{g,m}$ lives on the sphere $\int |u|^2 =m$, the trace of the left-hand side of~\eqref{eq:main DM alt} converges to the trace of the right-hand side. Both sides being positive trace-class operators, the usual criterion~\cite[Addendum-H]{Simon-79} upgrades the convergence to strong in the trace-class. 
	   
	Everything thus now relies on, given~\eqref{eq:main DM alt}, passing to the lim inf to prove the first inequality in~\eqref{eq:ener low}. We start from the fact that $\Gamma_{mT,T,g}^c$ minimizes the relative free energy functional
	$$
	\mathcal H\left(\Gamma, \Gamma^c_{mT,T,0}\right) + \frac{g}{mT^2} {\Tr}\left[w(\bx-\by)\Gamma^{(2)}\right] = {\Tr}\left[\Gamma (\log \Gamma-\log \Gamma^c_{mT,T,0})\right]+ \frac{g}{mT^2} {\Tr}\left[w(\bx-\by)\Gamma^{(2)}\right]
	$$
	The Berezin-Lieb inequality of \cite[Theorem 7.1]{LewNamRou-15} deals conveniently with the relative entropy term. Indeed,~\eqref{eq:main DM alt} says that $\nu$ is de Finetti measure of the sequence $\Gamma^c_{mT,T, g}$ at scale $T^{-1}$. Since we know from Theorem~\ref{thm:free case} that $\mu_{0,m}$ is de Finetti measure of the sequence $\Gamma^c_{mT,T,0}$ at scale $T^{-1}$, we obtain 
	\begin{equation}\label{eq:liminf ent}
	\liminf_{T\to \infty} {\Tr}\left[\Gamma^c_{mT,T, g} (\log \Gamma^c_{mT,T,g}-\log \Gamma^c_{mT,T,0})\right] \geq \int \frac{d\nu}{d\mu_{0,m}} \log \frac{d\nu}{d\mu_{0,m}} d\mu_{0,m}.
	\end{equation}
	There remains to pass to the lim inf in the interaction energy term. For a purely repulsive/defocusing interaction, this would be done exactly as in~\cite{LewNamRou-15}. The sequel is thus primarily aimed at allowing for an attractive/focusing component in $w$.
	
	\medskip
	
	\noindent\textbf{First} we claim that, with 
	\begin{equation}\label{eq:pert alpha}
	0 <\alpha < \frac{1}{2} - \frac{1}{s}
	\end{equation}
 	we have 
 	\begin{equation}\label{eq:low a priori}
 	\Tr \left[ \frac{h^\alpha}{T} \left(\Gamma^c_{mT,T, g} \right)^{(1)}\right] \leq C \cH \left(\Gamma^c_{mT,T, g}, \Gamma^c_{mT,T,0} \right) + C_\alpha. 
 	\end{equation}
    To this end we define the auxiliary $mT$-particles Gibbs state 
    $$ \Gamma_{mT,0}^\alpha := \frac{1}{Z_{mT,0}^\alpha} \exp\left(-\frac{1}{T} \sum_{j=1} ^{mT}\left( h_{\bx_j} - c h_{\bx_j} ^\alpha \right) \right)$$
	where $c$ is small enough for $h - c h ^\alpha \geq c' h$ to be a non-negative operator. Following the first few lines of the proof of~\cite[Theorem~6.1]{LewNamRou-20} we find 
	\begin{align*}
	&\cH \left(\Gamma^c_{mT,T, g}, \Gamma^c_{mT,T,0}\right) - \frac{c}{T} \Tr \left[ h^\alpha  \left(\Gamma^c_{mT,T, g} \right)^{(1)} \right]\\ 
	&= \frac{1}{T} \Tr \left[ \left(h - ch^\alpha  \right) \left(\Gamma^c_{mT,T, g} \right)^{(1)} \right] + \Tr \left[\Gamma^c_{mT,T, g} \log \left( \Gamma^c_{mT,T, g} \right)\right] - \frac{1}{T} \Tr \left[ h \left(\Gamma^c_{mT,T, 0} \right)^{(1)} \right] - \Tr \left[\Gamma^c_{mT,T, 0} \log \left( \Gamma^c_{mT,T, 0} \right)\right]\\
	&\geq \frac{1}{T} \Tr \left[ \left(h - ch^\alpha  \right) \left(\Gamma^\alpha _{mT,0} \right)^{(1)} \right] + \Tr \left[\Gamma^\alpha _{mT,0}  \log \left( \Gamma^\alpha _{mT,0}  \right)\right] - \frac{1}{T} \Tr \left[ h \left(\Gamma^c_{mT,T, 0} \right)^{(1)} \right] - \Tr \left[\Gamma^c_{mT,T, 0} \log \left( \Gamma^c_{mT,T, 0} \right)\right]\\
	&\geq -\frac{c}{T} \Tr \left[ h^\alpha  \left(\Gamma^\alpha _{mT,0} \right)^{(1)} \right]
	\end{align*}
    using successively the variational principles defining $\Gamma^\alpha_{mT,0}$ and $\Gamma^c_{mT,T, 0}$. Arguing as in Appendix~\ref{sec:DMs} (changing $h$ to $h-ch^\alpha$) we find that, for a well-chosen fixed $\nu \in \R$ 
	 $$\left(\Gamma^\alpha _{mT,0} \right)^{(1)} \leq C \frac{1}{\exp\left(T^{-1}\left(h - ch^\alpha + \nu \right)\right)-1} \leq C \frac{T}{h - c h^\alpha + \nu}.$$ 
	Inserting in the above we deduce  
	$$ \Tr \left[ \frac{h^\alpha}{T} \left(\Gamma^c_{mT,T, g} \right)^{(1)}\right] \leq \cH \left(\Gamma^c_{mT,T, g}, \Gamma^c_{mT,T,0}\right) + C\Tr h^{\alpha - 1}$$ 
	and since $h^{\alpha - 1}$ is trace-class for $\alpha$ as in~\eqref{eq:pert alpha} (see e.g.~\cite[Appendix~A]{DinRou-23}), we have proven~\eqref{eq:low a priori}.

	\medskip
	
	\noindent\textbf{Next,} we deduce that 
	\begin{equation}\label{eq:rel ent bound}
	 \cH \left(\Gamma^c_{mT,T, g}, \Gamma^c_{mT,T,0} \right) \leq C 
	\end{equation}
    independently of $T$. To this end we first recall that, under Assumption~\ref{asum:int} on the negative part of $w$ we have that, for any $t>0$, the two-particle operator
    \begin{equation}\label{eq:H2t}
     H_2 ^t := \frac{h_{\bx}^\alpha}{4}+ \frac{h_{\by}^\alpha}{4} + t w (\bx - \by) \geq - C
    \end{equation}
	is bounded below as a self-adjoint operator acting on $L^2(\R^2)$ for 
	$$ \alpha \geq \frac{1}{2p}.$$
	Indeed, following~\cite[Remark~3.1]{Rougerie-LMU} or adapting the proof of~\cite[Inequality~(3.3)]{NamRouSei-15} we see that~\eqref{eq:H2t} holds as soon as $H^\alpha$ continuously embeds into $L^{2q}$, with 
	$$ \frac{1}{p} + \frac{1}{q} = 1.$$
	This requires 
	$$ 2 \frac{p}{p-1} \leq \frac{2}{1- 2 \alpha} \mbox{ i.e. } \alpha \geq \frac{1}{2p}.$$
	But, with $p>\frac{s}{s-2}$, the right-hand side is smaller than $1/2 - 1/s$ so that we may pick  
	$$\alpha = \frac{1}{2} - \frac{1}{s} - \eta$$
	with $\eta>0$ sufficiently small. Then~\eqref{eq:pert alpha} is satisfied and we may apply the previous step to obtain~\eqref{eq:low a priori}. Combining this with the free-energy upper bound of Proposition~\ref{pro:up bound} we are led to 
	\begin{align*}
	C &\geq \mathcal H\left(\Gamma^c_{mT,T, g}, \Gamma^c_{mT,T,0}\right) + \frac{g}{mT^2} {\Tr}\left[w(\bx-\by)\left(\Gamma^c_{mT,T, g} \right)^{(2)}\right]\\
	&\geq  \frac{1}{2} \cH\left(\Gamma^c_{mT,T, g}, \Gamma^c_{mT,T,0}\right) + \frac{1}{2} \Tr \left[ \frac{h^\alpha}{T} \left(\Gamma^c_{mT,T, g} \right)^{(1)}\right] + \frac{g}{mT^2} {\Tr}\left[w(\bx-\by)\left(\Gamma^c_{mT,T, g} \right)^{(2)}\right] - C 
	 \\&\geq \frac{1}{2} \cH\left(\Gamma^c_{mT,T, g}, \Gamma^c_{mT,T,0}\right) + \frac{1}{T^2} \Tr \left[ H_2^t \left(\Gamma^c_{mT,T, g} \right)^{(2)}\right] - C
	 \end{align*}
	 for a suitably chosen fixed $t>0$. Inserting~\eqref{eq:H2t} we deduce~\eqref{eq:rel ent bound}.
	 
	 \medskip
	
	\noindent\textbf{Finally,} we let $\varepsilon >0$ and use~\eqref{eq:low a priori} again to get 
	 \begin{align*}
	 \begin{aligned}
	 \cH\left(\Gamma^c_{mT,T, g}, \Gamma^c_{mT,T,0}\right) &+ \frac{g}{mT^2} {\Tr}\left[w(\bx-\by)\left(\Gamma^c_{mT,T, g} \right)^{(2)}\right] 
	 \\
	 &\geq (1-\varepsilon)\cH\left(\Gamma^c_{mT,T, g}, \Gamma^c_{mT,T,0}\right) + \frac{1}{T^2} {\Tr}\left[ G_2 ^\eps \left(\Gamma^c_{mT,T, g} \right)^{(2)}\right] - C\varepsilon
	 \end{aligned}
	 \end{align*}
	 with, similarly as above 
	 $$ G_2 ^\varepsilon := \varepsilon \left(\frac{h_{\bx}^\alpha}{2}+ \frac{h_{\by}^\alpha}{2}\right)+ \frac{g}{m} w(\bx-\by) \geq - f(\varepsilon) $$
	 for some finite $f(\varepsilon)>0$. Using~\eqref{eq:rel ent bound} and the fact that the trace of $\left(\Gamma^c_{mT,T, g} \right)^{(2)}$ is by definition 
	 $$\frac{mT(mT-1)}{2} \propto \frac{m^2 T^2}{2}$$ we infer 
	 \begin{align*}
	 \begin{aligned}
	 &\cH\left(\Gamma^c_{mT,T, g}, \Gamma^c_{mT,T,0}\right) + \frac{mg}{T^2} {\Tr}\left[w(\bx-\by)\left(\Gamma^c_{mT,T, g} \right)^{(2)}\right] 
	 \\
	 &\quad \geq \cH\left(\Gamma^c_{mT,T, g}, \Gamma^c_{mT,T,0}\right) + \frac{1}{T^2} {\Tr}\left[ \left(G_2 ^\eps + f(\eps) \right) ^{1/2}\left(\Gamma^c_{mT,T, g} \right)^{(2)}\left(G_2 ^\eps + f(\eps) \right) ^{1/2}\right] - C\varepsilon - \frac{m^2}{2} f(\varepsilon).
	 \end{aligned}
	 \end{align*}
	 Using Proposition~\ref{pro:up bound} it follows that the non-negative operator 
	 $$
	 T^{-2}\left(G_2 ^\eps + f(\eps) \right) ^{1/2}\left(\Gamma^c_{mT,T, g} \right)^{(2)}\left(G_2 ^\eps + f(\eps) \right) ^{1/2}
	 $$
	 is bounded in trace-class. Modulo subsequence it converges weakly-$\star$ when $T\to \infty$ to some operator that we may identify with the help of~\eqref{eq:main DM alt}:
	 $$
	 \frac{1}{T^2}\left(G_2 ^\eps + f(\eps) \right) ^{1/2}\left(\Gamma^c_{mT,T, g} \right)^{(2)}\left(G_2 ^\eps + f(\eps) \right) ^{1/2}\wto^\star \frac{1}{2}\left(G_2 ^\eps + f(\eps) \right) ^{1/2} \int |u^{\otimes 2} \rangle \langle u^{\otimes 2} | d\nu (u)\left(G_2 ^\eps + f(\eps) \right) ^{1/2}.
	 $$
	Using Fatou's lemma for trace-class operators (i.e. the fact that the trace-class norm is weakly-$\star$ lower semi-continuous) and combining with~\eqref{eq:liminf ent} we get 
	\begin{align*}
	\begin{aligned}
	 \liminf_{T\to \infty} \Big( \cH\left(\Gamma^c_{mT,T, g}, \Gamma^c_{mT,T,0}\right) &+ \frac{g}{mT^2} {\Tr}\left[w(\bx-\by)\left(\Gamma^c_{mT,T, g} \right)^{(2)}\right] \Big)
	 \\ 
	 &\geq \cH_{\rm cl} \left( \nu, \mu_{0,m}\right) + \frac{1}{2}\int \left\langle u^{\otimes 2} | G_2^\eps + f(\eps) | u^{\otimes 2} \right\rangle d\nu (u) - C\varepsilon - \frac{m^2}{2} f(\varepsilon).	 
	 \end{aligned}
	 \end{align*}
	 Since $G_2^\eps + f(\epsilon)\geq 0$ it follows that $\cH_{\rm cl} \left( \nu, \mu_{0,m}\right)$ is finite and hence $\nu$ is absolutely continuous with respect to $\mu_{0,m}$. In particular $\int |u|^2 = m $ for $\nu$-almost every $u$ so that  
	\begin{align*}
	\begin{aligned}
	 \liminf_{T\to \infty} \Big( \cH\left(\Gamma^c_{mT,T, g}, \Gamma^c_{mT,T,0}\right) &+ \frac{g}{mT^2} {\Tr}\left[w(\bx-\by)\left(\Gamma^c_{mT,T, g} \right)^{(2)}\right] \Big)
	 \\ 
	 &\geq \cH_{\rm cl} \left( \nu, \mu_{0,m}\right) + \frac{1}{2}\int \left\langle u^{\otimes 2} | G_2^\eps  u^{\otimes 2} \right\rangle d\nu (u) - C\varepsilon.
	 \end{aligned}
	 \end{align*}
	Since by definition 
	$$G_2 ^\varepsilon \geq \frac{g}{m} w(\bx-\by)$$
	we may finally pass to the limit $\varepsilon\to 0$ to conclude that
	\begin{align*}
	\begin{aligned}
	 \liminf_{T\to \infty} \Big( \cH\left(\Gamma^c_{mT,T, g}, \Gamma^c_{mT,T,0}\right) &+ \frac{g}{mT^2} {\Tr}\left[w(\bx-\by)\left(\Gamma^c_{mT,T, g} \right)^{(2)}\right] \Big)
	 \\ 
	 &\geq \cH_{\rm cl} \left( \nu, \mu_{0,m}\right) + \frac{g}{2m}\int \left\langle u^{\otimes 2} | w(\bx-\by)  u^{\otimes 2} \right\rangle d\nu (u)
	 \end{aligned} 	 
	 \end{align*}
	 which is the desired lower bound on the relative free energy functional, concluding the proof.
	\end{proof}

	\newpage
	
	\appendix
	
	\section{Interacting Gibbs measure} \label{sec:sub-inte-meas}
	We recap the construction of the interacting Gibbs measure, vindicating that Definition~\ref{def:int measure} makes sense. More details may be found e.g.~\cite{DinRou-23,DinRouTolWan-23}.
	
	We will assume that 
	$$
	w = w_1+w_2, \quad w_1 \in \mathcal M, \quad w_2 \in L^p \text{ with } 1< p <\infty,
	$$
	where $\mathcal M$ is the set of bounded (Radom) measures. The measure part $w_1$ can include a delta function. 
	
	Under this assumption, we will show that the interacting Gibbs measure $\mu_m$ is well-defined as a probability measure. We will show that the partition function is finite by considering separately two cases: defocusing $w\geq 0$ and focusing $w \leq 0$.
	
	\subsection{Reduction to local interactions}
	
	\noindent \textbf{Defocusing part of the interaction.}
	
	For a non-negative interaction potential
	$$
	F_{\NL}(u):=\iint_{X\times X} |u(\bx)|^2 |u(\by)|^2 w(\bx-\by) d\bx d\by \geq 0
	$$
	and since $\mu_{0,m}$ is a probability measure, we have $z^r_m \leq 1$. To see that $z^r_m>0$, we use the Jensen inequality 
	$$
	z^r_m \geq \exp \left( -\frac{1}{2} \int F_{\NL}(u) d\mu_{0,m}(u)\right). 
	$$
	Thus the problem is reduced to showing that
	$$
	\int F_{\NL}(u) d\mu_{0,m}(u) <\infty. 
	$$
	By Young's inequality, we have
	\begin{align}\label{youn-ineq}
	\begin{aligned}
		F_{\NL}(u) &= \int_X (|u|^2 \ast w) |u(\bx)|^2 d\bx \\
		&= \int_X (|u|^2 \ast w_1) |u(\bx)|^2 d\bx + \int_X (|u|^2 \ast w_2) |u(\bx)|^2 d\bx \\
		&\leq \||u|^2\|_{L^2} \||u|^2 \ast w_1\|_{L^2} + \||u|^2\|_{L^r} \||u|^2 \ast w_2\|_{L^{r'}} \\
		&\leq \|w_1\|_{L^1} \||u|^2\|^2_{L^2} + \|w_2\|_{L^p} \||u|^2\|^2_{L^r} \\
		&= \|w_1\|_{L^1} \|u\|^4_{L^4} + \|w_2\|_{L^p} \|u\|^4_{L^{2r}},
	\end{aligned}
	\end{align}
	where $\frac{2}{r} +\frac{1}{p}=2$. Since $1< p<\infty$, we have $r \in (1,2)$. By interpolation, we have
	$$
	\|u\|_{L^{2r}} \leq \|u\|_{L^4}^{\theta} \|u\|_{L^2}^{1-\theta}
	$$
	with 
	$$
	\theta = 2\left(\frac{1}{r}-\frac{1}{2}\right) \in (0,1).
	$$
	By the Young inequality
	$$
	a^\theta b^{1-\theta} \leq a + C(\theta) b,
	$$
	we have
	$$
	\|u\|_{L^{2r}} \leq \|u\|_{L^4} + C(\theta) \|u\|_{L^2}.
	$$
	Since $\|u\|^2_{L^2}=m$ on the support of $\mu_{0,m}$, it is enough to prove that
	\begin{align}\label{boun-defo}
	\int \|u\|^4_{L^4} d\mu_{0,m}(u) <\infty.
	\end{align}
	
	\medskip
	
	\noindent\textbf{Focusing part of the interaction}
	
	For a non-positive interaction, we have $z^r_m \geq 1$. It remains to prove $z^r_m < \infty$. 
	As in the defocusing case, we estimate
	$$
	\begin{aligned}
	-F_{\NL}(u) &\leq \int_X |(|u|^2 \ast w)(\bx)| |u(\bx)|^2 d\bx \\
	&\leq \int_X |(|u|^2\ast w_1)(\bx)| |u(\bx)|^2 d\bx + \int_X |(|u|^2 \ast w_2)(\bx)| |u(\bx)|^2 d\bx \\
	&\leq \|w_1\|_{L^1} \|u\|^4_{L^4} + \|w_2\|_{L^p} \|u\|^4_{L^{2r}}.
	\end{aligned}
	$$
	On the support of $\mu_{0,m}$, we have
	$$
	\|u\|^4_{L^{2r}} \leq \|u\|^4_{L^4} + C(\theta) \|u\|^4_{L^2} \leq C(m,\theta) + \|u\|^4_{L^4}.
	$$
	The problem is now reduced to proving
	\begin{align} \label{boun-focu}
	\int e^{\|u\|^4_{L^4}} d\mu_{0,m}(u) <\infty.
	\end{align}
	
	\subsection{Exponential integrability of the $L^4$ norm}

	As per the above, to make sense of the interacting measure it suffices to prove~\eqref{boun-focu}, which implies~\eqref{boun-defo}. The proof is done in two steps.
	
	\medskip
	
	\noindent {\bf Step 1. Decay of $L^4$-norm in high frequency.} We first prove that there exist $C,c>0$ such that for all $0\leq \rho <\frac{s-1}{2s}$, all $\Lambda$ sufficiently large and all $R>0$,
	\begin{align} \label{L4-decay}
	\mu_{0,m}(\|P_\Lambda^\perp u\|_{L^4}>R) \leq C e^{-c \Lambda^\rho R^2}.
	\end{align}
	Since $\mu_{\epsilon,m} \to \mu_{0,m}$ as $\epsilon \to 0$ (cf Section~\ref{sec:def cond meas}), it suffices to prove that for $\epsilon>0$ small enough,
	\begin{align} \label{L4-decay-eps}
	\mu_{\epsilon,m}(\|P_\Lambda^\perp u\|_{L^4}>R) \leq C e^{-c \Lambda^\rho R^2}.
	\end{align}
	We have
	$$
	\begin{aligned}
	\mu_{\epsilon,m} (\|P_\Lambda^\perp u\|_{L^4} >R) &\leq e^{-tR^2} \int e^{t \|P_\Lambda^\perp u\|_{L^4}^2} d\mu_{\epsilon,m}(u) \\
	&= e^{-tR^2} \sum_{k\geq 0} \frac{t^k}{k!} \int \|P_\Lambda^\perp u\|^{2k}_{L^4} d\mu_{\epsilon,m}(u).
	\end{aligned}
	$$
	We write
	$$
	\begin{aligned}
	\int &\|P_\Lambda^\perp u\|^{2k}_{L^4} d\mu_{\epsilon,m}(u) \\
		 &= \frac{1}{z^r_{\epsilon,m}} \int \|P_\Lambda^\perp u\|^{2k}_{L^4} \exp \left(-\frac{1}{\epsilon}(\langle u, u\rangle -m)^2\right) d\mu_0(u) \\
	     &= \frac{1}{z^r_{\epsilon,m}} \int \|P_\Lambda^\perp u\|^{2k}_{L^4} \left( \exp \left(-\frac{1}{\epsilon}(\|P_\Lambda u\|_{L^2}^2 + \|P^\perp_\Lambda u\|_{L^2}^2-m)^2\right) d\mu_{0,\Lambda}(u)\right) d\mu_{0, \Lambda}^\perp (u).
	\end{aligned}
	$$
	Let $g_\Lambda$ be the density function of $\|P_\Lambda u\|_{L^2}^2$ with respect to $\mu_{0,\Lambda}$. As in the proof of Proposition~\ref{pro:fixed mass meas}, we have
	$$
	\begin{aligned}
	\frac{1}{\sqrt{\epsilon}} \int \exp \left( -\frac{1}{\epsilon}(\|P_\Lambda u\|_{L^2}^2+ \|P^\perp_\Lambda u\|_{L^2}^2-m)^2\right) d\mu_{0,\Lambda}(u) &= \frac{1}{\sqrt{\epsilon}} \int_0^{+\infty} \exp \left(-\frac{1}{\epsilon}(y-m+\nu_\Lambda)^2\right) g_\Lambda(y) dy \\
	&\leq \sqrt{\pi} \|g_\Lambda\|_{L^\infty} \\
	&\leq C,
	\end{aligned}
	$$
	where $\nu_\Lambda = \|P^\perp_\Lambda u\|_{L^2}^2$, where $g_\Lambda$ is uniformly bounded as long as $\Lambda$ is sufficiently large. Moreover, for $\epsilon>0$ small enough,
	$$
	\frac{1}{\sqrt{\epsilon}} z^r_{\epsilon,m} \geq \frac{1}{2} f_0(m) \sqrt{\pi}.
	$$
	It follows that
	$$
	\int \|P_\Lambda^\perp u\|^{2k}_{L^4} d\mu_{\epsilon,m}(u) \leq \frac{C}{f_0(m)} \int \|P^\perp_\Lambda u\|^{2k}_{L^4} d\mu_{0,\Lambda}^\perp (u). 
	$$
	Arguing as in \cite[Lemma 3.4]{DinRou-23}, we obtain
	$$
	\int \|P_\Lambda^\perp u\|^{2k}_{L^4} d\mu_{0,\Lambda}^\perp (u) \leq 2! k! B^k_\Lambda,
	$$
	where
	$$
	B_\Lambda := \left(\int_{X} \left((P_\Lambda^\perp h)^{-1}(x,x) \right)^2 dx \right)^{1/2} \leq \Lambda^{-\rho} \left(\int_{X} \left(h^{\rho-1}(x,x) \right)^2 dx \right)^{1/2} \leq C \Lambda^{-\rho},
	$$
	where $0\leq \rho<\frac{s-1}{2s}$. This shows that for $\epsilon>0$ small,
	$$
	\mu_{\epsilon,m}(\|P_\Lambda^\perp u\|_{L^4}>R)) \leq 2! e^{-tR^2} \sum_{k\geq 0} (Ct \Lambda^{-\rho})^k 
	$$
	Taking $t=\nu \Lambda^\rho$ with $\nu>0$ small, we obtain \eqref{L4-decay-eps}, hence \eqref{L4-decay}.
	
	\medskip
	
	\noindent {\bf Step 2. The exponential bound.} We have
	$$
	\begin{aligned}
		\int e^{\|u\|^4_{L^4}} d\mu_{0,m}(u) &= \int_0^\infty \mu_{0,m}(e^{\|u\|^4_{L^4}}>\lambda) d\lambda \\
		&= \int_0^\infty \mu_{0,m}\left(\|u\|^4_{L^4} > \lambda \right) e^\lambda d\lambda \\
		&= C(\lambda_0) + \int_{\lambda_0}^\infty \mu_{0,m}\left(\|u\|_{L^4}>\lambda^{1/4}\right) e^\lambda d\lambda
	\end{aligned}
	$$
	for some $\lambda_0>0$ to be fixed later. Using a Sobolev embedding and the fact that $\|u\|_{L^2}^2=m$ on the support of $\mu_{0,m}$, we have
	$$
	\begin{aligned}
		\|P_\Lambda u\|_{L^4} &\leq C \|\langle \nabla \rangle^{\frac{1}{4}} P_\Lambda u\|_{L^2} \leq C \|h^{1/8} P_\Lambda u\|_{L^2} \\
		&\leq C \Lambda^{1/8}\|P_\Lambda u\|_{L^2} \leq C \Lambda^{1/8} \|u\|_{L^2} \\
		&\leq C \Lambda^{1/8} m^{1/2}.
	\end{aligned}
	$$
	For $\lambda>\lambda_0$, we take 
	$$
	\Lambda_0 = \left(\frac{1}{2C}\right)^8 \frac{\lambda^2}{m^4},
	$$
	we obtain
	$$
	\|P_{\Lambda_0} u\|_{L^4} \leq \frac{1}{2} \lambda^{1/4}.
	$$
	Using the triangle inequality, we have for $u \in \text{supp}(\mu_{0,m})$ satisfying $\|u\|_{L^4} >\lambda^{1/4}$, we have
	$$
	\|P_{\Lambda_0}^\perp u\|_{L^4} \geq \|u\|_{L^4}-\|P_{\Lambda_0} u\|_{L^4} > \frac{1}{2} \lambda^{1/4}.
	$$
	Thus
	$$
	\mu_{0,m} \left(\|u\|_{L^4}>\lambda\right) \leq \mu_{0,m} \left(\|P_{\Lambda_0}^\perp u\|_{L^4}>\frac{1}{2}\lambda^{1/4}\right).
	$$
	Increasing $\lambda_0$ a bit if necessary, we can apply Step 1 to get
	$$
	\mu_{0,m} \left(\|u\|_{L^4}>\lambda^{1/4}\right) \leq C e^{-c \Lambda_0^\rho \lambda^{1/2}} = C e^{-c \lambda^{2\rho+1/2}}
	$$
	for all $0\leq \rho<\frac{s-1}{2s}$. Since $s>2$, we can choose $\rho$ so that
	$$
	\frac{1}{4}<\rho <\frac{s-1}{2s}.
	$$
	With such a choice, we have $2\rho +1/2>1$. In particular, 
	$$
	\mu_{0,m} \left(\|u\|_{L^4}>\lambda^{1/4}\right) e^\lambda \in L^1((\lambda_0,+\infty))
	$$
	which conclude the proof of~\eqref{boun-focu}.
	
	\section{Bounds on canonical density matrices}\label{sec:DMs}

	We will need estimates on canonical density matrices to replace the extensive use of Wick's theorem in the grand-canonical case. We recall here bounds from~\cite[Appendix~A]{DeuSeiYng-18} that relate canonical quantities in terms of their grand-canonical counterparts.
	
	Let $h$ be a non-negative operator with compact, trace-class resolvent and $\nu>0$ be a fixed real number. Consider the free grand canonical Gibbs state
	$$
	\Gamma^{\nu}_T = \frac{1}{Z^\nu_T} \exp \left(-\frac{1}{T} d\Gamma(h+\nu)\right).
	$$ 
	The quantum Wick theorem (cf e.g.~\cite[Section~2, Appendix~A]{LewNamRou-15}) gives
	$$
	(\Gamma^{\nu}_T)^{(k)} = P_\sym^k \left(\frac{1}{e^{(h+\nu)/T}-1}\right)^{\otimes k} P^k_\sym.
	$$
	Here $P_\sym^k$ is the bosonic symmetrizer 
	$$
	P_\sym^k = \frac{1}{k!}\sum_{\sigma \in \Sigma^k} U_{\sigma}
	$$
	where the sum is over the permutation group and $U_\sigma$ permutes particle labels according to $\sigma$.
	
	In particular, we have
	\begin{align} \label{Gam-nu-T}
		(\Gamma^{\nu}_T)^{(k)} \leq C_k \left(\frac{T}{h+\nu}\right)^{\otimes k}
	\end{align}
	as operators. In fact, it suffices to prove for $k=1$, i.e.,
	$$
	\langle u_j \vert (\Gamma^{\nu}_T)^{(1)} \vert u_j\rangle \leq \frac{T}{\lambda_j + \nu},
	$$
	where $(\lambda_j,u_j)_{j\geq 1}$ are eigenvalues and (normalized) eigenfunctions of $h$. We have
	$$
	\langle u_j \vert (\Gamma^{\nu}_T)^{(1)} \vert u_j\rangle = \left\langle u_j \Big\vert \frac{1}{e^{(h+\nu)/T}-1} \Big\vert u_j \right\rangle = \frac{1}{e^{(\lambda_j+\nu)/T}-1} \leq \frac{T}{\lambda_j +\nu}.
	$$
	We need similar bounds for canonical density matrices:
		
	\begin{lemma}[\textbf{Canonical density matrices} {\cite[Proposition A.2]{DeuSeiYng-18}}]\mbox{}\\
		Set $N= m T$ with $m>0$ fixed, and let $\nu= \nu (N,T) >-\lambda_1$ be such that
		\begin{equation}\label{eq:GC comp}
		N = {\Tr}\left[\mathcal N \, \Gamma^\nu_T\right]= \sum_{j\geq 1} \frac{1}{e^{(\lambda_j+\nu)/T}-1}.
		\end{equation}
		There exists constants $c_1,c_2\in \R$ depending only on $m$ such that 
		$$c_1 \leq \nu \leq c_2$$
		independently of $T$.
		
		For $n\leq N$ let $\Gamma^c_{n,T,0}$ the free $n$-particles canonical Gibbs case be defined as in Section~\ref{sec:trial state}. With $\nu= \nu (N,T)$ as above we have that
		\begin{align} \label{DM-1}
			(\Gamma^c_{n,T,0})^{(1)}(x,x) \leq \frac{40}{1.8} (\Gamma^{\nu}_T)^{(1)}(x,x), \quad n=0,...,N
		\end{align}
		and
		\begin{align} \label{DM-2}
			(\Gamma^c_{n,T,0})^{(2)}(x,y;x,y) \leq 4\left(\frac{40}{1.8}\right)^2 (\Gamma^{\nu}_T)^{(1)}(x,x) (\Gamma^{\nu}_T)^{(1)}(y,y), \quad n=0,...,N
		\end{align}
		almost everywhere. In addition, we have 
		\begin{align} \label{DM-1-ope}
		(\Gamma^c_{n,T,0})^{(1)} \leq \frac{40}{1.8} (\Gamma^\nu_T)^{(1)}, \quad n =0,...,N
		\end{align}
		as operators. 
	\end{lemma}
	
	\begin{proof}
	The map 
	$$ \nu \mapsto N (\nu) = {\Tr}\left[\mathcal N \Gamma^\nu_T\right]$$
	is easily seen to be decreasing on $(-\lambda_1,\infty)$. It tends to $+\infty$ when $\nu \to -\lambda_1$ and to $0$ when $\nu \to +\infty$. Hence for all $N$ there exists a unique $\nu \in (-\lambda_1,\infty)$ such that~\eqref{eq:GC comp} holds. Since 
	$$ 
	T^{-1} {\Tr}\left[\mathcal N \Gamma^\nu_T\right]\underset{T \to \infty}{\longrightarrow}\Tr \frac{1}{h+\nu}
	$$
	one can see that, given $m>0$, there exists $\nu_1 < \nu_2$ such that for $T$ large enough $T^{-1}N (\nu) \leq m/2$ if $\nu < \nu_1$ and $T^{-1}N (\nu) \geq 2m$ if $\nu > \nu_2$. Hence indeed if $N= mT$, the unique $\nu$ satisfying~\eqref{eq:GC comp} is bounded as a function of $m$ alone.   
	
		The inequalities \eqref{DM-1} and \eqref{DM-2} are proven in \cite[Proposition A.2]{DeuSeiYng-18} when $n=N$, based on input from~\cite{Suto-04b}. The cases $n\leq N$ follow because the maps $n \mapsto \langle f(\mathcal N_j) \rangle_{\Gamma^c_{n,T,0}}$ are non-decreasing for all non-negative non-decreasing functions $f$, where $\mathcal N_j =\ada_j a_j = \ada(u_j) a(u_j)$ (\cite[Proposition A.1]{DeuSeiYng-18}). Here $(u_j)_{j\geq 1}$ are normalized eigenfunctions of $h$ which forms an orthonormal basis of $\mathfrak h$. The proof of~\eqref{DM-1-ope} follows from similar arguments. It suffices to prove for any $j\geq 1$,
		$$
		\langle u_j | (\Gamma^c_{n,T,0})^{(1)} | u_j \rangle \leq \frac{40}{1.8} \langle u_j | (\Gamma_T^{\nu})^{(1)} | u_j \rangle, \quad n=0,...,N.
		$$
		Observe that
		$$
		\langle u_j |(\Gamma^c_{n,T,0})^{(1)} |u_j \rangle = {\Tr}\left[\ada_j a_j \Gamma^c_{n,T,0}\right] = \langle \mathcal N_j \rangle_{\Gamma^c_{n,T,0}}.
		$$
		Since $n \mapsto \langle \mathcal N_j \rangle_{\Gamma^c_{n,T,0}}$ is non-decreasing, we get
		$$
		\langle u_j |(\Gamma^c_{n,T,0})^{(1)} |u_j \rangle \leq \langle \mathcal N_j \rangle_{\Gamma^c_{N,T,0}} \leq \langle u_j |(\Gamma^c_{N,T,0})^{(1)} |u_j \rangle, \quad n =0,..., N.
		$$
		Therefore, we only consider the case $n=N$. We write
		$$
		\Gamma^{\nu}_T = \bigoplus_{n=0}^\infty \frac{1}{Z^\nu_T} \exp \left(-\nu \frac{n}{T} \right) \exp \left(-\frac{1}{T} \sum_{j=1}^n h_j\right) = \bigoplus_{n=0}^\infty \frac{Z^c_{n,T,0}}{Z^\nu_T} \exp \left(-\nu \frac{n}{T}\right) \Gamma^c_{n,T,0} = \bigoplus_{n=0}^\infty c_n \Gamma^c_{n,T,0},
		$$
		where 
		$$
		c_n = \frac{Z^c_{n,T,0}}{Z^\nu_T} \exp \left(-\nu \frac{n}{T}\right).
		$$
		We observe that $c_n \geq 0$ for all $n\geq 0$ and $\sum_{n=0}^\infty c_n =1$. It follows that
		$$
		\begin{aligned}
		\langle u_j | (\Gamma^{\nu}_T)^{(1)} | u_j \rangle = {\Tr}\left[\ada_j a_j \Gamma^{\nu}_T\right] = \langle \mathcal N_j \rangle_{\Gamma^{\nu}_T} = \sum_{n=0}^\infty c_n \langle \mathcal N_j \rangle_{\Gamma^c_{n,T,0}}, 
		\end{aligned}
		$$
		Since $\langle \mathcal N_j \rangle_{\Gamma^c_{n,T}} \geq \langle \mathcal N_j \rangle_{\Gamma^c_{N,T,0}}$ for all $n\geq N$, we get
		$$
		\langle u_j | (\Gamma^{\nu}_T)^{(1)} | u_j\rangle \geq \left(\sum_{n=N}^\infty c_n\right) \langle \mathcal N_j \rangle_{\Gamma^c_{N,T,0}} = \left(\sum_{n=N}^\infty c_n\right) \langle u_j | (\Gamma^c_{N,T,0})^{(1)} | u_j \rangle. 
		$$
		The desired inequality follows from the fact (see \cite[between (A.7) and (A.8)]{DeuSeiYng-18}) that
		$$
		\sum_{n=N}^\infty c_n \geq \frac{1.8}{40}.
		$$
	\end{proof}

%

\end{document}